\documentclass[preprint,3p]{elsarticle}

\usepackage{amssymb}
\usepackage{amsmath}
\usepackage{amsthm}

\usepackage{amsfonts, hyperref, bbm, cancel, wrapfig, enumitem, tikz, graphicx}

\newtheorem{theorem}{Theorem}[section]
\newtheorem{corollary}[theorem]{Corollary}
\newtheorem{lemma}[theorem]{Lemma}
\newtheorem{propn}[theorem]{Proposition}
\theoremstyle{definition}

\theoremstyle{remark}
\newtheorem{remark}{Remark}[section]
\theoremstyle{remark}
\newtheorem{example}{Example}[section]

\newcommand{\CC}{{\mathbb C}}
\newcommand{\RR}{{\mathbb R}}

\newcommand{\ZZ}{{\mathbb Z}}
\newcommand{\NN}{{\mathbb N}}

\newcommand{\cl}{\operatorname{cl}}

\newcommand{\cQ}{{\cal Q}}
\newcommand{\cO}{{\cal O}}

\newcommand{\cX}{{\cal X}}

\newcommand{\cF}{{\cal F}}

\newcommand{\cN}{{\cal N}}
\newcommand{\cP}{{\cal P}}
\newcommand{\cC}{{\cal C}}
\newcommand{\cH}{{\cal H}}
\newcommand{\cR}{{\cal R}}

\newcommand{\cS}{{\cal S}}

\numberwithin{equation}{section}

\newenvironment{pf*}[1]{\proof[#1]}{\endproof}
\usepackage{euscript}

\journal{Journal of Differential Equations}

\begin{document}

\begin{frontmatter}



\title{Stability and Bifurcations of Planar Switched Linear and Homogeneous Systems}


\author[UWaterloo]{Ivan O. Shevchenko\corref{cor1}}
\ead{i3shevch@uwaterloo.ca}
\cormark[1]

\author[UWaterloo]{Xinzhi Liu}
\ead{xzliu@uwaterloo.ca}

\affiliation[UWaterloo]{organization={Department of Applied Mathematics, University of Waterloo},
addressline={200 University Avenue West},
city={Waterloo},
postcode={N2L 3G1},
state={Ontario},
country={Canada}}

\cortext[cor1]{Corresponding author}

\begin{abstract}
We prove new necessary and sufficient conditions for uniform asymptotic stability under arbitrary switching of two-dimensional switched homogeneous systems with a finite number  of subsystems using a worst-case switching analysis. The novelty of our approach is in its explicit nature, which allows us to then study in detail the codimension-one bifurcations of stability of the origin in switched linear systems and further conclude new local and global stability results for certain classes of nonlinear switched systems. In particular, we formulate an analogue of Lyapunov's indirect method for $\mathcal{C}^{1}$ switched nonlinear systems and derive a new method for determining the existence of a bounded basin of attraction for a class of switched nonlinear systems.
\end{abstract}



\end{frontmatter}



\section{Introduction}\label{section:intro}

Dynamical models of real-word systems are commonly assumed to either be purely continuous (as in ordinary or partial differential equations) or purely discrete (as in difference equations). However, many systems in the real world are most accurately described by hybrid models which incorporate interactions between continuous and discrete dynamics. Legged robotic motion (\cite{raibert1986legged,wieber2016modeling}), traffic control systems (\cite{horowitz2002control,livadas2002high,yuan2009traffic}), complex manufacturing systems (\cite{pepyne2002optimal,song2002integration}), and automatic transmissions in automobiles (\cite{brockett1993hybrid}) are all examples which give rise to hybrid behaviour (see also \cite{antsaklis2003hybrid,branicky1995studies,goebel2009hybrid}).

Such hybrid behaviour can be modelled through the framework of switched dynamical systems (\cite{goebel2009hybrid,liberzon2003switching}). Switched systems are given by a collection of continuous dynamical systems represented by autonomous ordinary differential equations (ODEs) and a time-dependent ``switching signal'' controlling which ODE governs the dynamics at a given moment in time. More formally, a family of functions $f_{p}:\RR^{n} \to \RR^{n}$, where $p \in \cP$ for some finite indexing set $\cP$, gives rise to a collection of subsystems
\begin{equation*}
    \dot{x} = f_{p}(x), \hspace{10px} p \in \cP.
\end{equation*}
A switching signal is defined to be a piecewise-constant function $\sigma: \RR^{\geq0} \to \cP$ which is everywhere continuous from the right and with left limits existing everywhere (\cite{liberzon2003switching}) such that $\textrm{card}(\sigma^{-1}(p) \cap J)$ is finite for every bounded interval $J$ and $p \in \cP$. This assumption is required to avoid ``Zeno-like'' behaviour where we have an infinite number of switches in a finite amount of time (\cite{liberzon2003switching,van2007introduction}). Such behaviour is usually considered an undesirable modelling artifact of hybrid dynamical systems (\cite{goebel2009hybrid}), so this assumption does not greatly restrict the class of ``useful'' systems in our analysis (\cite{liberzon2003switching}). For a given indexing set $\cP$, we denote the set of all such switching signals by $\Sigma_{\cP}$. Finally, a switched dynamical system is specified by the equation
\begin{equation}\label{eq:defn_of_switched_system_in_intro}
    \dot{x}(t) = f_{\sigma(t)}(x(t)).
\end{equation}
We are interested in the simplified cases where we either set $f_{p}(x) = A_{p}x$ for some matrices $A_{p} \in \RR^{2 \times 2}$, or we enforce that the functions $f_{p}$, $p \in \cP$, in \eqref{eq:defn_of_switched_system_in_intro} are $\cC^{1}$ homogeneous functions of fixed degree $\beta>1$ mapping $\RR^{2}$ into itself.

Much work has been done to study stability properties of switched linear and homogeneous systems, which generally falls into the two approaches: providing sufficient conditions for stability which are easier to verify, or providing necessary and sufficient conditions for stability which are harder (and often intractable) to verify. For switched linear systems, examples of the first approach include conditions related to the triangular structure of the matrices (\cite{cohen1997exponential,mori1997solution}) and Lie algebraic conditions (\cite{agrachev2012robust,agrachev2001lie,liberzon1999stability,sharon2007third}) for switched linear systems, and dwell-time conditions (\cite{hespanha1999stability, morse2002supervisory, aleksandrov2012asymptotic}), and methods involving linear matrix inequalities to determine existence of a common quadratic Lyapunov function for the switched system (\cite{shorten2003result, shorten2002necessary, zhang2007stability}) for both switched linear and homogeneous systems. An overview of stability methods of this type for switched linear systems may be found in the excellent survey \cite{lin2009stability} and in parts of the textbooks \cite{liberzon2003switching}, \cite{sun2011stability}, and more recently \cite{chitour2025dynamics}.

Explicit necessary and sufficient conditions for uniform asymptotic stability of two-dimensional switched linear and homogeneous systems have appeared in the literature since at least the 1980s with the work of Filippov (\cite{filippov1980stability}) in the context of differential inclusions. However, the proof of sufficiency and necessity in this paper does not provide insight into how the trajectories of the switched system change qualitatively between stable and unstable regimes, which limits of possibility of bifurcation analysis and further investigation of stability behaviour for nonlinear switched systems.

Other conditions proposed by Pyatnitsky and Rapoport in \cite{pyatnitskiy1996criteria} are more transparent and applicable to bifurcation analysis, however the conditions are checked by solving a nonlinear equation in multiple unknown variables which is difficult to do in practice. A number of results in the literature (for example \cite{balde2009note,boscain2002stability,perez2013switched}) focus on the case of switching between two second-order linear or homogeneous subsystems, but their methods do not readily generalize to switched systems consisting of $n>2$ subsystems.

The papers \cite{holcman2002stability} and \cite{yang2012sufficient} provide necessary and sufficient conditions for uniform asymptotic stability of a certain class of planar switched homogeneous and linear systems consisting of $n$ subsystems. The first paper uses generalized first integrals (\cite{margaliot2003necessary}), and the second paper considers ``worst-case'' switching behaviour on regions of state space (which is similar to our approach; see \cite{margaliot2006stability} for a general survey of stability analysis of switched systems using worst-case switching laws) and works in polar coordinates. However, in both papers there is an additional assumption (Assumption 1 in \cite{holcman2002stability} and Assumption 3.3 in \cite{yang2012sufficient}) which restricts the class of switched systems they consider. Although this assumption may be treated in a separate case, it provides an artificial roadblock to a unified study of bifurcations of these systems.

The necessary and sufficient conditions proved in this paper are valid for switched homogeneous and linear systems under the mild assumptions that the functions $f_{p}$ defining the subsystems are $\cC^{1}$ in the homogeneous case, and that the switching signal is not ``Zeno-like'' as defined above. Our conditions also admit a pleasing graph-theoretic interpretation which makes them easier to interpret and visualize. Under the non-Zeno assumption, all of the existing works above outlining necessary and sufficient conditions for uniform asymptotic stability are special cases of Theorem \ref{thm:iff_conditions_for_stability} and Corollary \ref{cor:iff_conditions_for_stability_homog} in this paper. In addition to this stability result for switched homogeneous systems, the explicit approach we use in the proof of the result (that is, working closer to the actual trajectories of the switched system) allows us to draw further qualitative conclusions about how stability is lost in the system, and also allows us to extend our methods to certain classes of nonlinear systems (in contrast to the similar approaches used in \cite{holcman2002stability} and \cite{yang2012sufficient}). Theorem \ref{thm:hartman_grobman} provides an easily-checked test for the definitive presence or lack of local uniform asymptotic stability for $\cC^{1}$ switched nonlinear systems which admit a common equilibrium point at the origin, in the spirit of Lyapunov's indirect method in systems theory (\cite{khalil1992nonlinear}), and proves that there must always exist a periodic solution of the nonlinear switched system arbitrarily close to the origin if certain conditions do not hold. Theorem \ref{thm:basin_of_attraction_existence} yields similar results about the existence of a bounded basin of attraction for a more specific class of switched nonlinear systems.

Although bifurcations have been studied extensively in piecewise-smooth dynamical systems (\cite{di2008bifurcations,filippov2013differential}), little progress has been made in this direction for switched systems with arbitrary switching. We hope to remedy this situation with Theorem \ref{thm:bifurcation_analysis}, which gives a classification of codimension-one bifurcations of stability of the origin in switched linear systems. Again, our explicit framework is crucial for us to be able to make the qualitative assertions in Theorem \ref{thm:bifurcation_analysis} of exactly how uniform asymptotic stability is lost in the switched system; this level of analysis could not be achieved with the more high-level approaches of \cite{holcman2002stability} and \cite{yang2012sufficient}. Additionally, Theorem \ref{thm:basin_of_attraction_existence} yields some partial results which may help to classify bifurcations of the existence of a bounded basin of attraction for certain classes of switched \textit{nonlinear} systems.

The outline for the rest of the paper is as follows. We go over the preliminary assumptions and derive some groundwork results in Section \ref{section:preliminaries}, namely, we classify zero sets of quadratic forms $x^{\top}Ax$ for Hurwitz matrices $A$ and then show that we may without loss of generality impose some additional simplifying assumptions to the matrices defining our switched system, which greatly simplifies the number of cases we need to consider in the stability analysis later in the paper. Section \ref{subsection:stability_under_assumptions} is devoted to proving a preliminary version of the main theorem concerning necessary and sufficient stability conditions for switched linear systems under some additional technical assumptions. Briefly, this is done by first defining a fixed number of rays $S_{[i]}$, $i \in \{1, \ldots, n\}$, and showing that in the ``worst-case'' behaviour the trajectories of the switched system may only switch subsystems upon crossing one of these rays. Then, having essentially reduced the problem to checking a finite number of possible ``worst'' switching combinations, Theorem \ref{thm:iff_conditions_for_stability_assuming_all_edges_exist} concludes with the main necessary and sufficient stability conditions of this subsection. In Section \ref{subsection:stability_without_assumptions} we extend this result to the case of general switched linear systems by considering an augmented version of the original switched system to which Theorem \ref{thm:iff_conditions_for_stability_assuming_all_edges_exist} applies, and then showing that stability properties of the two switched systems are equivalent; the culmination of the work in this subsection is Theorem \ref{thm:iff_conditions_for_stability}. Section \ref{section:consequences} discusses some of the useful applications of Theorem \ref{thm:iff_conditions_for_stability}: we classify the codimension-one bifurcations of uniform asymptotic stability of the origin of switched linear systems and formulate theorems which provide easily-checked conditions for stability of the origin and the existence of a bounded basin of attraction for certain classes of switched nonlinear systems. Examples are provided to illustrate the usefulness and elegance of these results. We conclude this section by extending our stability results from switched linear systems to switched homogeneous systems. Section \ref{section:summary} summarizes. For ease of readability, the proof of Theorem \ref{thm:iff_conditions_for_stability_assuming_all_edges_exist} and relevant preliminary results are moved to \ref{section:omitted_proofs_1}, and the proof of Theorem \ref{thm:iff_conditions_for_stability} and relevant preliminary results are moved to \ref{section:omitted_proofs_2}.

\section{Preliminaries}\label{section:preliminaries}

\subsection{Preliminary assumptions and results}

This paper concerns stability and bifurcation properties of switched linear and homogeneous time-invariant, continuous-time dynamical systems in $\RR^{2}$. That is, for $x(t) \in \RR^{2}$, we consider the switched linear system
\begin{equation}\label{intro.1}
    \dot{x}(t) = A_{\sigma(t)} x(t), \hspace{20px} \sigma : \RR^{\geq0} \to \cP,
\end{equation}
and the switched homogeneous system
\begin{equation}\label{intro.1_homog}
    \dot{x}(t) = f_{\sigma(t)}(x(t)), \hspace{20px} \sigma : \RR^{\geq0} \to \cP,
\end{equation}
where $x(0) \in \RR^{2}$, $A_{p}$ is a real-valued $2 \times 2$ matrix for $p \in \cP$, $f_{p}$ is a $\cC^{1}$ homogeneous function of fixed degree $\beta$, and $\sigma$ is a switching signal as defined in Section \ref{section:intro}. Recall that $f_{p}$ being homogeneous of degree $\beta$ means that $f_{p}(cx_{1}, cx_{2}) = c^{\beta}f(x_{1}, x_{2})$ for all $c \in \RR$, $p \in \cP$, and $x_{1}, x_{2} \in \RR$.

Since each $f_{p}$ is $\cC^{1}$ in both the linear and homogeneous cases, we have existence and uniqueness of solutions of the unswitched subsystems (\cite{hirsch2013differential}, \S 17). We assume throughout the paper that the indexing set $\cP$ is finite, for example, $\cP = \{1, \ldots, m\}$ without loss of generality, and that $A_{p} \neq A_{q}$ and $f_{p} \neq f_{q}$ for $p \neq q$.

We will be primarily concerned with local and global uniform asymptotic stability of switched linear and homogeneous systems (\cite{liberzon2003switching}). We note that a necessary condition for uniform asymptotic stability is that each subsystem is asymptotically stable, since otherwise we can take the constant switching signal corresponding to a subsystem which is not asymptotically stable. In particular, for the linear switched system \eqref{intro.1} this means we require each matrix $A_{p}$ for $p \in \cP$ to be Hurwitz, that is, to have eigenvalues with negative real parts.

As shown in \cite{angeli1999note}, linear switched systems admit an equivalence between many of the common notions of stability, including uniform asymptotic stability. In particular, they are all equivalent to the notion of \textit{attractivity}. The system  \eqref{intro.1} is attractive if for every solution $x(t)$ of this system,
\begin{equation}\label{intro.3}
    \lim_{t \to \infty}||x(t)|| = 0,
\end{equation}
where $||\cdot||$ is the Euclidean norm in $\RR^{2}$. So, a sufficient and necessary condition for a lack of uniform asymptotic stability of \eqref{intro.1} is the existence of some $\sigma : \RR^{\geq0} \to \cP$ and an initial value $\xi \in \RR^{2}$ such that if $x(t)$ solves \eqref{intro.1} with switching signal $\sigma$ and $x(0)=\xi$, then $\lim_{t \to \infty}||x(t)||$ is either nonzero or does not exist.

In Section \ref{section:consequences}, we will study bifurcations of uniform asymptotic stability of switched linear systems as we vary the matrices $A_{p}$ defining the subsystems with respect to some external parameter. There we will consider the system
\begin{equation}\label{intro.2}
    \dot{x}(t) = A_{\sigma(t)}(\theta) x(t), \hspace{20px} \sigma : \RR^{\geq0} \to \cP,
\end{equation}
where $x(t) \in \RR^{2}$ and $A_{p}(\theta)$ is a real-valued $2 \times 2$ matrix for $p \in \cP$ which varies continuously with respect to the parameter $\theta \in [\alpha, \beta] \subseteq \RR$. A bifurcation in this context is a change in uniform asymptotic stability of the system. It is clear that the switched system \eqref{intro.2} loses uniform asymptotic stability when any one of the matrices $A_{p}(\theta)$ ceases to be Hurwitz, so for this reason we restrict our attention to the case when the matrices $A_{p}(\theta)$ are Hurwitz for all $p \in \cP$ and $\theta \in [\alpha, \beta]$ in Section \ref{section:consequences}.

\subsection{Analysis of degree two quadratic forms over $\RR$}

Let $A$ be a $2 \times 2$ real-valued matrix. We wish to classify all such matrices $A$ according to the zero set of their corresponding quadratic form $x^{\top}Ax$, $\{x : x^{\top}Ax = 0\}$. This quadratic form will turn out to play an important role in our analysis in Section \ref{section:stability_analysis}, since this quantity for fixed $x$ measures how much the vector $Ax$ is ``pointing away'' from the origin.

Let $A_{sym} = \frac{1}{2}(A + A^{\top})$, noting that $x^{\top}Ax = x^{\top}A_{sym}x$. For computational reasons and since the eigenvalues of symmetric matrices are always real, it will be convenient to work with $A_{sym}$ instead of $A$. Since $A_{sym}$ is symmetric, there is an orthogonal matrix $U$ such that
\begin{equation*}
    A_{sym} = U \begin{pmatrix} \mu_{1} & 0\\0 & \mu_{2} \end{pmatrix} U^{\top},
\end{equation*}
where $\mu_{1} \leq \mu_{2}$ are the eigenvalues of $A_{sym}$. Noting that there is a bijective correspondence between $(x_{1}, x_{2})^{\top} = x$ and $(y_{1},y_{2})^{\top} = y = U^{\top}x$, the zero set of the quadratic form $x^{\top}A_{sym}x$ can be written as
\begin{equation}\label{class.1}
    0 = x^{\top}A_{sym}x = x^{\top}U \begin{pmatrix} \mu_{1} & 0\\0 & \mu_{2} \end{pmatrix} U^{\top}x = y^{\top} \begin{pmatrix} \mu_{1} & 0\\0 & \mu_{2} \end{pmatrix} y = \mu_{1}y_{1}^{2} + \mu_{2}y_{2}^{2}.
\end{equation}
We have the following cases:
\begin{itemize}
    \item $\mu_{1}, \mu_{2} < 0$ or $\mu_{1}, \mu_{2} > 0$, so that $A_{sym}$ is negative or positive definite respectively. Then the equation $\mu_{1}y_{1}^{2} + \mu_{2}y_{2}^{2} = 0$ only admits the zero solution $y = 0$ and so $x^{\top}A_{sym}x=0$ if and only if $x = 0$.
    \item $\mu_{1} < 0$ and $\mu_{2}=0$, or $\mu_{1} = 0$ and $\mu_{2}>0$, so that $A_{sym}$ is negative semi-definite or positive semi-definite respectively. Then $x^{\top}A_{sym}x=0$ if and only if $\mu_{1}y_{1}^{2} = 0$ and $y_{2} \in \RR$ in the negative semi-definite case, or $\mu_{2}y_{2}^{2} = 0$ and $y_{1} \in \RR$ in the positive semi-definite case. So, the set of all $x$ satisfying \eqref{class.1} forms a one-dimensional subspace.
    \item $\mu_{1} = \mu_{2} = 0$, so that $A_{sym}$ is the zero matrix. Then \eqref{class.1} holds for all $x \in \RR^{2}$.
    \item $\mu_{1} < 0 < \mu_{2}$, so that $A_{sym}$ is indefinite. Then \eqref{class.1} holds if and only if $y_{2} = \pm \sqrt{\frac{-\mu_{1}}{\mu_{2}}}y_{1}$, so that all $x$ for which $x^{\top}Ax = 0$ form two distinct one-dimensional subspaces.
\end{itemize}

The following lemma proves the existence of regions $C_{i}^{+}$ on which we have a negative upper bound on a quantity that we will later relate to the growth of trajectories of the switched system \eqref{intro.1}; this result will be used in Section \ref{subsection:stability_under_assumptions} and \ref{section:omitted_proofs_1}. Below and throughout the rest of the paper, $S^{1}$ denotes the unit circle embedded into $\RR^{2}$.

\begin{lemma} \label{lemma:quadratic_form_negative_lower_bound}
Let $A$ be a $2 \times 2$ Hurwitz real-valued matrix with real eigenvalues $\lambda_{1} \leq \lambda_{2}<0$. Furthermore, assume $A$ is not a scalar multiple of the identity matrix. Then there exists an open cone $C^{+}$ with apex $0$ with the property that any function $y(t)$ solving $\dot{y} = Ay$ with $y(0) \notin \cl{C^{+}}$, $y(0) \neq 0$, eventually lies in $\cl{C^{+}} \setminus \{0\}$ for all large enough time $t>0$, and furthermore that
\begin{equation} \label{eq:quadratic_form_bounded_above_by_negative}
    \sup_{z \in S^{1} \cap \cl{C^{+}}}z^{\top}Az < -k < 0
\end{equation}
for some constant $k>0$.
\end{lemma}

\begin{remark}
We will show in Subsection \ref{subsection:trajectory_equivalence} that the assumption of $A$ not being a scalar multiple of the identity matrix yields no loss of generality when using this lemma. This is because if any subsystem of \eqref{intro.1} is defined by such a matrix, then the uniform stability properties of the original switched system are equivalent to the uniform stability properties of the switched system with that particular subsystem removed.
\end{remark}

\begin{proof} We split into cases depending on equality of the eigenvalues and the zero set of the quadratic form $x^{\top}Ax$, which by the above work is directly related to the eigenvalues of $A_{sym}$.

\textit{\underline{Case 1.}} Suppose $\lambda_{1} < \lambda_{2} < 0$ with corresponding eigenvectors $w_{1},w_{2}$, and that the eigenvalues of $A_{sym}$ are both negative. Then by the above work, we have $z^{\top}Az < 0$ for all $z \neq 0$. Taking $C^{+} = \RR^{2} \setminus \{0\}$ and noting that $S^{1}$ is compact and does not contain zero, the extreme value theorem provides a constant $k>0$ as in \eqref{eq:quadratic_form_bounded_above_by_negative}.

\textit{\underline{Case 2.}} Suppose $\lambda_{1} < \lambda_{2} < 0$ with corresponding eigenvectors $w_{1},w_{2}$, and that the eigenvalues of $A_{sym}$ are of opposite sign. As in the discussion preceding the lemma statement, let $v_{1}, v_{2}$ be linearly independent such that $v_{1,2}^{\top}Av_{1,2} = 0$, with the additional property that $w_{1} = \alpha v_{1} + \beta v_{2}$ for $\alpha, \beta \geq 0$. In fact we have $\alpha, \beta > 0$, since if (for example) $w_{1} = \alpha v_{1}$ then $w_{1}^{\top}Aw_{1}=0$; however, $y(t)$ solving $\dot{y} = Ay$ with $y(0) \in \{\alpha v_{1} : \alpha>0\}$ written explicitly is $y(t) = \alpha e^{\lambda_{1}t}v_{1}$, hence
\begin{equation*}
    \lambda ||y(t)|| = \frac{d}{dt}||y(t)|| = \frac{y(t)^{\top}A y(t)}{||y(t)||},
\end{equation*}
and in particular this holds for $t=0$. Thus
\begin{equation*}
    0 = y(0)^{\top}A y(0) = w_{1}^{\top}Aw_{1} = \lambda||y(0)||^{2} < 0,
\end{equation*}
a contradiction. Now consider the set $C_{1}^{+} = \{a(w_{1}+v_{1}) + b(w_{1}+v_{2}) : ab>0\}$. Suppose $y(t)$ solves $\dot{y} = Ay$ with $y(0) = \xi w_{1} + \zeta w_{2} \notin \cl{C^{+}}$, $y(0) \neq 0$. Then $y(t) = \xi e^{\lambda_{1}t}w_{1} + \zeta e^{\lambda_{2}t}w_{2}$ with $e^{\lambda_{1}t}/e^{\lambda_{2}t} \xrightarrow{t \to \infty} 0$, hence $y(t)$ approaches tangency with $\text{span}\{w_{1}\} \setminus \{0\} \subseteq C^{+}$, so for large enough $t$ we have $y(t) \in C^{+}$.

We will now show \eqref{eq:quadratic_form_bounded_above_by_negative}. Since $z^{\top}Az = 0$ if and only if $z \in \text{span}\{v_{1}\} \cup \text{span}\{v_{2}\}$, and $w_{1} \in C^{+}$ with $w_{1}Aw_{1}<0$, from
\begin{align*}
    \cl{C^{+}} \setminus \{0\} &= \{a(w_{1}+v_{1})+b(w_{1}+v_{1}) : a>0, b \geq 0 \text{ or } a \geq 0, b>0\} \\
    &= \{(a\alpha + a + b\alpha)v_{1} + (a\beta + b\beta +b)v_{2} : a>0, b \geq 0 \text{ or } a \geq 0, b>0\} \\
    &\subseteq \{av_{1} + bv_{2} : ab>0\},
\end{align*}
we have $z^{\top}Az < 0$ for all $z \in \cl{C^{+}} \setminus \{0\}$. Thus by compactness of $S^{1} \cap \cl{C^{+}} \subseteq \cl{C^{+}} \setminus \{0\}$, there exists some $k>0$ for which \eqref{eq:quadratic_form_bounded_above_by_negative} holds.

\textit{\underline{Case 3.}} Suppose $\lambda_{1} < \lambda_{2} < 0$ with corresponding eigenvalues $w_{1},w_{2}$, and that there is exactly one zero eigenvalue of $A_{sym}$. Let $v \in \RR^{2} \setminus \{0\}$ be a vector with $v^{\top}Av = 0$, as in the discussion preceding this lemma. Similar to the previous case, $w_{1}^{\top}Aw_{1} < 0$, hence $w_{1}, v$ must be linearly independent. Setting $C^{+} = \{a(w_{1}+v) + b(w_{1}-v) : ab>0\}$, since $\text{span}\{w_{1}\} \setminus \{0\} \subseteq C^{+}$ and $A$ has two distinct eigenvalues, the exact same reasoning as in the previous case shows that if $y(t)$ solves $\dot{y} = Ay$ with $y(0) \notin \cl{C^{+}}$, $y(0) \neq 0$, then $y(t) \in C^{+}$ for all large enough $t$. Finally, since
\begin{align*}
    \cl{C^{+}} \setminus \{0\} &= \{a(w_{1}+v)+b(w_{1}-v) : a>0, b\geq0 \text{ or } a\geq0, b>0\} \\
    &\subseteq \RR^{2} \setminus \text{span}\{v\}
\end{align*}
and $w_{1} \in \cl{C^{+}} \setminus \{0\}$ with $w_{1}^{\top}Aw_{1}<0$, as before we can conclude the existence of $k>0$ satisfying \eqref{eq:quadratic_form_bounded_above_by_negative}.

\textit{\underline{Case 4.}} Suppose $A$ has one repeated eigenvalue $\lambda$. By assumption $A$ cannot be a scalar multiple of the identity, so $A$ must be similar to the matrix $\begin{pmatrix}\lambda & 1 \\ 0 & \lambda\end{pmatrix}$. Let $w_{1}$ be an eigenvector of $A$ with corresponding generalized eigenvector $w_{2}$. As in the previous two cases, $w_{1}^{\top}Aw_{1}<0$. Denote 
\begin{equation*}
    \xi' = \inf\{\xi \geq 0 : (w_{1}+\xi w_{2})^{\top}A(w_{1}+\xi w_{2}) = 0\} \in \RR^{\geq0} \cup \{\infty\},
\end{equation*}
and let $\alpha := \min \{1, \frac{1}{2} \xi'\}$. Define
\begin{align*}
    C^{+} &= \{aw_{1}+b(w_{1}+\alpha w_{2}) : ab>0\} \\
    &= \{a(w_{1}+b \alpha w_{2}) : |a|>0, b \in (0,1)\}.
\end{align*}
Then for $y(t)$ solving $\dot{y} = Ay$ with $y(0) = \xi w_{1} + \zeta w_{2} \notin \cl{C^{+}}$, $y(0) \neq 0$, we have
\begin{equation*}
    y(t) = \xi e^{\lambda t}w_{1} + \zeta e^{\lambda t}(tw_{1} + w_{2}) = (\xi + t \zeta)e^{\lambda t}w_{1} + \zeta e^{\lambda t}w_{2}.
\end{equation*}
If $\zeta = 0$, $y(t) = \xi e^{\lambda t}w_{1} \in \text{span}\{w_{1}\} \setminus \{0\} \subseteq \cl{C^{+}} \setminus \{0\}$ for all $t \geq 0$. If $\zeta > 0$, then for all $t$ large enough such that $(\xi + t \zeta)e^{\lambda t} > 0$ and $\zeta e^{\lambda t} \in (0, \alpha)$, we have $y(t) \in C^{+}$. Analogous reasoning shows the same conclusion for $\zeta<0$. Finally, note that by definition of $\alpha$ we have
\begin{equation*}
    (w_{1}+b \alpha w_{2})^{\top}A(w_{1}+b \alpha w_{2})>0
\end{equation*}
for all $b \in [0,1]$, hence $z^{\top}Az>0$ for $z \in \cl{C^{+}} \setminus \{0\}$. The same reasoning as in the previous cases shows the existence of $k>0$ such that \eqref{eq:quadratic_form_bounded_above_by_negative} holds.
\end{proof}

\subsection{Trajectory equivalence simplification} \label{subsection:trajectory_equivalence}

For the switched system \eqref{intro.1} with switching signal $\sigma \in \Sigma_{\cP}$, denote by $\cO_{\sigma}^{+}(x)$ the forward orbit of the switched system with initial condition $x \in \RR^{2}$. We can write $\cO_{\sigma}^{+}(x)$ as the union of trajectory arcs in between the points of discontinuity of $\sigma$. More precisely, suppose $0 < t_{1} < t_{2} < \cdots$ are the (possibly infinitely many) points of discontinuity of $\sigma$ and set $t_{0} = 0$. If there are finitely many such points, say $t_{1} < t_{2} < \cdots < t_{n}$, set $t_{n+1} = \infty$. Let $p_{i} \in \cP$ denote the value taken by $\sigma$ in the interval $[t_{i-1}, t_{i})$, set $\tau_{i} = t_{i}-t_{i-1}$, and formally set $x_{0}(\tau_{0}) := x$. Then
\begin{align*}
    \cO_{\sigma}^{+}(x) &= \bigcup_{i=1}^{d} \{x_{i}(t) : \dot{x}_{i}(t) = A_{p_{i}}x_{i}(t), \ x_{i}(0) = x_{i-1}(\tau_{i-1}), \ 0 \leq t < \tau_{i} \} =: \bigcup_{i=1}^{d} \cX_{i},
\end{align*}
where $d=\infty$ in the case of infinitely many $t_{i}$ and $d=n+1$ otherwise, and where each set
\begin{equation}\label{eq:calA_definition}
    \cX_{i} = \{ x_{i}(t) : x_{i}(t) = \exp(A_{p_{i}}t) x_{i-1}(\tau_{i-1}), \ 0 \leq t < \tau_{i} \}
\end{equation}
is a trajectory arc between successive switching times of the system.

Denote the set of all possible forward orbits of under arbitrary switching signals by
\begin{equation*}
    \cF(x, \cP) = \{\cO_{\sigma}^{+}(x) : \sigma \in \Sigma_{\cP}\}.
\end{equation*}
The following lemma shows that the set of forward orbits of a switched system does not change if we add or remove time-scaled duplicate subsystems of other subsystems corresponding to indices in $\cP$.

\begin{lemma}\label{lemma:time_scaling_orbit_equivalence}
Let $p, q \in \cP$ such that $A_{q} = \alpha A_{p}$ for some $\alpha>0$. Then for all $x \in \RR^{2}$, $\cF(x, \cP) = \cF(x, \cP \setminus \{q\})$.
\end{lemma}
\begin{proof}
Since $\cP \setminus \{q\} \subseteq \cP$, it is straightforward to see that $\cF(x, \cP \setminus \{q\}) \subseteq \cF(x, \cP)$. Now let $\cO_{\sigma}^{+}(x) \in \cF(x, \cP)$, and write
\begin{equation*}
     \cO_{\sigma}^{+}(x) = \bigcup_{i=1}^{d} \cX_{i}
\end{equation*}
with $\cX_{i}$ as in \eqref{eq:calA_definition}. Setting
\begin{equation*}
    B_{p_{i}} = \begin{cases} A_{p} & \mathrm{if} \ p_{i} = q \\ A_{p_{i}} & \mathrm{if} \ p_{i} \neq q \end{cases} \hspace{20px} \mathrm{and} \hspace{12px} \tau'_{i} = \begin{cases} \alpha \tau_{i} & \mathrm{if} \ p_{i} = q \\ \tau_{i} & \mathrm{if} \ p_{i} \neq q \end{cases},
\end{equation*}
we have
\begin{align*}
    \cX_{i} &= \{ x_{i}(t) : x_{i}(t) = \exp(A_{p_{i}}t) x_{i-1}(\tau_{i-1}), \ 0 \leq t < \tau_{i} \} \\
    &= \{ x_{i}(t) : x_{i}(t) = \exp(B_{p_{i}}t) x_{i-1}(\tau'_{i-1}), \ 0 \leq t < \tau'_{i} \}.
\end{align*}
Thus there exists some $\sigma' \in \Sigma_{\cP \setminus \{q\}}$ for which $\cO_{\sigma}^{+}(x) = \cO_{\sigma'}^{+}(x) \in \cF(x, \cP \setminus \{q\})$, proving $\cF(x, \cP) \subseteq \cF(x, \cP \setminus \{q\})$.
\end{proof}

\begin{remark}\label{rmk:removal_of_scalar_subsystems}
Since uniform asymptotic stability of \eqref{intro.1} is equivalent to all trajectories tending to the origin in forward time, Lemma \ref{lemma:time_scaling_orbit_equivalence} shows that if $p,q \in \cP$ are such that $A_{q} = \alpha A_{p}$ for some $\alpha>0$, then the switched system \eqref{intro.1} is uniformly asymptotically stable under arbitrary switching signals in $\Sigma_{\cP}$ if and only if it is uniformly asymptotically stable under arbitrary switching signals in $\Sigma_{\cP \setminus \{q\}}$. Hence, if we are trying to show attractivity of a given forward trajectory of the switched system \eqref{intro.1} to the origin during the course of a proof, we may without loss of generality rescale the matrix $A_{p}$ of any subsystem $\dot{x} = A_{p}x$ and consider the same (as a set) forward trajectory but with time rescaled. This time-rescaled forward trajectory will have identical attractivity properties to the original one.
\end{remark}

\begin{remark}\label{rmk:can_exclude_minus_identity}
It is a similarly straightforward to show that if there exists some $p \in \cP$ such that $A_{p} = -I_{2}$ (the $2 \times 2$ identity matrix), then \eqref{intro.1} is uniformly asymptotically stable under switching signals in $\Sigma_{\cP}$ if and only if it is uniformly asymptotically stable under switching signals in $\Sigma_{\cP \setminus \{p\}}$. Intuitively, this is true because for any forward orbit $\cO_{\sigma}^{+}(x) \in \cF(x,\cP)$ we may ``insert'' any number of unswitched trajectories $x_{i}(t)$ solving $\dot{x}_{i}(t) = -x_{i}(t)$ and the absolute value of the full trajectory $||x(t)||$ will only decrease in those unswitched segments, with linearity of each subsystem guaranteeing that attractivity behaviour of the full trajectory is preserved.
\end{remark}

For distinct $p,q \in \cP$, define the function
\begin{equation}\label{eq:Delta_pq}
    \Delta_{p,q}(x) = \det(A_{p}x \ | \ A_{q}x)
\end{equation}

Expanding out this expression it is easily verified that $\det(A_{p}x \ | \ A_{q}x) = x^{\top}(A_{p}^{\top} R(-\pi/2) A_{q})x$, where $R(-\pi/2) = \begin{pmatrix}0 & 1\\ -1 & 0\end{pmatrix}$ is the $2 \times 2$ rotation matrix by $-\pi/2$. For distinct $p,q \in \cP$, the zero set $\Delta_{p,q}^{-1}(0)$ is exactly the set of vectors $x$ for which $A_{p}x$ and $A_{q}x$ are collinear. By the discussion before Lemma \ref{lemma:quadratic_form_negative_lower_bound} in the preceding subsection, the fact that we may rewrite each set $\Delta_{p,q}^{-1}(0)$ for distinct $p,q \in \cP$ as a quadratic form implies that each such set is either the origin, a single line through the origin, a pair of distinct lines through the origin, or the entire plane.

\begin{lemma}\label{lemma:no_scalar_multiples_of_A}
Suppose for distinct $p,q \in \cP$ that $\Delta_{p,q}^{-1}(0) = \RR^{2}$. Then $A_{p}$ is a scalar multiple of $A_{q}$.
\end{lemma}
\begin{proof}
Since $\Delta_{p,q}^{-1}(0) = \RR^{2}$, we have that $\det(A_{p}x \ | \ A_{q}x) = 0$ for all $x \in \RR^{2}$. Thus the vectors $A_{p}x$ and $A_{q}x$ are linearly dependent for all $x$, so there exists some function $\beta: \RR^{2} \to \RR$ for which 
\begin{equation}\label{eq:beta_condition}
    A_{p}x = \beta(x) \cdot A_{q}x
\end{equation}
for all $x \in \RR^{2}$. In particular this holds for $\alpha x + y \in \RR^{2}$ where $x,y \in \RR^{2}$ and $\alpha \in \RR$ are arbitrary, and by linearity we obtain
\begin{equation*}
    \alpha \cdot A_{p}x + A_{p}y = \alpha \cdot \beta(\alpha x + y) \cdot A_{q} x + \beta(\alpha x + y) \cdot A_{q} y.
\end{equation*}
Substituting $A_{p}x = \beta(x) \cdot A_{q}x$ and $A_{p}y = \beta(y) \cdot A_{q}y$ yields
\begin{equation*}
    \alpha \cdot \beta(x) \cdot A_{q}x + \beta(y) \cdot A_{q}y = \alpha \cdot \beta(\alpha x + y) \cdot A_{q} x + \beta(\alpha x + y) \cdot A_{q} y.
\end{equation*}
Denote by $e_{1},e_{2}$ the standard basis vectors of $\RR^{2}$ and note that since $A_{q}$ is Hurwitz, it is invertible. Substituting $x = A_{q}^{-1}e_{1}$ and $y = c \cdot A_{q}^{-1}e_{2}$ for arbitrary $c \in \RR$ and simplifying, we obtain
\begin{equation*}
    \begin{pmatrix} \alpha \beta(A_{q}^{-1}e_{1}) \\ c \beta(A_{q}^{-1}e_{2}) \end{pmatrix} = \begin{pmatrix} \alpha \beta(\alpha A_{q}^{-1}e_{1} + c A_{q}^{-1}e_{2}) \\ c \beta(\alpha A_{q}^{-1}e_{1} + c A_{q}^{-1}e_{2}) \end{pmatrix}.
\end{equation*}
Equating the first row and using the fact that $A_{q}^{-1}e_{1}, A_{q}^{-1}e_{2}$ are linearly independent (since $e_{1}, e_{2}$ are linearly independent and $A_{q}$ is invertible) and $\alpha,c \in \RR$ are arbitrary, we conclude that $\beta(x)$ is constant for $x \neq 0$; since $\beta(0)$ may be arbitrary for the equality \eqref{eq:beta_condition} to hold, we may take $\beta$ to be a constant function, which completes the proof.
\end{proof}

Henceforth, we will assume $\cP$ contains no indices $p, q$ with $\Delta_{p,q}^{-1}(0) = \RR^{2}$; this assumption will be used in the construction of the set $\cS$ at the beginning of the next section. We claim this assumption does not change the uniform asymptotic stability properties of \eqref{intro.1}. If we had any two such indices $p,q$, then Lemma \ref{lemma:no_scalar_multiples_of_A} shows that $A_{q} = \alpha A_{p}$ for some scalar $\alpha$. If $\alpha > 0$ the discussion in Remark \ref{rmk:removal_of_scalar_subsystems} allows us to remove one of $p,q$ without changing stability properties of the switched system. If $\alpha \leq 0$, since $A_{p}$ is Hurwitz we cannot have that $A_{q} = \alpha A_{p}$ is Hurwitz as well, so this case leads to a contradiction. Our claim is proven. Additionally, by Remark \ref{rmk:can_exclude_minus_identity} we may further assume $A_{p} \neq -\alpha I_{2}$ for all $\alpha > 0$ and $p \in \cP$.

\section{Stability analysis}\label{section:stability_analysis}

\subsection{Analysis under additional assumptions}\label{subsection:stability_under_assumptions}

The main result of this subsection is Theorem \ref{thm:iff_conditions_for_stability_assuming_all_edges_exist}, which gives necessary and sufficient conditions for stability of \eqref{intro.1} under an additional assumption. This assumption will be made clear below.

First, we will define a set of rays $S_{[i]}$, $i \in \{1, \ldots, n\}$, which will allow us to simplify the worst-case behaviour of the switched system to the behaviour of only those trajectories which are allowed to switch when they intersect one of these rays.

Consider now those indices $p \in \cP$ such that the following holds:
\begin{enumerate}
    \item \label{item_1:starting_assumptions_on_A_p} The eigenvalues $\lambda_{1} \leq \lambda_{2} < 0$ of $A_{p}$ are real.
    \item \label{item_2:starting_assumptions_on_A_p} We are not in the case where the eigenvalues $\mu_{1} \leq \mu_{2}$ of $(A_{p})_{sym}$ are both negative and $\lambda_{1} < \lambda_{2}$ (that is, we are not in Case 1 in the proof of Lemma \ref{lemma:quadratic_form_negative_lower_bound}).
\end{enumerate}
If $\lambda_{1} < \lambda_{2} < 0$, denote by $w_{1},w_{2}$ the corresponding eigenvectors, and if $\lambda_{1} = \lambda_{2}$, denote by $w_{1}$ an eigenvector and by $w_{2}$ the corresponding generalized eigenvector. Since we are working in $\RR^{2}$, it is only possible to have one eigenvalue with infinitely many linearly independent eigenvectors if $A_{p}$ is a scalar multiple of the identity matrix; this case is thus ruled out by our assumption above that $A_{p} \neq -\alpha I_{2}$. With $v_{1}, v_{2}, v, \alpha$ as defined in the proof of Lemma \ref{lemma:quadratic_form_negative_lower_bound} (these quantities implicitly depend on the matrix $A_{p}$), let
\begin{equation} \label{eq:Cp+_cases_definition}
    C_{p}^{+}=\begin{cases}
    			\{a(w_{1}+v_{1}) + b(w_{1}+v_{2}) : ab>0\}, & \text{if $\lambda_{1} < \lambda_{2}$ and $\mu_{1} < 0 < \mu_{2}$,}\\
                \{a(w_{1}+v) + b(w_{1}-v) : ab>0\}, & \text{if $\lambda_{1} < \lambda_{2}$ and $\mu_{i} = 0, \mu_{j} \neq 0$ for $i,j \in \{1,2\}$,} \\
                \{aw_{1}+b(w_{1}+\alpha w_{2}) : ab>0\}, & \text{if $\lambda_{1} = \lambda_{2}$}.
    		 \end{cases}
\end{equation}
Denote by $\cS_{1}$ the union of $\partial C_{p}^{+}$ over all such $p \in \cP$ with $A_{p}$ and $(A_{p})_{sym}$ satisfying \ref{item_1:starting_assumptions_on_A_p} and \ref{item_2:starting_assumptions_on_A_p} above. Now let $\cS_{2}$ denote the collection of all sets $\Delta_{p,q}^{-1}(0)$ for distinct $p,q \in \cP$, noting that $\Delta_{p,q}^{-1}(0) \neq \RR^{2}$ by the assumption above. Finally, let $\cS_{3}$ be the set of all rays $\{a w : a > 0\}$ where $w$ is an eigenvector of any matrix $A_{p}$, $p \in \cP$, with a real corresponding eigenvalue. Again since  $A_{p} \neq -\alpha I_{2}$ for all $p \in \cP$, the set $\cS_{3} \setminus \{0\}$ is a finite collection of half-open rays emitting from the origin.

We emphasize that only those matrices $A_{p}$ with real eigenvalues may contribute to the sets $\cS_{1}$ and $\cS_{3}$, in addition to possibly contributing to the set $\cS_{2}$. Matrices $A_{p}$ whose eigenvalues have nonzero imaginary parts may \textit{only} contribute to the set $\cS_{2}$.

Denote by $\cS$ the collection of connected components of the set $(\cS_{1} \cup \cS_{2} \cup \cS_{3}) \setminus \{0\}$. Since $\cS$ is a finite set (say, of cardinality $n$) and the elements of $\cS$ are all half-open rays emanating from the origin, for ease of notation we can label the rays using numbers modulo $n$: pick an arbitrary ray from the collection and label it $[1]$ corresponding to the equivalence class of $0$ mod $n$, and going counterclockwise about the origin, label each consecutive element of $\cS$ as $[2], [3], \ldots, [n]$, until we circle around to our initial ray labelled $[n+1] = [1]$. That is, we write $\cS = \{S_{[1]}, S_{[2]}, \ldots, S_{[n]}\}$, so that for each $i \in \{1, \ldots, n\}$ we have that $S_{[i+1]}$ is the ray following $S_{[i]}$ in the counterclockwise direction.

\begin{figure}[h]
\begin{center}
\begin{tikzpicture}
\draw[thick,->,gray] (0,0) -- (3,0);
\draw[thick,->,gray] (0,0) -- (0,3);
\draw[thick,->,gray] (0,0) -- (-3,0);
\draw[thick,->,gray] (0,0) -- (0,-3);

\node[draw=none] (ellipsis1) at (0.5,2.2) {$\cdots$};
\node[draw=none] (ellipsis1) at (0.5,-2.4) {$\cdots$};

\draw[thick,->] (0,0) -- (1.3,2.7037) node[anchor=north west] {$S_{[3]}$};
\draw[thick,->] (0,0) -- (2.7,1.3076) node[anchor=north west] {$S_{[2]}$};
\draw[thick,->] (0,0) -- (2.8,-1.077) node[anchor=south west] {$S_{[1]}$};
\draw[thick,->] (0,0) -- (1.2,-2.828) node[anchor=south west] {$S_{[n]}$};

\draw[dashed] (0,0) -- (1.15*1.3,1.15*2.7037);
\draw[dashed] (0,0) -- (1.2*2.7,1.2*1.3076);
\draw[dashed] (0,0) -- (1.2*2.8,-1.2*1.077);
\draw[dashed] (0,0) -- (1.15*1.2,-1.15*2.828);
\end{tikzpicture}
\end{center}
\caption{A sample ordering of $S_{[i]}$ rays.}
\end{figure}
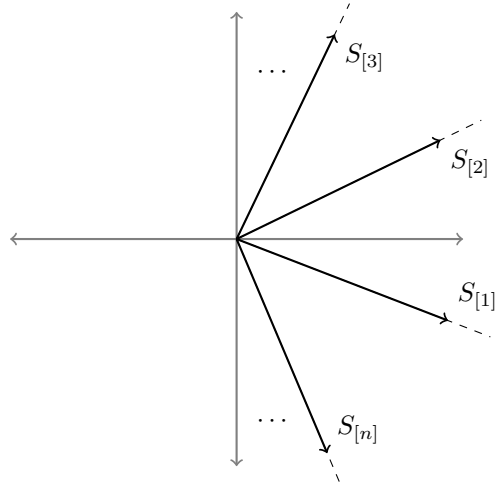

\begin{remark}\label{rmk:no_rays}
For the rest of the paper we will assume that $\cS$ is not empty. Otherwise, we may arbitrarily add any single half-open ray emitting from the origin to $\cS$ and all of the work below will hold. In fact, it is easy to verify using Theorem \ref{thm:iff_conditions_for_stability} in Section \ref{subsection:stability_without_assumptions} that all such switched systems for which $\cS$ is empty are uniformly asymptotically stable.
\end{remark}

Consider the forward trajectory of a solution to the system 
\begin{equation}\label{stability.2}
    \dot{x} = A_{p}x
\end{equation}
for some fixed $p \in \cP$, with initial condition $\xi \in S_{[i]}$. Since each $A_{p}$ is Hurwitz and the phase space is two-dimensional, the forward trajectory of a solution to this initial value problem satisfies one of the following:
\begin{itemize}
    \item The forward trajectory intersects $S_{[i-1]}$ in finite time, and does not intersect $S_{[i+1]}$ prior to this.
    \item The forward trajectory intersects $S_{[i+1]}$ in finite time, and does not intersect $S_{[i-1]}$ prior to this.
    \item The forward trajectory does not intersect $S_{[i-1]}$ or $S_{[i]}$ or $S_{[i+1]}$ aside from the initial time, and tends to the origin.
    \item The forward trajectory is contained in $S_{[i]}$ for all forward time, and tends to the origin.
\end{itemize}
To each $S_{[i]}$ we can associate two sets $\cP^{+}_{[i]}$ and $\cP^{-}_{[i]}$, each containing elements of $\cP$, as follows. For every $p \in \cP$, if the forward trajectory of a solution of \eqref{stability.2} with initial condition $\xi \in S_{[i]}$ intersects $S_{[i+1]}$ in finite time before it (possibly) intersects $S_{[i-1]}$, then $p \in \cP^{+}_{[i]}$; likewise, if the forward trajectory intersects $S_{[i-1]}$ in finite time before intersecting $S_{[i+1]}$, then $p \in \cP^{-}_{[i]}$. Since the subsystems of \eqref{intro.1} are linear, in particular they are homogeneous, so the above definition is well-defined (it does not depend on the choice of initial value $\xi \in S_{[i]}$).

The following lemma uses linearity of the unswitched subsystems to show that crossing times of unswitched trajectories with $S_{[i+1]}$ starting on $S_{[i]}$ are the same, and likewise for crossing times of unswitched trajectories with $S_{[i-1]}$ starting on $S_{[i]}$.

\begin{lemma}\label{lemma:same_hitting_time}
Let $i \in \{1, \ldots, n\}$ and $p \in \cP^{+}_{[i]}$. Then there exists a unique positive real number $\tau$ depending on $[i]$ and $p$ such that the following holds. For $\xi \in S_{[i]}$ and $\alpha \in \RR^{+}$, denoting by $x_{1}(t)$ and $x_{2}(t)$ the solutions to \eqref{stability.2} with initial conditions $\xi, \alpha \xi$ respectively, we have:
\begin{enumerate}
    \item $x_{1}(\tau), x_{2}(\tau) \in S_{[i+1]}$;
    \item For all $0 < t < \tau$, $x_{1}(t) \notin S_{[i+1]}$ and $x_{2}(t) \notin S_{[i+1]}$.
\end{enumerate}
The analogous result holds if we replace $p \in \cP^{+}_{[i]}$ with $p \in \cP^{-}_{[i]}$ and $S_{[i+1]}$ with $S_{[i-1]}$.
\end{lemma}
\begin{proof}
Let $\tau>0$ be the (unique) first time for which $x_{1}(t)$ intersects $S_{[i+1]}$; such a time exists since $p \in \cP^{+}_{[i]}$. The solutions to the linear system \eqref{stability.2} with initial condition $\xi, \alpha \xi$ can be written as $x_{1}(t) = e^{A_{p}t}(\xi)$, $x_{2}(t) = e^{A_{p}t}(\alpha\xi)$ respectively. Thus 
\begin{equation*}
    x_{2}(\tau) = e^{A_{p}\tau}(\alpha\xi) = \alpha e^{A_{p}\tau}(\xi) = \alpha x_{1}(\tau) \in S_{[i+1]}.
\end{equation*}
Suppose $x_{2}(t) \in S_{[i+1]}$ for some $0 < t < \tau$. Then
\begin{equation*}
    x_{1}(t) = e^{A_{p}t}(\xi) = \frac{1}{\alpha} e^{A_{p}t}(\alpha\xi) = \frac{1}{\alpha}x_{2}(t) \in S_{[i+1]},
\end{equation*}
contradicting the choice of $\tau$.
\end{proof}

We now begin to construct a weighted directed graph corresponding to system \eqref{intro.1}, which will allow us to greatly simplify the analysis of stability and bifurcation properties of the switched system. Define the vertex and (directed) edge sets
\begin{equation}\label{eq:V_and_E_definition}
\begin{split}
    &V = \{[1], [2], \ldots [n]\}, \\
    &E = \{({[i],[i+1])} : \cP^{+}_{[i]} \neq \emptyset\} \cup \{({[i],[i-1])} : \cP^{-}_{[i]} \neq \emptyset\}.
\end{split}
\end{equation}
For our analysis it will be necessary to define a weight function $w: E \to \mathbb{R}^{+}$ for the graph $(V, E)$, where $w([i], [i+1])$ denotes how far the trajectory solving system $\eqref{intro.1}$ can travel away from the origin going from $S_{[i]}$ to $S_{[i+1]}$ and $w([i+1], [i])$ denotes how far the trajectory can travel away from the origin going from $S_{[i+1]}$ to $S_{[i]}$. Notably, the order of $[i]$ and $[i+1]$ as arguments of function $w(\cdot, \cdot)$ matters. More precisely, denote by $\tau_{+}$ the unique crossing time from $S_{[i]}$ to $S_{[i+1]}$ for the unswitched trajectory corresponding to the subsystem with index $p_{[i]}^{+} \in \cP_{[i]}^{+}$, which in given by Lemma \ref{lemma:same_hitting_time}; this number $\tau_{+}$ implicitly depends on $[i]$. For $x(t)$ solving $\dot{x} = A_{p_{[i]}^{+}}x$ with initial condition $x(0) = \xi \in S_{[i]}$, define
\begin{equation}\label{eq:weight_defn_+}
    w([i], [i+1]) = \dfrac{||x(\tau_{+})||}{||x(0)||}.
\end{equation}
Likewise, denoting by $\tau_{-}$ the unique crossing time from $S_{[i+1]}$ to $S_{[i]}$ for the unswitched trajectory corresponding to $p_{[i+1]}^{-} \in \cP_{[i-1]}^{-}$ ($\tau_{-}$ also implicitly depends on $[i]$), for $x(t)$ solving $\dot{x} = A_{p_{[i+1]}^{-}}x$ with initial condition $x(0) = \xi \in S_{[i+1]}$ we define
\begin{equation}\label{eq:weight_defn_-}
    w([i+1], [i]) = \dfrac{||x(\tau_{-})||}{||x(0)||}.
\end{equation}
Lemma \ref{lemma:same_hitting_time} shows that $w([i], [i+1])$ and $w([i+1], [i])$ do not depend on the initial condition $\xi \in S_{[i]}$ of $x(t)$, since linearly scaling the initial condition scales $x(\tau_{\pm})$ linearly by the same amount.

We note that the values $w([i],[i+1])$ and $w([i+1],[i])$ need not exist for every index $[i]$, since the function $w$ is only defined on the set of edges $E$ of the graph $(V,E)$. In particular, a pair $([i],[i+1])$ is not in $E$ if there is no unswitched trajectory corresponding to a subsystem of \eqref{intro.1} which reaches the ray $S_{[i+1]}$ in finite time when starting from the ray $S_{[i]}$. This may happen, for example, if the ray $S_{[i+1]}$ is contained in a one-dimensional invariant subspace of all the subsystems.

The following two propositions are key to giving us complete control over the existence of a trajectory of \eqref{intro.1} which is not attracted to the origin, under the additional assumption that $w([i], [i+1])$ and $w([i], [i+1])$ both exist for all $i \in \{1, \ldots, n\}$. Their proofs, which require a large amount of mathematical machinery, will be given in \ref{section:omitted_proofs_1}.

\begin{propn} \label{propn:crossing_between_sector_boundaries}
Denote by $C$ the open cone bounded between $S_{[i]}$ and $S_{[i+1]}$ for some $i \in \{1, \ldots, n\}$ in the counterclockwise direction. Let $x(t)$ be a given solution to the switched system \eqref{intro.1} with initial condition $x(0) \in \partial C \setminus \{0\}$ and arbitrary fixed switching signal $\sigma$, and suppose $\tau = \sup \{t>0 : x(t) \in \cl{C} \setminus \{0\}\} < \infty$. Suppose further that $w([i], [i+1])$, $w([i+1], [i])$ both exist and $w([i], [i+1]) \cdot w([i+1], [i]) < 1$. Since $\tau < \infty$, one of the following possibilities must hold:
\begin{enumerate}
\item \label{item:1_crossing_between_sector_boundaries} $x(0) \in S_{[i]}$ and $x(\tau) \in S_{[i+1]}$. Then $||x(\tau)|| \leq w([i], [i+1]) \cdot ||x(0)||$.
\item \label{item:2_crossing_between_sector_boundaries} $x(0) \in S_{[i+1]}$ and $x(\tau) \in S_{[i]}$. Then $||x(\tau)|| \leq w([i+1], [i]) \cdot ||x(0)||$.
\item \label{item:3_crossing_between_sector_boundaries} $x(0) \in S_{[i]}$ and $x(\tau) \in S_{[i]}$. Then  $||x(\tau)|| < ||x(0)||$.
\item \label{item:4_crossing_between_sector_boundaries} $x(0) \in S_{[i+1]}$ and $x(\tau) \in S_{[i+1]}$. Then  $||x(\tau)|| < ||x(0)||$.
\end{enumerate}
\end{propn}

Intuitively, this proposition allows us to control the growth of trajectories of the switched system \eqref{intro.1}, in terms of $w([i], [i+1])$ and $w([i+1], [i])$, at the times when the trajectories cross the sector boundaries $S_{[i]}$. Cases \ref{item:1_crossing_between_sector_boundaries} and \ref{item:2_crossing_between_sector_boundaries} say that for every trajectory traversing the cone $C$ from one sector boundary to another, the norm of the trajectory $||x(t)||$ grows by at most a factor of the corresponding weight $w([i], [i+1])$ or $w([i+1], [i])$. Cases \ref{item:3_crossing_between_sector_boundaries} and \ref{item:4_crossing_between_sector_boundaries} imply that if a trajectory starts on some sector boundary, spends time in an adjacent closed cone $\cl C \setminus \{0\}$ (possibly without hitting any other sector boundaries), and then crosses the same sector boundary at which it started, then the norm $||x(t)||$ strictly decreases from the initial time $t = 0$ to the hitting time $t = \tau$.

\begin{propn}\label{propn:staying_within_sector_boundaries}
Let $i \in \{1, \ldots, n\}$ be arbitrary and suppose that for $j \in \{i, i+1\}$ and all $p \in \cP_{[j]}^{+}, q \in \cP_{[j+1]}^{-}$ we have $\Delta_{p,q}(x) > 0$ for all nonzero $x$ in the closed cone between $S_{[j]}$ and $S_{[j+1]}$ in the counterclockwise direction. Denote by $C$ the open cone bounded between $S_{[i]}$ and $S_{[i+1]}$ in the counterclockwise direction. Let $x(t)$ be a given solution to the switched system \eqref{intro.1} with initial condition $x(0) \in \cl{C} \setminus \{0\}$ and arbitrary fixed switching signal $\sigma$, and suppose $x(t) \in \cl{C} \setminus \{0\}$ for all $t \geq 0$. Suppose further that $w([i], [i+1])$, $w([i+1], [i])$ both exist and $w([i], [i+1]) \cdot w([i+1], [i]) < 1$. Then $\lim_{t \to \infty} x(t) = 0$.
\end{propn}

This proposition treats the complementary case to Proposition \ref{propn:crossing_between_sector_boundaries}, wherein a trajectory remains in some closed region $\cl C \setminus \{0\}$, possibly without hitting any sector boundaries, for all forward time. Of course in this situation we want the trajectory to tend to the origin, which is exactly what this proposition guarantees.

Finally, we state a preliminary version of our final theorem concerning necessary and sufficient conditions for uniform asymptotic stability of \eqref{intro.1}; again, the (technical) proof of this theorem is delegated to \ref{section:omitted_proofs_1}. In the below theorem, we require the critical assumption that $w([i], [i+1])$ and $w([i+1], [i])$ exist for all $i \in \{1, \ldots, n\}$; in Theorem \ref{thm:iff_conditions_for_stability}, this assumption will be dropped.

\begin{theorem} \label{thm:iff_conditions_for_stability_assuming_all_edges_exist}
Assume $w([i], [i+1])$ and $w([i+1], [i])$ exist for all $i \in \{1, \ldots, n\}$. Then the switched system \eqref{intro.1} is uniformly asymptotically stable if and only if all of the following conditions hold:
\begin{enumerate}
\item \label{item:1_nonfinal_stability_condition} For $i \in \{1, \ldots, n\}$ and all $p \in \cP_{[i]}^{+}, q \in \cP_{[i+1]}^{-}$, if $x \neq 0$ is in the closed cone between $S_{[i]}$ and $S_{[i+1]}$ in the counterclockwise direction then $\Delta_{p,q}(x) > 0$.
\item \label{item:3_nonfinal_stability_condition} For all $i \in \{1, \ldots, n\}$, both $\prod_{i=1}^{n}w([i], [i+1])<1$ and $\prod_{i=1}^{n}w([i+1], [i])<1$ hold.
\end{enumerate}
\end{theorem}

\begin{remark}\label{rmk:condition_and_graph_equivalence}
Using Lemma \ref{lemma:main_thm_condition_redundancy} from \ref{section:omitted_proofs_1}, it is straightforward to see that \ref{item:1_nonfinal_stability_condition} and \ref{item:3_nonfinal_stability_condition} are together equivalent to the condition that the product of weights of every cycle in the weighted graph $(V, E,w)$ is less than one, where $V$ and $E$ are defined as in \eqref{eq:V_and_E_definition}.
\end{remark}

In a sense, the proof of Theorem \ref{thm:iff_conditions_for_stability_assuming_all_edges_exist} is just piecing together the conclusions asserted by Propositions \ref{propn:crossing_between_sector_boundaries} and \ref{propn:staying_within_sector_boundaries}. Intuitively, every trajectory of the switched system \eqref{intro.1} either intersects a finite number of sector boundaries and then remains in some closed cone $\cl C \setminus \{0\}$ forever after, or it intersects an infinite number of sector boundaries $S_{[i]}$ in forward time; the first case is handled by Proposition \ref{propn:staying_within_sector_boundaries} and the second case is handled by Proposition \ref{propn:crossing_between_sector_boundaries}.

\subsection{Analysis without additional assumptions}\label{subsection:stability_without_assumptions}

We now formulate the final stability result. To drop the assumption that $w([i], [i+1])$ and $w([i+1], [i])$ exist for all $i \in \{1, \ldots, n\}$ from Theorem \ref{thm:iff_conditions_for_stability_assuming_all_edges_exist}, we will use the trick of adding additional subsystems to the original switched system \eqref{intro.1} to obtain an ``auxiliary'' switched system to which we can apply Theorem \ref{thm:iff_conditions_for_stability_assuming_all_edges_exist}. We will then show that stability properties of the original switched system and the auxiliary system are equivalent, which will complete the proof. We only provide an informal overview of the approach here; all the details and formal proofs are delegated to \ref{section:omitted_proofs_2}.

It will be sufficient to append two additional subsystems $\dot{x} = Ax$ determined by matrices $\cR_{\delta^{+}}, \cR_{\delta^{-}}$ of the form
\begin{equation}
    \cR_{\delta^{+}} = \begin{pmatrix} -1 & -\delta \\ \delta & -1 \end{pmatrix}, \ \cR_{\delta^{-}} = \begin{pmatrix} -1 & \delta \\ -\delta & -1 \end{pmatrix}
\end{equation}
where $\delta>0$ is sufficiently small. We will fix $\delta$ during the proof of Theorem \ref{thm:iff_conditions_for_stability} in \ref{section:omitted_proofs_2}. Note that the trajectories of the unswitched linear systems corresponding to the matrices $\cR_{\delta^{+}}$ and $\cR_{\delta^{-}}$ spiral counterclockwise and clockwise about the origin respectively, and as $\delta \to 0$ the rotational part of the dynamics tends to zero and both matrices $\cR_{\delta^{\pm}}$ tend to the negative identity matrix, $-I_{2}$, component-wise. From Remark \ref{rmk:can_exclude_minus_identity} we know that appending subsystems to the switched system \ref{intro.1} whose dynamics are described by $-I_{2}$ does not change the uniform asymptotic stability of the switched system; intuitively, from this observation we would hope that appending subsystems whose dynamics are described by a small perturbation of $-I_{2}$ (for example, $\cR_{\delta^{+}}$ and $\cR_{\delta^{-}}$ for very small $\delta>0$) would also not change uniform asymptotic stability of the switched system. This is shown rigorously in the proof of the theorem.

Additionally, since the trajectories of $\cR_{\delta^{\pm}}$ spiral counterclockwise and clockwise about the origin as previously mentioned, appending these subsystems to \ref{intro.1} would mean that it is always possible to switch into the $\cR_{\delta^{+}}$-subsystem at any ray $S_{[i]}$ and intersect the ray $S_{[i+1]}$ in finite positive time, and likewise it is always possible to switched into the $\cR_{\delta^{-}}$-subsystem at any ray $S_{[i+1]}$ and intersect the ray $S_{[i]}$ in finite positive time. As such, for this new ``auxiliary'' switched system we would have that $w([i],[i+1])$ and $w([i+1],[i])$ exist for all $i \in \{1, \ldots, n\}$ and hence Theorem \ref{thm:iff_conditions_for_stability_assuming_all_edges_exist} may readily be applied.

Having motivated our approach, we may now state the final theorem concerning necessary and sufficient conditions for uniform asymptotic stability of \eqref{intro.1}, which no longer requires the assumption that $w([i], [i+1])$ and $w([i+1], [i])$ exist for all $i \in \{1, \ldots, n\}$.

\begin{theorem}\label{thm:iff_conditions_for_stability}
The switched system \eqref{intro.1} is uniformly asymptotically stable if and only if all of the following conditions hold:
\begin{enumerate}
\item \label{item:1_final_stability_condition} For $i \in \{1, \ldots, n\}$ and all $p \in \cP_{[i]}^{+}, q \in \cP_{[i+1]}^{-}$, if $x \neq 0$ is in the closed cone between $S_{[i]}$ and $S_{[i+1]}$ in the counterclockwise direction then $\Delta_{p,q}(x) > 0$.
\item \label{item:3_final_stability_condition} If $w([i], [i+1])$ exists for all $i \in \{1, \ldots, n\}$ then $\prod_{i=1}^{n}w([i], [i+1])<1$, and if $w([i+1], [i])$ exists for all $i \in \{1, \ldots, n\}$ then $\prod_{i=1}^{n}w([i+1], [i])<1$.
\end{enumerate}
\end{theorem}

The full proof of Theorem \ref{thm:iff_conditions_for_stability}, which is given in \ref{section:omitted_proofs_2}, relies on showing the existence of $\delta'>0$ such that the following hold:
\begin{itemize}
    \item The uniform asymptotic stability of system \eqref{intro.1} is equivalent to the uniform asymptotic stability of the auxiliary switched system (described above) for all $\delta \in (0, \delta']$.
    
    \item The conditions \ref{item:1_final_stability_condition}, \ref{item:3_final_stability_condition} of Theorem \ref{thm:iff_conditions_for_stability} for the system \eqref{intro.1} are equivalent to the conditions \ref{item:1_nonfinal_stability_condition}, \ref{item:3_nonfinal_stability_condition} of Theorem \ref{thm:iff_conditions_for_stability_assuming_all_edges_exist} for the auxiliary switched system for all $\delta \in (0, \delta']$.
\end{itemize}
Then, applying Theorem \ref{thm:iff_conditions_for_stability_assuming_all_edges_exist} to the auxiliary switched system with any choice of $\delta \in (0, \delta']$ will complete the proof. A primary difficulty of this approach is that the set of rays $\cS = \{S_{[1]}, S_{[2]}, \ldots, S_{[n]}\}$ is different for the original switched system \ref{intro.1} and the auxiliary switched system, and in fact the set of rays for the auxiliary system, call it $\tilde{\cS}(\delta)$, depends on $\delta$. As such, to relate the conditions of the two main stability theorems above for the original and auxiliary switched systems, several preliminary results are needed in \ref{section:omitted_proofs_2} to show that the set of rays $\tilde{\cS}(\delta)$ in some sense ``tends to'' the set $\cS$ as $\delta \to 0$.

\section{Consequences of Theorem \ref{thm:iff_conditions_for_stability}}\label{section:consequences}

\subsection{Analysis of bifurcations in \eqref{intro.2}}\label{subsection:bifurcations}

Using Theorem \ref{thm:iff_conditions_for_stability} it is now straightforward to study the codimension-one bifurcations of uniform asymptotic stability in the switched linear systems \eqref{intro.1} we have been considering thus far. For the switched system \eqref{intro.2} with bifurcation parameter $\theta$ the conditions \ref{item:1_final_stability_condition} and \ref{item:3_final_stability_condition} of Theorem \ref{thm:iff_conditions_for_stability} may be used to find the bifurcation points in parameter space, and furthermore to identify the exact mechanisms through which uniform asymptotic stability is lost.

More precisely, suppose the bifurcation parameter $\theta \in [\alpha, \beta] \subseteq \RR$ is such that all $A_{p}(\theta)$ are Hurwitz for $p \in \cP$. The property of being uniformly asymptotically stable for a switched linear system is robust under perturbations of the corresponding matrices; this is a straightforward consequence of the converse Lyapunov theorem for switched linear systems (Theorem 6 in \cite{molchanov1989criteria}) and the fact that there is a strict inequality sign in \eqref{eq:lyapunov_function_condition}. Hence all $\theta \in [\alpha, \beta]$ for which \eqref{intro.1} is uniformly asymptotically stable form an open set, so the set of all $\theta \in [\alpha, \beta]$ such that at least one of the conditions \ref{thm:iff_conditions_for_stability}-\ref{item:1_final_stability_condition}, \ref{thm:iff_conditions_for_stability}-\ref{item:3_final_stability_condition} fails is closed (in the subspace topology).

Thus any transition of the system \eqref{intro.1} from uniformly asymptotically stable behaviour to the existence of a solution which does not converge to the origin must occur in the following way: there must be an interval $(\theta_{0}, \theta_{1}) \subseteq [\alpha, \beta]$ such that \eqref{intro.1} is uniformly asymptotically stable for all $\theta \in (\theta_{0}, \theta_{1})$ and not uniformly asymptotically stable for $\theta = \theta_{0}$ and $\theta = \theta_{1}$. To fix ideas, we consider when uniform asymptotic stability fails at $\theta = \theta_{1}$; then at least one of the conditions \ref{thm:iff_conditions_for_stability}-\ref{item:1_final_stability_condition}, \ref{thm:iff_conditions_for_stability}-\ref{item:3_final_stability_condition} fails at $\theta_{1}$. The next theorem analyzes each of these cases separately.

\begin{theorem} \label{thm:bifurcation_analysis} Consider the switched system \eqref{intro.1}.
\begin{enumerate}
    \item \label{item:1_bifurcation_analysis} Suppose condition \ref{item:1_final_stability_condition} of Theorem \ref{thm:iff_conditions_for_stability} fails to hold at $\theta = \theta_{1}$. Then for the parameter value $\theta_{1}$ there exists a one-dimensional subspace of $\RR^{2}$ at which condition \ref{thm:iff_conditions_for_stability}-\ref{item:1_final_stability_condition} fails with equality. Additionally, if at $\theta_{1}$ we have strict inequality ``$<$'' in condition \ref{thm:iff_conditions_for_stability}-\ref{item:1_final_stability_condition} then there exists a periodic solution of the system \eqref{intro.1}.
    \item \label{item:3_bifurcation_analysis} Suppose condition \ref{item:3_final_stability_condition} of Theorem \ref{thm:iff_conditions_for_stability} fails to hold at $\theta = \theta_{1}$. Then for the parameter value $\theta_{1}$ there exists a periodic solution of the system \eqref{intro.1}.
\end{enumerate}
\end{theorem}

\begin{proof}
For \ref{item:1_bifurcation_analysis}, suppose there exists $i \in \{1, \ldots, n\}$ and $p \in \cP_{[i]}^{+}, q \in \cP_{[i+1]}^{-}$ and some $x$ in the closed cone between $S_{[i]}$ and $S_{[i+1]}$ for which $\Delta_{p,q}(x) \leq 0$. If $\Delta_{p,q}(x) = 0$ for some such $x$ then equality holds for all $ax$, $a \in \RR$, so suppose we have strict inequality $\Delta_{p,q}(x) < 0$. As in the proof of Lemma \ref{lemma:collinear_opposite_directions_implies_not_stable} in \ref{section:omitted_proofs_1}, there exists a solution $x(t)$ of the linear switched system \eqref{intro.1} with initial condition $\xi_{0} \in S_{[i+1]}$ which solves $\dot{x} = A_{q}x$ until the first positive time of intersection with the ray $\{ax:a>0\}$, the point of intersection being $\xi$, and then solves $\dot{x} = A_{p}x$ until the first positive time of intersection with $S_{[i+1]}$ again at the point $\xi_{1}$, and we have $||\xi_{1}|| > ||\xi_{0}||$.

Let $x_{0}(t)$ solve $\dot{x} = A_{q}x$ with $x_{0}(0) = \xi_{0}$, and let $x_{1}$ solve $\dot{x} = A_{p}x$ with $x_{1}(0) = \xi_{1}$. Depending on whether the eigenvalues of $A_{p}$ and $A_{q}$ have nonzero imaginary parts or not, similar to the proof of Proposition \ref{propn:staying_within_sector_boundaries} there exists a continuous strictly decreasing function $\rho:\RR \to \RR$ with $\rho(0)=0$, $\lim_{t \to \infty}\rho(t) = \infty$, and a continuous function $\lambda: \RR^{\geq0} \to \RR^{\geq0}$ such that either $x_{0}(\rho(t)) = \lambda(t) \cdot x_{1}(t)$ or $x_{0}(-t) = \lambda(t) \cdot x_{1}(-\rho(t))$ for all $t \geq 0$. Without loss of generality, suppose the former holds. By construction, $\lambda(0) < 1$. Since $A_{p}, A_{q}$ are both Hurwitz matrices,
\begin{equation*}
    \lim_{t \to \infty}||x_{0}(\rho(t))|| = \lim_{t \to -\infty}||x_{0}(t)|| = \infty
\end{equation*}
and
\begin{equation*}
    \lim_{t \to \infty}||x_{1}(t)|| = 0.
\end{equation*}
By continuity of all functions involved there must be some $\tau>0$ such that $\lambda(\tau)=1$, so that $x_{0}(\rho(\tau)) = x_{1}(\tau) =: \tilde{\xi}$. Therefore the trajectory of the switched system which repeatedly solves $\dot{x} = A_{q}x$ from $\tilde{\xi}$ until the trajectory reaches the point $\xi$ and then solves $\dot{x} = A_{p}x$ from $\xi$ until the trajectory reaches the point $\tilde{\xi}$, gives the desired periodic solution of \eqref{intro.1}.

Denoting by $y(t)$ the trajectory above with $y(0) = \xi$ and denoting by $t_{0}>0$ the minimal time such that $y(t_{0}) = \tilde{\xi}$, note that $\Delta_{p,q}(y(0)) = \Delta_{p,q}(\xi) < 0$ and $\Delta_{p,q}(y(t_{0})) = \Delta_{p,q}(\tilde{\xi}) > 0$ by the converse of Lemma \ref{lemma:unswitched_trajectory_growth_ordering} \ref{item:2ii_unswitched_ordering} applied to some open cone containing $\tilde{\xi}$ (compare with Remark \ref{rmk:contrapositive_of_growth_ordering_theorem} and Remark \ref{rmk:WLOG_shrinking_sector_C}). Then by continuity of $y(t)$ and $\Delta_{p,q}$ there exists some $t \in (0, t_{0})$ for which $\Delta_{p,q}(y(t_{0})) = 0$, hence $\Delta_{p,q}(ay(t_{0})) = 0$ for all $a \in \RR$.

For \ref{item:3_bifurcation_analysis}, we may without loss of generality assume $\prod_{i=1}^{n}w([i], [i+1]) \geq 1$. For arbitrary $\xi \in S_{[1]}$, let $x(t)$ be the solution of \eqref{intro.1} constructed as in the proof of the forward direction of Theorem \ref{thm:iff_conditions_for_stability_assuming_all_edges_exist} with $x(0) = \xi$. Letting $\tau>0$ be the minimal positive time for which $x(\tau) \in S_{[1]}$, by assumption $||x(\tau)|| \geq ||x(0)||$. If we have equality it is clear that $x(t)$ thus defined with repeated switching signal is periodic, so suppose $||x(\tau)|| > ||x(0)||$. Denote by $\gamma$ the Jordan curve formed by taking the union of $x([0, \tau])$ and the line segment between $x(0)$ and $x(\tau)$, and let $B$ be the connected component of $\RR^{2} \setminus \gamma$ containing the origin. Then if $x_{0}(t)$ solves $\dot{x} = A_{p_{[1]}^{+}}x$, since $\lim_{t \to \infty} ||x_{0}(t)|| = 0$, $x_{0}(t) \in B$ for all large enough $t$. Since $x_{0}(0) \notin B$ there must exist a positive time $t_{0}$ for which $x_{0}(t_{0}) \in \partial B$; since the angular velocity of $x_{0}(t)$ is positive for all $t \geq 0$ by Corollary \ref{cor:theta_minimal_speed}, $x_{0}(t)$ cannot cross into $B$ through the line segment between $x(0)$ and $x(\tau)$, so we must have $x_{0}(t_{0}) \in x([0, \tau])$. Existence of a periodic orbit follows as before.
\end{proof}

\begin{example}\label{example:1}
Consider the switched linear system \eqref{intro.2} where $\cP = \{1, 2\}$, $\theta \in [2, 4]$, and
\begin{equation}\label{eq:ex1_matrices}
    A_{1}(\theta) = \begin{pmatrix} -1 & \theta \\ -1 & -1 \end{pmatrix}, \ A_{2}(\theta) = \begin{pmatrix} -1 & -1 \\ \theta & -1 \end{pmatrix}.
\end{equation}
Both $A_{1}(\theta)$ and $A_{2}(\theta)$ are Hurwitz for $\theta \in [2,4]$. This switched system admits the common Lyapunov function $x \mapsto x^{\top}x$ for $\theta \in [2, 3)$, hence is uniformly asymptotically stable. Additionally, for $\theta \in [3, 4]$ the convex combination
\begin{equation}
    \frac{1}{2}A_{1}(\theta) + \frac{1}{2}A_{2}(\theta) = \begin{pmatrix} -1 & (\theta-1)/2 \\ (\theta-1)/2 & -1 \end{pmatrix}
\end{equation}
fails to be Hurwitz, so for these values of $\theta$ the switched system is not uniformly asymptotically stable by Corollary 2.3 in \cite{liberzon2003switching}.

Despite the classical analysis being seemingly straightforward for this example, one still needs to come up with the common Lyapunov function for $\theta \in [2, 3)$ and deduce that uniform asymptotic stability fails due to a certain convex combination not being Hurwitz; a priori, it is not obvious to proceed in this way. Alternatively, one may use the theory above to conclude the presence or lack of uniform asymptotic stability. Although a larger number of computations is required, the process is algorithmic and may easily be automated using a computer. Additionally, using our approach we get an understanding of exactly what happens to the switched system trajectories when the uniform asymptotic stability fails for $\theta \in [3,4]$.

For $\theta \in [2, 4]$, $A_{1}(\theta)$ and $A_{2}(\theta)$ both have eigenvalues with nonzero imaginary parts. Letting $R(-\pi/2) = \begin{pmatrix} 0 & 1 \\ -1 & 0 \end{pmatrix}$, for $x \in \RR^{2}$ we have
\begin{equation*}
    \det(A_{1}(\theta)x \ | \ A_{2}(\theta)x) = x^{\top}(A_{1}(\theta)^{\top}R(-\pi/2)A_{2}(\theta))x = x^{\top} (\theta + 1)\begin{pmatrix} -1 & 0 \\ \theta-1 & -1 \end{pmatrix} x,
\end{equation*}
with this quadratic form being positive definite for $\theta \in [2, 3)$, positive semi-definite for $\theta = 3$, and indefinite for $\theta \in (3, 4]$. Since the quadratic form is positive definite for $\theta \in [2, 3)$, there are no rays $S_{[i]}$ for these values of $\theta$ and by Remark \ref{rmk:no_rays}, the switched system is uniformly asymptotically stable. For $\theta \in [3, 4]$, the zeros of this quadratic form are the subspaces
\begin{equation*}
    \text{span} \left\{ \begin{pmatrix} 1 \\ (\theta-1 + \sqrt{(\theta-1)^{2}-4})/2 \end{pmatrix} \right\}, \ \ \text{span} \left\{ \begin{pmatrix} 1 \\ (\theta-1 - \sqrt{(\theta-1)^{2}-4})/2 \end{pmatrix} \right\}.
\end{equation*}
Therefore for $\theta = 3$,
\begin{equation}\label{eq:ex1_theta_3_Si}
    \cS = \cS_{3} = \left\{ \left\{ a \begin{pmatrix} 1 \\ 1 \end{pmatrix} : a > 0 \right\}, \left\{ a \begin{pmatrix} 1 \\ 1 \end{pmatrix} : a < 0 \right\} \right\} = \{S_{[1]}, S_{[2]}\} = \cS,
\end{equation}
and for $\theta \in (3, 4]$,
\begin{align*}
    \cS_{3} &= \bigg\{ \left\{ a \begin{pmatrix} 1 \\ (\theta-1 - \sqrt{(\theta-1)^{2}-4})/2 \end{pmatrix} : a > 0 \right\}, \ \left\{ a \begin{pmatrix} 1 \\ (\theta-1 + \sqrt{(\theta-1)^{2}-4})/2 \end{pmatrix} : a > 0 \right\}, \\
    & \hspace{20px} \left\{ a \begin{pmatrix} 1 \\ (\theta-1 - \sqrt{(\theta-1)^{2}-4})/2 \end{pmatrix} : a < 0 \right\}, \ \left\{ a \begin{pmatrix} 1 \\ (\theta-1 + \sqrt{(\theta-1)^{2}-4})/2 \end{pmatrix} : a < 0 \right\} \bigg\} \\
    &= \{S_{[1]}, S_{[2]}, S_{[3]}, S_{[4]} \} = \cS.
\end{align*}
A straightforward computation shows that for every $\theta \in [3,4]$ and $x \in S_{[i]}$ we have $\Delta_{1,2}(x) = 0$ and furthermore for $\theta \in (3, 4]$, choosing any $x \neq 0$ in the open cone between $S_{[1]}, S_{[2]}$ or $S_{[3]}, S_{[4]}$ we have $\Delta_{2,1}(x)<0$. By Theorem \ref{thm:iff_conditions_for_stability} we conclude that the switched system \eqref{intro.1} is uniformly asymptotically stable for $\theta \in [2,3)$ and not uniformly asymptotically stable for $\theta \in [3,4]$, and additionally by Theorem \ref{thm:bifurcation_analysis} \ref{item:1_bifurcation_analysis} we know there exists a periodic solution of the switched system for $\theta \in (3,4]$. \qed

\vspace{10px}
\begin{figure}[h]
\centering
\includegraphics[scale=0.5]{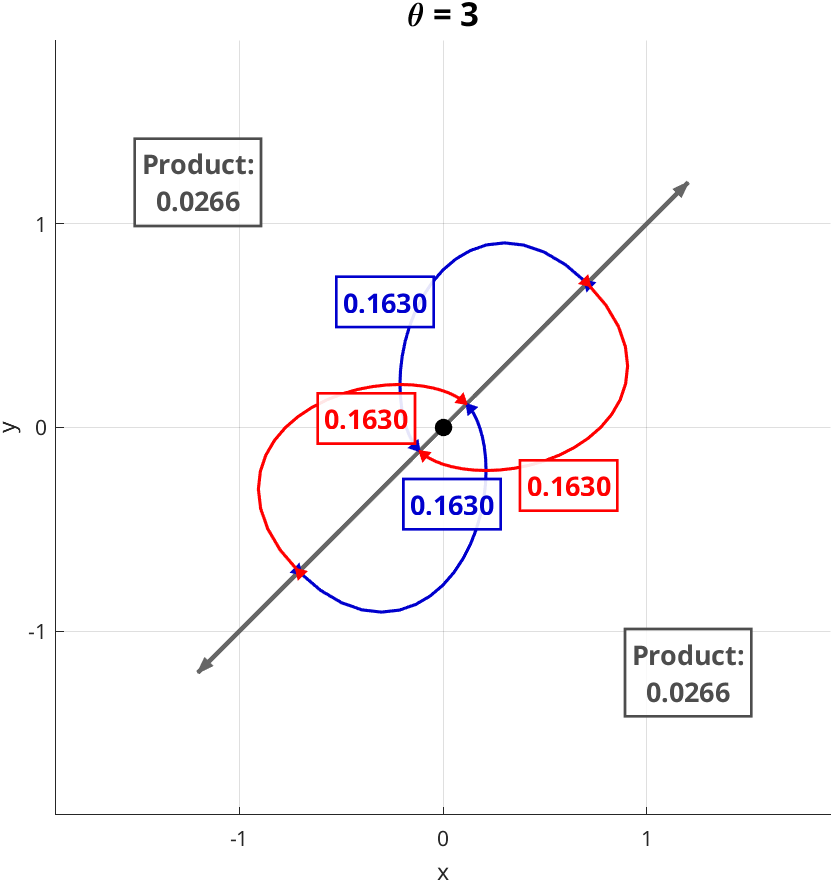}%
\includegraphics[scale=0.5]{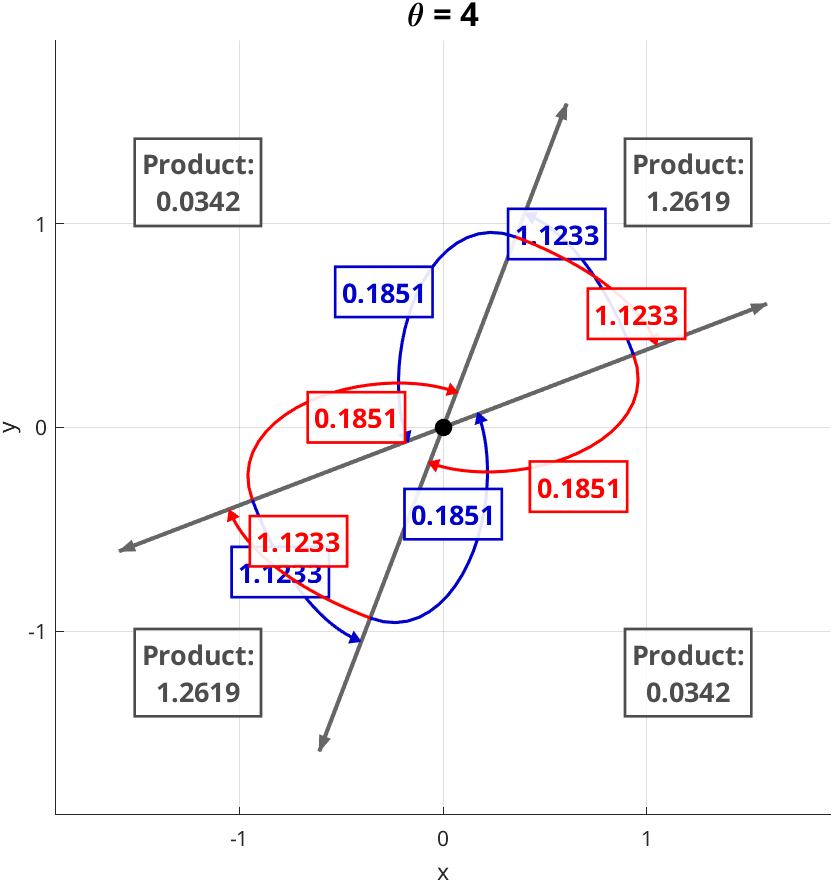}
\caption{$S_{[i]}$ rays and sample trajectories for Example \ref{example:1}, $\theta = 3$ and $\theta=4$. The plots include the product computations $w([i], [i+1]) \cdot w([i+1], [i])$ where applicable; compare with Lemma \ref{lemma:main_thm_condition_redundancy}.}
\label{fig:example_1}
\end{figure}

\end{example}

\begin{example}\label{example:2}
In this example, the bifurcation parameter $\theta$ will be a rotation angle for one of the matrices defining a subsystem, and we will have an additional parameter $\beta$ affecting both the subsystems. Specifically, consider the switched linear system defined as follows. Let
\begin{equation}
    R(\theta) = \begin{pmatrix} \cos(\theta) & -\sin(\theta) \\ \sin(\theta) & \cos(\theta) \end{pmatrix}
\end{equation}
be the $2 \times 2$ rotation matrix by an angle $\theta \in [\alpha, \beta] := [0, 2\pi]$, and for another parameter $\beta>0$ define the matrices
\begin{equation}
\begin{split}
    A_{1}(\theta) &= \begin{pmatrix} -1 & -1 \\ \beta & -1 \end{pmatrix}, \\
    A_{2}(\theta) &= R(\theta)  \begin{pmatrix} -1 & -1 \\ \beta & -1 \end{pmatrix} R(\theta)^{-1} = \begin{pmatrix} -1 + \frac{1-\beta}{2}\sin(2\theta) & -\frac{1+\beta}{2}-\frac{1-\beta}{2}\cos(2\theta) \\ \frac{1+\beta}{2}-\frac{1-\beta}{2}\cos(2\theta) & -1-\frac{1-\beta}{2}\sin(2\theta) \end{pmatrix}.
\end{split}
\end{equation}
By construction, the trajectories of the unswitched system $\dot{x} = A_{2}(\theta)x$ are the same as the trajectories of $\dot{x} = A_{1}(\theta)x$ rotated by an angle of $\theta$. These trajectories are ellipsoidal, and become more eccentric as the parameter $\beta$ is increased. Since $A_{1}(\theta)$ is Hurwitz for all $\beta>0$, so is $A_{2}(\theta)$ for all $\theta \in [0, 2\pi]$. If $\beta=1$ or $\theta \in \{0, \pi, 2\pi\}$, then $A_{1}(\theta)=A_{2}(\theta)$ and the switched system is uniformly asymptotically stable; henceforth for this example we assume $\beta \neq 1$ and $\theta \notin \{0, \pi, 2\pi\}$.

Both matrices have eigenvalues with nonzero imaginary parts, so the only possible rays $S_{[i]}$ come from the set
\begin{equation*}
    \{x : \det(A_{1}(\theta)x \ | \ A_{2}(\theta)x) = 0\} = \{y : y \ | \ A_{2}(\theta)A_{1}(\theta)^{-1}y = 0\}
\end{equation*}
since $A_{1}(\theta)$ is invertible for all parameter values. Thus the rays $S_{[i]}$ are the connected components of the union of the two distinct eigenspaces of $A_{2}(\theta)A_{1}(\theta)^{-1}$ with the origin removed. The union of these eigenspaces is the zero set of the quadratic form $c_{11}(\theta)x_{1}^{2} + c_{12}(\theta)x_{1}x_{2} + c_{22}(\theta)x_{2}^{2}$, where
\begin{align*}
    c_{11}(\theta) &= \sin(\theta)(\beta \cos(\theta)+\sin(\theta)), \\
    c_{12}(\theta) &= \frac{1}{2} \Big( (1+\beta)(1-\cos(2\theta))- 2 \sin(2\theta) \Big), \\
    c_{22}(\theta) &= \sin(\theta)(\cos(\theta)+\sin(\theta)).
\end{align*}
Since all trajectories corresponding to both matrices $A_{1}(\theta), A_{2}(\theta)$ have strictly positive angular velocity, the values $w([i+1], [i])$ do not exist for all $i \in \{1, 2, 3, 4\}$. By Theorem \ref{thm:iff_conditions_for_stability}, the only criteria that needs to be verified to determine uniform asymptotic stability of the switched system is whether or not $\prod_{i=1}^{4}w([i], [i+1])<1$.

\vspace{10px}
\begin{figure}[h]
\centering
\includegraphics[scale=0.32]{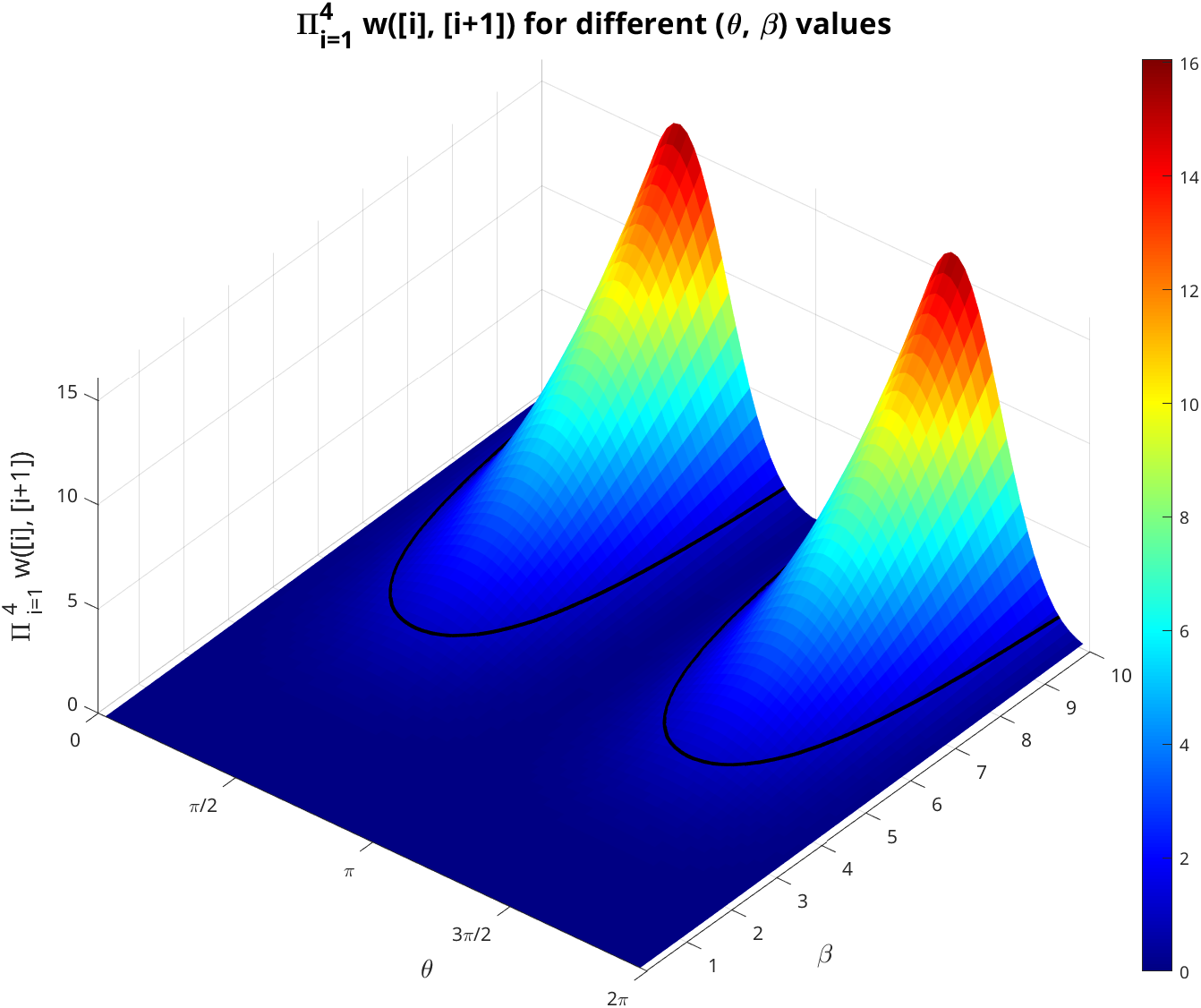}%
\hspace{3px}
\includegraphics[scale=0.32]{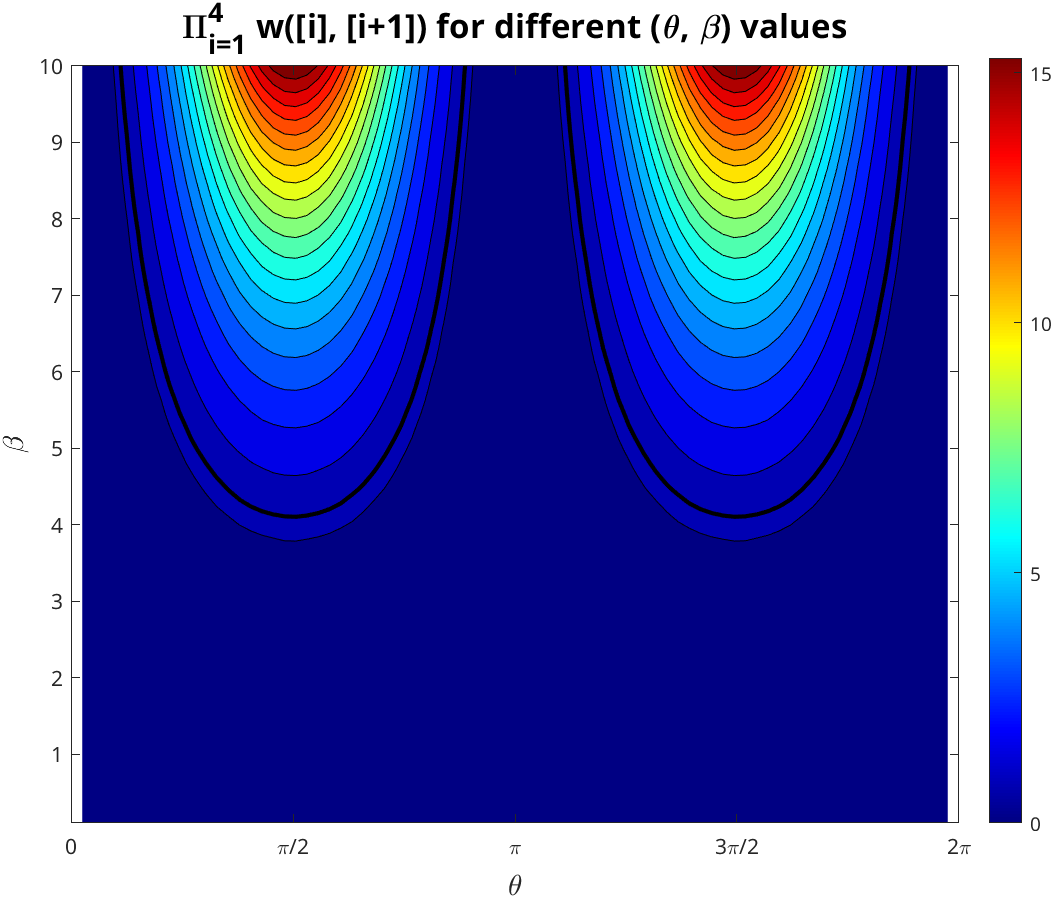}
\caption{Values of $\prod_{i=1}^{4}w([i], [i+1])$ corresponding to various parameter values  $(\theta, \beta) \in [0, 2\pi] \times (0,1) \cup (1, 10]$, surface plot (left) and contour plot (right). Black curves in both plots indicate the level curves where $\prod_{i=1}^{4}w([i], [i+1])=1$.}
\label{fig:example_2}
\end{figure}

Figure \ref{fig:example_2} shows surface and contour plots of the numerical values of $\prod_{i=1}^{4}w([i], [i+1])$ corresponding to the switched system \eqref{intro.1} with parameter values $(\theta, \beta) \in [0, 2\pi] \times (0,1) \cup (1, 10]$. For a fixed value $\beta \in (0, 10] \setminus \{1\}$ the directed weight products are maximized for $\theta \in \{\pi/2, 3\pi/2\}$, and for fixed $\theta$ the weight products are monotone increasing with $\beta$. Roughly speaking, larger $\beta$ values correspond to greater eccentricity of the ellipsoidal trajectories and $\theta$ values of $\pi/2$ and $3\pi/2$ correspond to the maximal angles between the major axes of the ellipsoidal trajectories. \qed
\end{example}

\subsection{Some applications of Theorem \ref{thm:iff_conditions_for_stability} to nonlinear switched systems}\label{section:nonlinear_applications}

Using Theorem \ref{thm:iff_conditions_for_stability}, it is possible to formulate conditions for local and global stability properties of particular nonlinear switched systems which can be approximated sufficiently well by linear switched systems either locally or globally. In particular, for $x(t) \in \RR^{2}$ we will consider switched systems of the form
\begin{equation}\label{eq:nonlinear_switched_system}
    \dot{x}(t) = f_{\sigma(t)}(x(t)), \hspace{20px} \sigma : \RR^{\geq0} \to \cP := \{1, \ldots, m\},
\end{equation}
where $x(0) \in \RR^{2}$, $f_{p}: \RR^{2} \to \RR^{2}$ is a $\cC^{1}$ function, and $\sigma$ is a switching signal as defined in Section \ref{section:intro}.

We will state and prove two such theorems here; one concerns \textit{local} uniform asymptotic stability of a switched system of $\cC^{1}$ functions, and the other is about existence of a basin of attraction for switched systems whose subsystems can be uniformly approximated arbitrarily well by linear systems as $||x|| \to \infty$. Their proofs are similar, both critically relying on Theorem \ref{thm:iff_conditions_for_stability} and the converse Lyapunov theorem (\cite{molchanov1989criteria}, Theorem 6) for switched linear systems. We emphasize that although similar results may be obtained by applying this converse Lyapunov theorem to other necessary and sufficient conditions for uniform asymptotic stability of planar linear switched systems in the literature (see the discussion in Section \ref{section:intro}), the strength of our approach is in obtaining more information on exactly how uniform asymptotic stability is lost. In the first result below, which may be viewed as a variant of Corollary 2.10 in \cite{liberzon2003switching}, this manifests in item \ref{item:2_thm:hartman_grobman} where we may conclude that stability is lost through the existence of periodic solutions of the nonlinear switched system.

\begin{theorem}\label{thm:hartman_grobman}
Suppose that $f_{p}(0)=0$ for every subsystem of the switched system \eqref{eq:nonlinear_switched_system}, let $A_{p} := Df_{p}(0)$, and consider the switched linear system \eqref{intro.1} with subsystems defined by these matrices $A_{p}$.
\begin{enumerate}
    \item \label{item:1_thm:hartman_grobman} If the conditions \ref{item:1_final_stability_condition}, \ref{item:3_final_stability_condition} of Theorem \ref{thm:iff_conditions_for_stability} hold for the switched system \eqref{intro.1} then the switched system \eqref{eq:nonlinear_switched_system} is locally uniformly asymptotically stable.
    \item \label{item:2_thm:hartman_grobman} If any of the conditions \ref{item:1_final_stability_condition}, \ref{item:3_final_stability_condition} of Theorem  \ref{item:1_final_stability_condition} fail with a strict inequality sign for \eqref{intro.1} (for example, $\Delta_{p,q}(x)<0$ for some $x$ or $\prod_{i=1}^{n} w([i], [i+1]) > 1$) then in any neighbourhood of the origin there exists a solution to the switched system \eqref{eq:nonlinear_switched_system} which does not converge to the origin. In fact, there always exists such a non-convergent solution which is periodic.
\end{enumerate}
\end{theorem}

\begin{proof}
For \ref{item:1_thm:hartman_grobman}, suppose conditions \ref{item:1_final_stability_condition} and \ref{item:3_final_stability_condition} of Theorem \ref{thm:iff_conditions_for_stability} hold for the system \eqref{intro.1}. By Theorem \ref{thm:iff_conditions_for_stability}, the switched system is uniformly asymptotically stable. As in the proof of Theorem \ref{thm:iff_conditions_for_stability}, by Theorem 6 of \cite{molchanov1989criteria} it is necessary and sufficient for uniform asymptotic stability of \eqref{intro.1} that there exist $N, p \in \ZZ^{+}$, vectors $l_{i} \in \RR^{2}$, $i \in \{1, \ldots, N\}$, and $\eta>0$ such that \eqref{eq:lyapunov_function_condition} holds for all $q \in \cP$. Let $\rho_{1}, \rho_{2} \in (0,1)$ be such that for all $x,\tilde{x} \in \RR^{2}$ with $||x|| < \rho_{1}, \ ||\tilde{x}||<\rho_{2}$ we have
\begin{equation}
    \sum_{i=1}^{N} 2p(l_{i} \cdot x)^{2p-1}(l_{i} \cdot (A_{q}x + \tilde{x})) < - \eta ||x||^{2p}.
\end{equation}
for all $q \in \cP$. Since each function $f_{q}$ is $\cC^{1}$, we may write
\begin{equation}
    f_{q}(x) = A_{q}x + g_{q}(x)x
\end{equation}
for $q \in \cP$, where $g_{q}(x) \to 0$ as $||x|| \to 0$. Since $\cP$ is finite, we may choose $\delta \in (0, \rho_{1})$ small enough that for all $x$ in the closed ball $B_{\delta}(0)$ centered at zero and all indices $q$, $||g_{q}(x)x|| < \rho_{2}$. Then
\begin{equation}\label{eq:lyapunov_function_condition_nonlinear_modification}
    \sum_{i=1}^{N} 2p(l_{i} \cdot x)^{2p-1}(l_{i} \cdot f_{q}(x)) = \sum_{i=1}^{N} 2p(l_{i} \cdot x)^{2p-1}(l_{i} \cdot (A_{q}x + g_{q}(x)x)) < - \eta ||x||^{2p}
\end{equation}
for all $q \in \cP$ and $x \in B_{\delta}(0)$, hence $v(x) = \sum_{i=1}^{N}(l_{i} \cdot x)^{2p}$ is a common Lyapunov function for the switched system \eqref{eq:nonlinear_switched_system} in a neighbourhood of the origin and so this switched system is locally uniformly asymptotically stable.

Proving \ref{item:2_thm:hartman_grobman} is done with similar constructions as in the proof of Theorem \ref{thm:bifurcation_analysis}; we suppose each condition \ref{item:1_final_stability_condition}, \ref{item:3_final_stability_condition} of Theorem \ref{thm:iff_conditions_for_stability} fails with strict inequality, and then construct a periodic solution to the nonlinear switched system \eqref{eq:nonlinear_switched_system} arbitrarily close to the origin. Since the inequalities in the proof of Theorem \ref{thm:bifurcation_analysis} are strict, there is enough ``wiggle room'' to allow the proof to work in the nonlinear setting of system \eqref{eq:nonlinear_switched_system}, which can be made arbitrarily close to the linear switched system \eqref{intro.1} by working in a sufficiently small neighbourhood of the origin. We demonstrate this strategy below on one case, noting the other case follows analogous modifications from the proof of Theorem \ref{thm:bifurcation_analysis}.

Let $U$ be a neighbourhood of the origin and suppose condition \ref{thm:iff_conditions_for_stability}-\ref{item:3_final_stability_condition} fails with strict inequality; without loss of generality we assume $\prod_{i=1}^{n}(w([i], [i+1])+\epsilon) > 1$ for some $\epsilon>0$. Since $g_{q}(x) \to 0$ as $||x|| \to 0$ for all $q \in \cP$ and $g_{q}(x) \in \cC^{1}$, by continuous dependence of solutions of ordinary differential equations on parameters (\cite{hirsch2013differential}, \S 17.3) and boundedness of each $w([i], [i+1])$ there is some small enough ball $B_{\delta}(0) \subseteq U$ and an initial condition $\xi_{1} \in S_{[1]} \cap B_{\delta}(0)$ for which the following holds:
\begin{itemize}
    \item For the solution $x(t)$ of the switched system \eqref{eq:nonlinear_switched_system} with initial condition $\xi_{1}$ which solves $\dot{x} = f_{p_{[i]}^{+}}(x)$ until intersecting $S_{[i+1]}$ with $x(t_{i}) = \xi_{i+1} \in S_{[i+1]}$ for each $i \in \{1, \ldots, n\}$, we have $x([0, t_{n}]) \subseteq B_{\delta}(0)$.
    \item With $\xi_{i}$, $i \in \{1, \ldots, n\}$ as above, $||\xi_{i+1}||/||\xi_{i}|| < w([i], [i+1]) + \epsilon$ for all $i$.
    \item The solution $x_{0}(t)$ to $\dot{x} = f_{p_{[i]}^{+}}(x)$ with initial condition $x_{0}(0) = \xi_{n+1}$ has positive angular velocity and satisfies $x_{0}(t) \in B_{\delta}(0)$ for all $t \in [0, \tau]$, where $\tau<\infty$ is such that $x_{0}(\tau)$ lies in the connected component of $\RR^{2} \setminus (x([0, t_{n}]) \cup \{(1-s) x(0) + s x(t_{n}) : s \in [0,1]\})$ containing the origin.
\end{itemize}
Then if $x(t)$ solves $\dot{x} = f_{p_{[1]}^{+}}(x)$ for all $t \geq t_{n}$, as in the proof of Theorem \ref{thm:bifurcation_analysis} we must have $x(\tau') \in x([0, t_{n}])$ for some $\tau'>t_{n}$, proving existence of a periodic (hence non-convergent) trajectory contained in $U$.
\end{proof}

\begin{remark}
By the bound \eqref{eq:lyapunov_function_condition_nonlinear_modification} and Remark 2.1 in \cite{liberzon2003switching}, in the proof of \ref{item:1_thm:hartman_grobman} above we actually show that the switched system \eqref{eq:nonlinear_switched_system} is locally uniformly \textit{exponentially} stable, which is stronger than local uniform asymptotic stability.
\end{remark}

\begin{remark}
The proof of Theorem \ref{thm:hartman_grobman} \ref{item:2_thm:hartman_grobman} relied heavily on the techniques developed in Section \ref{subsection:stability_under_assumptions} and \ref{section:omitted_proofs_1} to finely control trajectories of the switched system.
\end{remark}

If any of the conditions \ref{item:1_final_stability_condition}, \ref{item:3_final_stability_condition} of Theorem \ref{thm:iff_conditions_for_stability} fail with an equality sign for \eqref{intro.1} (for example, $\Delta_{p,q}(x)=0$ for some $x \in \RR^{2}$ and we do not have $\Delta_{p,q}(x)<0$ for any $x \in \RR^{2}$) and the rest of these conditions hold, then the test is inconclusive: the switched system \eqref{eq:nonlinear_switched_system} may be either locally uniformly exponentially stable or not.

\begin{remark}
Theorem \ref{thm:hartman_grobman} may be loosely interpreted as an analogue to Lyapunov's indirect method in systems theory (\cite{khalil1992nonlinear}, Theorem 3.5 and its unstable analogue from a reversal of time), in the sense that a set of conditions satisfied with strict inequality in one direction guarantees stability, strict inequality in the other direction guarantees lack of stability, and equality renders the test inconclusive.
\end{remark}

\begin{example}\label{ex:hartman_grobman_converse}
Consider the nonlinear switched system \eqref{eq:nonlinear_switched_system} where $\cP = \{1, 2\}$ and
\begin{equation}\label{eq:ex3-1}
f_{1}(x) = \begin{pmatrix} -1 & 3 \\ -1 & -1 \end{pmatrix} x - ||x||^{2}x, \ f_{2}(x) = \begin{pmatrix} -1 & -1 \\ 3 & -1 \end{pmatrix} x - ||x||^{2}x.
\end{equation}
The linearized switched system at the origin is precisely the switched system \eqref{intro.1} with $\cP = \{1, 2\}$ and matrices $A_{1}(\theta), A_{2}(\theta)$ defined as in \eqref{eq:ex1_matrices} with fixed parameter $\theta=3$. By the work in Example \ref{example:1}, for this system $\cS$ is given by \eqref{eq:ex1_theta_3_Si} and it is straightforward to explicitly compute that
\begin{equation*}
    w([1],[2]) = w([2],[1]) = ||e^{A_{1}(3) \cdot \pi/3}|| = ||e^{A_{2}(3) \cdot \pi/3}|| < 1
\end{equation*}
and $\Delta_{1,2}(x) = 0$ for all $x \in S_{[1]} \cup S_{[2]}$ and $\Delta_{2,1}(x) > 0$ elsewhere. Let $V: \RR^{2} \to \RR^{2}$, $V(x) = \frac{1}{2}x^{\top}x$ be a candidate common Lyapunov function for this switched system. For $i \in \{1,2\}$,
\begin{equation*}
    \dot{V}(x) = \frac{1}{2}x^{\top}(A_{i} + A_{i}^{\top})x - ||x||^{4} = x^{\top} \begin{pmatrix} -1 & 1 \\ 1 & -1\end{pmatrix}x - ||x||^{4} = -(x_{1}-x_{2})^{2} - (x_{1}^{2}+x_{2}^{2})^{2}<0
\end{equation*}
where $x = (x_{1}, x_{2})^{\top} \neq 0$, hence this switched system is uniformly asymptotically stable.

Now consider the same nonlinear switched system but with $f_{1},f_{2}$ defined as
\begin{equation}\label{eq:ex3-2}
f_{1}(x) = \begin{pmatrix} -1 & 3 \\ -1 & -1 \end{pmatrix} x - ||x||^{2}x, \ f_{2}(x) = \begin{pmatrix} -1 & -1 \\ 3 & -1 \end{pmatrix} x + ||x||^{2}x.
\end{equation}
The unswitched system
\begin{equation*}
    \dot{x} = \frac{1}{2}f_{1}(x) + \frac{1}{2}f_{2}(x) = \begin{pmatrix} -1 & 1 \\ 1 & -1 \end{pmatrix}x =: Ax
\end{equation*}
is not asymptotically stable since $A$ has a zero eigenvalue; in particular, any solution of this system with initial condition on the zero eigenspace $\text{span}\{(1,1)^{\top}\}$ is stationary. A straightforward application of the infinite-time relaxation theorem from \cite{ingalls2003infinite} yields a trajectory of the switched system defined by matrices \eqref{eq:ex3-2} which is contained in an arbitrarily small neighbourhood about the origin yet does not converge to the origin. \qed
\end{example}

The second result is useful for establishing the existence of a bounded basin of attraction for a family of switched dynamical systems whose subsystems are approximately linear as $||x|| \to \infty$. In applications, such a result may be used to show that trajectories of the switched system remain bounded in forward time. We omit the proof, which is similar to that of Theorem \ref{thm:hartman_grobman}.

\begin{theorem}\label{thm:basin_of_attraction_existence}
Suppose that for every subsystem $f_{p}$ of the switched system \eqref{eq:nonlinear_switched_system} there exists a real-valued $2 \times 2$ Hurwitz matrix $A_{p}$ with the property that for all $\epsilon>0$, there exists some $N > 0$ such that for all $x \in \RR^{2}$ with $||x|| \geq N$ we have $||f_{p}(x) - A_{p}x|| < \epsilon$. Consider the switched system \eqref{intro.1} with subsystems defined by these matrices $A_{p}$.
\begin{enumerate}
    \item \label{item:1_basin_of_attraction_existence} If both conditions \ref{item:1_final_stability_condition}, \ref{item:3_final_stability_condition} of Theorem \ref{thm:iff_conditions_for_stability} hold for the switched system \eqref{intro.1} then there exists a bounded basin of attraction for the switched system \eqref{eq:nonlinear_switched_system}.
    \item \label{item:2_basin_of_attraction_existence} If any of the conditions \ref{item:1_final_stability_condition}, \ref{item:3_final_stability_condition} of Theorem \ref{thm:iff_conditions_for_stability} fail with a strict inequality sign for \eqref{intro.1} then there exists a solution $x(t)$ of the switched system \eqref{eq:nonlinear_switched_system} such that $||x(t)|| \xrightarrow{t \to \infty} \infty$.
\end{enumerate}
\end{theorem}

\begin{example}
Consider the nonlinear switched system \eqref{eq:nonlinear_switched_system} where $\cP = \{1, 2\}$ and
\begin{equation}\label{eq:ex4-1}
f_{1}(x) = \begin{pmatrix} -1 & 2 \\ -1 & -1 \end{pmatrix} x + \frac{2x}{1+||x||^{2}}, \ f_{2}(x) = \begin{pmatrix} -1 & -1 \\ 2 & -1 \end{pmatrix} x + \frac{2x}{1+||x||^{2}}.
\end{equation}
As $||x|| \to \infty$, each subsystem tends to the subsystems of \eqref{intro.1} with $\cP = \{1, 2\}$ and matrices $A_{1}(\theta), A_{2}(\theta)$ defined as in \eqref{eq:ex1_matrices} with fixed parameter value $\theta=2$. As shown in Example \ref{example:1}, this linear switched system is uniformly asymptotically stable. By Theorem \ref{thm:basin_of_attraction_existence} \ref{item:1_basin_of_attraction_existence}, there exists a bounded basin of attraction for the switched system defined by functions \eqref{eq:ex4-1}. Additionally, using the Poincaré-Bendixson Theorem (\cite{hirsch2013differential}, \S 10.5) and a Lyapunov function for each of $\dot{x} = A_{1}(2)x$, $\dot{x} = A_{2}(2)x$ it is straightforward to show the existence of a limit cycle for each subsystem.

Without using Theorem \ref{thm:basin_of_attraction_existence}, it is not at all clear how to proceed with analyzing the behaviour of trajectories of the switched system \eqref{eq:nonlinear_switched_system} defined by \eqref{eq:ex4-1} since each subsystem (at least locally) demonstrates complicated behaviour. In fact, the authors are not aware of any other results in the literature which deal with analyzing global stability behaviour of switched systems where each subsystem admits a limit cycle. \qed
\end{example}

As is the case for Theorem \ref{thm:hartman_grobman}, if any of the conditions \ref{item:1_final_stability_condition}, \ref{item:3_final_stability_condition} of Theorem \ref{thm:iff_conditions_for_stability} fail with an equality sign for \eqref{intro.1} and the rest of these conditions hold, then the test is inconclusive: there may either exist a bounded basin of attraction for \eqref{eq:nonlinear_switched_system} or there may not. Examples may be produced similarly to the one after the proof of Theorem \ref{thm:hartman_grobman} by perturbing a switched linear system for which one of the conditions fails with equality.

\subsection{Extension of stability results to homogeneous switched systems.} Using a converse Lyapunov theorem from \cite{mancilla2000converse}, it is now straightforward to generalize Theorem \ref{thm:iff_conditions_for_stability} to the case of homogeneous switched systems of degree $\beta>1$.

Let $\cH_{\beta}$ denote the set of all $\cC^{1}$ functions $f = (f_{1}, f_{2}): \RR^{2} \to \RR^{2}$ such that both $f_{1}, f_{2}: \RR^{2} \to \RR$ are homogeneous of degree $\beta>1$, and denote by $\cH_{\beta,s}$ the set of all functions $f \in \cH_{\beta}$ such that the origin is an asymptotically stable equilibrium point for $\dot{x} = f(x)$. Fix a set of indices $\cP = \{1, \ldots, m\}$ and some degree of homogeneity $\beta>1$, and denote by $\Gamma$ the collection of all sets $\{f_{1}, \ldots, f_{m}\} \subseteq \cH_{\beta,s}$ such that the switched system \eqref{eq:nonlinear_switched_system} with this choice of subsystems is uniformly asymptotically stable.

\begin{theorem}\label{thm:homog_in_Gamma_iff_linearized_in_Gamma}
Let $\{f_{1}, \ldots, f_{m}\} \subseteq \cH_{\beta,s}$. Then $\{f_{1}, \ldots, f_{m}\} \in \Gamma$ if and only if $\{A_{1}, \ldots, A_{m}\} \in \Gamma$, where $A_{p} := Df_{p}(0)$ for $p \in \cP$.
\end{theorem}

\begin{proof}
Suppose $\{f_{1}, \ldots, f_{m}\} \in \Gamma$. Theorem 3.1 from \cite{mancilla2000converse} gives the existence of a common Lyapunov function $V(x)$ for $f_{1}, \ldots, f_{m}$, that is, $DV(x) \cdot f_{p}(x) < 0$ for all $x \in \RR^{2} \setminus \{0\}$. Near the origin each function $f_{p}$ is approximated arbitrarily well by $A_{p}$, so in a small enough neighbourhood of the origin $V(x)$ is a common Lyapunov function for $A_{1}, \ldots, A_{m}$. This proves the switched system \eqref{intro.1} with matrices $A_{p}$ is locally uniformly asymptotically stable, hence globally uniformly asymptotically stable. The reverse direction is proved identically.
\end{proof}

The following corollary is an immediate consequence of Theorems \ref{thm:iff_conditions_for_stability} and \ref{thm:homog_in_Gamma_iff_linearized_in_Gamma}, and naturally extends the necessary and sufficient conditions of Theorem \ref{thm:iff_conditions_for_stability} to the case of homogeneous switched systems.

\begin{corollary}\label{cor:iff_conditions_for_stability_homog}
Consider the switched system \eqref{eq:nonlinear_switched_system} where the functions defining the subsystems are \\ $\{f_{1}, \ldots, f_{m}\} \subseteq \cH_{\beta,s}$. Then \eqref{eq:nonlinear_switched_system} is uniformly asymptotically stable if and only if conditions \ref{item:1_final_stability_condition} and \ref{item:3_final_stability_condition} of Theorem \ref{thm:iff_conditions_for_stability} hold for the switched linear system \eqref{intro.1} defined by the matrices $A_{p} := Df_{p}(0)$ for $p \in \{1, \ldots, m\}$.
\end{corollary}

\begin{remark}
Fix $\{f_{1}, \ldots, f_{m}\} \subseteq \cH_{\beta,s}$. By Theorem \ref{thm:hartman_grobman} \ref{item:2_thm:hartman_grobman}, if any of the conditions \ref{item:1_final_stability_condition}, \ref{item:3_final_stability_condition} of Theorem  \ref{item:1_final_stability_condition} fail with a strict inequality sign for the switched system \eqref{intro.1} with subsystems defined by the linearization matrices $A_{p} = Df_{p}(0)$, then there exists a periodic solution of the switched system \eqref{eq:nonlinear_switched_system} with subsystems defined by the functions $f_{p}$. However we cannot conclude anything if at least one of the conditions \ref{item:1_final_stability_condition}, \ref{item:3_final_stability_condition} fails with equality and the remaining conditions hold, so there is no direct analogue to Theorem \ref{thm:bifurcation_analysis} for homogeneous switched systems.
\end{remark}

\section{Conclusion}\label{section:summary}

We developed a new approach for analyzing stability properties of planar switched linear and homogeneous dynamical systems, which is applicable under minimal assumptions and case been explicitly worked out in the case of an arbitrary finite number of subsystems. The necessary and sufficient conditions of Theorem \ref{thm:iff_conditions_for_stability} and Corollary \ref{cor:iff_conditions_for_stability_homog} are straightforward to verify and transparent enough to allow for a successful bifurcation analysis via Theorem \ref{thm:bifurcation_analysis}. Theorems \ref{thm:hartman_grobman} and \ref{thm:basin_of_attraction_existence} formulate useful applications to stability properties of certain nonlinear switched systems. An intriguing future direction of study is to generalize these results to higher dimensions than two; unfortunately, as noted in Remark \ref{rmk:validity_only_in_2_dimensions}, our approach is fundamentally limited to two dimensions. The problem of finding necessary and sufficient stability conditions in three dimensions seems to have been first posed in \cite{boscain2008review} and some recent developments are discussed in \cite{chitour2025dynamics}, \S 5, wherein the authors study extremal trajectories (analogous to worst-case switching trajectories in two dimensions) with respect to a Barabanov norm associated to the matrices defining the subsystems.


\appendix

\setcounter{theorem}{0}
\renewcommand{\thetheorem}{\Alph{section}.\arabic{theorem}}

\setcounter{remark}{0}
\renewcommand{\theremark}{\Alph{section}.\arabic{remark}}

\section{Technical proofs from Section \ref{subsection:stability_under_assumptions}} \label{section:omitted_proofs_1}

Writing $x(t) \in \RR^{2} \setminus \{0\}$ in polar coordinates as $x(t) = (r(t) \cos \theta(t), r(t) \sin \theta(t))$, a straightforward but tedious computation shows that if $x(t)$ solves $\dot{x} = A_{p}x$ for some $p \in \cP$, then
\begin{equation}\label{eq:rotational_derivative}
    \dot{\theta}(t) = \dfrac{\det(x(t) \ | \ A_{p}x(t))}{(r(t))^{2}}.
\end{equation}

\begin{lemma}\label{lemma:det_x_Apx_positive}
Denote by $C$ the open cone bounded between $S_{[i]}$ and $S_{[i+1]}$ in the counterclockwise direction, and let $p \in \cP_{[i]}^{+}$. Then $\det(x \ | \ A_{p}x) > 0$ for all $x \in \cl{C} \setminus \{0\}$.

The analogous result holds for the open cone $C$ bounded between $S_{[i+1]}$ and $S_{[i]}$ in the clockwise direction and $p \in \cP_{[i+1]}^{-}$, where we would then have  $\det(x \ | \ A_{p}x) < 0$ for all $x \in \cl{C} \setminus \{0\}$.
\end{lemma}
\begin{proof}
Suppose first that $\det(y \ | \ A_{p}y) = 0$ for some $y \in \cl{C} \setminus \{0\}$, so that $y$ and $A_{p}y$ are linearly dependent. Then there exists $\lambda \in \RR$ such that $A_{p}y = \lambda y$, that is, $y$ is an eigenvector for $A_{p}$ with eigenvalue $\lambda$. Since $A_{p}$ is Hurwitz we have $\lambda<0$, thus $\dot{y} = A_{p}y = \lambda y$ for some nonzero real-valued $\lambda$. Thus the ray $R = \{\alpha y : \alpha \in \RR^{+}\}$ is invariant under the ODE $\dot{y} = A_{p}y$ in forward and backward time, contradicting the assumption that $p \in \cP_{[i]}^{+}$ since no trajectory starting from $S_{[i]}$ can cross $R$ to intersect $S_{[i+1]}$. Thus we cannot have $\det(y \ | \ A_{p}y) = 0$ for any $y \in \cl{C} \setminus \{0\}$.

Now let $x(t)$ solve $\dot{x} = A_{p}x$ with $x(0) \in \cl{C} \setminus \{0\}$, and write $x(t)$ in polar coordinates as
\begin{equation}\label{eq:x_polar_coordinates}
    x(t) = r(t) \begin{pmatrix} \cos \theta(t) \\ \sin \theta(t) \end{pmatrix}.
\end{equation}
By equation \eqref{eq:rotational_derivative},
\begin{equation}\label{eq:theta_dot_computation}
    \dot{\theta}(t) = \dfrac{ \det \left( r(t) \begin{array}{c|c} \begin{pmatrix} \cos \theta(t) \\ \sin \theta(t) \end{pmatrix} & A_{p} r(t) \begin{pmatrix} \cos \theta(t) \\ \sin \theta(t) \end{pmatrix} \end{array} \right) } {(r(t))^{2}} = \det \left( \begin{array}{c|c} \begin{pmatrix} \cos \theta(t) \\ \sin \theta(t) \end{pmatrix} & A_{p} \begin{pmatrix} \cos \theta(t) \\ \sin \theta(t) \end{pmatrix} \end{array} \right).
\end{equation}
Thus for all $t \in \RR$ such that $y(t) = \begin{pmatrix}\cos \theta(t) \\ \sin \theta(t) \end{pmatrix} \in \cl{C} \setminus \{0\}$, we have
\begin{equation}\label{eq:theta_dot_nonzero}
    \dot{\theta}(t) \neq 0;
\end{equation}
since $\cl{C} \setminus \{0\}$ is a closed cone, we conclude that \eqref{eq:theta_dot_nonzero} holds while $x(t) \in \cl{C} \setminus \{0\}$. By continuity of $\dot{\theta}(t)$, we must have either $\dot{\theta}(t) > 0$ or $\dot{\theta}(t) < 0$. By definition of $p \in \cP_{[i]}^{+}$, since any trajectory $x(t)$ starting at $S_{[i]}$ must cross $C$ before reaching $S_{[i+1]}$, there must exist a point $x(t_{0}) \in \cl{C} \setminus \{0\}$ on this trajectory for which
\begin{equation*}
    \dot{\theta}(t_{0}) = \det \left( \begin{array}{c|c} \begin{pmatrix} \cos \theta(t_{0}) \\ \sin \theta(t_{0}) \end{pmatrix} & A_{p} \begin{pmatrix} \cos \theta(t_{0}) \\ \sin \theta(t_{0}) \end{pmatrix} \end{array} \right) > 0.
\end{equation*}
This forces $\det(y | A_{p}y) > 0$ for all $y \in \cl{C} \setminus \{0\}$, completing the proof.
\end{proof}

\begin{corollary}\label{cor:theta_minimal_speed}
Denote by $C$ the open cone bounded between $S_{[i]}$ and $S_{[i+1]}$ in the counterclockwise direction, and let $p \in \cP_{[i]}^{+}$. Write $x(t)$ solving the initial value problem $\dot{x} = A_{p}x$, $x(0) \in \cl{C} \setminus \{0\}$ in polar coordinates as in \eqref{eq:x_polar_coordinates}, and suppose $J$ is the maximal time interval such that $x(t) \in \cl{C} \setminus \{0\}$ for $t \in J$. Then there exists a constant $c_{p}>0$ for which $\dot{\theta}(t)>c_{p}$ for all $t \in J$.

Analogously, for the open cone $C$ bounded between $S_{[i+1]}$ and $S_{[i]}$ in the clockwise direction and $\theta(t)$ the angular component of $x(t)$ solving $\dot{x} = A_{p}x$, $x(0) \in \cl{C} \setminus \{0\}$ for $p \in \cP_{[i+1]}^{-}$, there exists a constant $c_{p} < 0$ for which $\dot{\theta}(t) < c_{p}$ for the analogous maximal interval $J$.
\end{corollary}

\begin{proof}
By Lemma \ref{lemma:det_x_Apx_positive} and equation \eqref{eq:theta_dot_computation},
\begin{equation}\label{eq:theta_dot_determinant_greater_than_zero}
    \dot{\theta}(t) = \det \left( \begin{array}{c|c} \begin{pmatrix} \cos \theta(t) \\ \sin \theta(t) \end{pmatrix} & A_{p} \begin{pmatrix} \cos \theta(t) \\ \sin \theta(t) \end{pmatrix} \end{array} \right) > 0
\end{equation}
for all $t \in J$. Noting that $x(t_{0}) \in S_{[i]}$ for a maximal time $t_{0} \leq 0$, $x(t_{1}) \in S_{[i+1]}$ for a minimal time $t_{1} \geq 0$, and $x(t) \in \cl{C} \setminus \{0\}$ for $t \in [t_{0}, t_{1}]$, since $\dot{\theta}(t)>0$ we have that $J = [t_{0}, t_{1}]$. The result follows by continuity of $\theta(t)$ and compactness of $J$.
\end{proof}

The following technical lemma ensures that solutions to unswitched subsystems with indices in either $\cP_{[i]}^{+}$ or $\cP_{[i+1]}^{-}$ starting from anywhere between $S_{[i]}$ and $S_{[i+1]}$ intersect appropriate rays between $S_{[i]}$ and $S_{[i+1]}$.

\begin{lemma}\label{lemma:unswitched_unique_times}
Denote by $C$ the open cone bounded between $S_{[i]}$ and $S_{[i+1]}$ in the counterclockwise direction, and let $x(t)$ solve $\dot{x} = A_{p}x$ with $x(0) \in \cl{C} \setminus \{0\}$.
\begin{enumerate}
\item \label{item:1_unswitched_unique_times} Suppose $p \in \cP_{[i]}^{+}$ and $x(0) \notin S_{[i+1]}$. Let $R$ be a ray emitting from the origin between $\{c x(0) : c \in \RR^{+}\}$ and $S_{[i+1]}$ in the counterclockwise order, with $R \neq \{c x(0) : c \in \RR^{+}\}$. Then there exists a unique minimal time $t_{x}>0$ such that $x(t_{x}) \in R$.
\item \label{item:2_unswitched_unique_times} Suppose $p \in \cP_{[i+1]}^{-}$ and $x(0) \notin S_{[i]}$. Let $R$ be a ray emitting from the origin between $S_{[i]}$ and $\{c x(0) : c \in \RR^{+}\}$ in the counterclockwise order, with $R \neq \{c x(0) : c \in \RR^{+}\}$. Then there exists a unique minimal time $t_{x}>0$ such that $x(t_{x}) \in R$.
\end{enumerate}
\end{lemma}

\begin{proof}
We only prove \ref{item:1_unswitched_unique_times}, as the proof for \ref{item:2_unswitched_unique_times} is similar. Write $x(t)$ in polar coordinates as in equation \eqref{eq:x_polar_coordinates}, and let $\theta_{0}$ be the directed angle corresponding to $R$. By Corollary \ref{cor:theta_minimal_speed}, $\dot{\theta}(t)>c_{p}>0$ for all $t \in J = [t_{0}, t_{1}]$ with $x(t_{0}) \in S_{[i]}$ and $x(t_{1}) \in S_{[i+1]}$, so $\theta(t)$ is strictly increasing with $\theta(0) < \theta_{0} < \theta(t_{1})$. Continuity of $\theta(t)$ finishes the proof.
\end{proof}

\begin{remark}\label{rmk:negative_Ap_intersection_times}
Let $p \in \cP_{[i]}^{+}$. If we replace $A_{p}$ with $-A_{p}$ in \eqref{eq:theta_dot_determinant_greater_than_zero}, we obtain $\dot{\theta}(t) < -c_{p} < 0$ for all $t \in J$. Following the same steps as in the proof of Lemma \ref{lemma:unswitched_unique_times} with $x(t)$ solving $\dot{x} = -A_{p}x$, the same conclusion as in the lemma holds. This essentially gives us a time-reversed version of Lemma \ref{lemma:unswitched_unique_times}.
\end{remark}

The following lemma sheds lights as to why the $S_{[i]}$ were chosen to be defined as they were. Recall that by definition of $S_{[i]}$ and $S_{[i+1]}$, for $p, q \in \cP$ the function $\Delta_{p,q}(x)$ defined in \eqref{eq:Delta_pq} is nonzero and does not change sign for $x$ in the open cone bounded between $S_{[i]}$ and $S_{[i+1]}$ (since the one-dimensional subspaces at which $\Delta_{p,q}(x)$ could possibly change sign are contained in $\cS_{2}$). Additionally, by the alternating property of the determinant, $\Delta_{p,q}(x) = -\Delta_{q,p}(x)$.

\begin{lemma}\label{lemma:unswitched_trajectory_growth_ordering}
Denote by $C$ the open cone bounded between $S_{[i]}$ and $S_{[i+1]}$ for some $i \in \{1, \ldots, n\}$ in the counterclockwise direction. Let $p,q \in \cP$ be distinct, and suppose $x(t), \ y(t)$ solve $\dot{x} = A_{p}x, \ \dot{y} = A_{q}y$ respectively with $x(0) = y(0) \in \cl{C} \setminus \{0\}$.
\begin{enumerate}
\item \label{item:1_unswitched_ordering} Suppose $p, q \in \cP_{[i]}^{+}$ and $x(0) = y(0)  \notin S_{[i+1]}$. Let $R$ be a ray emitting from the origin between $\{c x(0) : c \in \RR^{+}\}$ and $S_{[i+1]}$ in the counterclockwise order with $R \neq \{c x(0) : c \in \RR^{+}\}$, and denote by $t_{x}, t_{y}>0$ the minimal positive times given by Lemma \ref{lemma:unswitched_unique_times} at which $x(t_{x}), y(t_{y}) \in R$.
\begin{enumerate}[label=(\roman*)]
    \item \label{item:1i_unswitched_ordering} If $\Delta_{p,q}$ is positive on $C$, then $||x(t_{x})|| > ||y(t_{y})||$.
    \item \label{item:1ii_unswitched_ordering} If $\Delta_{p,q}$ is negative on $C$, then $||x(t_{x})|| < ||y(t_{y})||$.
\end{enumerate}
\item \label{item:2_unswitched_ordering} Suppose $p \in \cP_{[i]}^{+}, q \in \cP_{[i+1]}^{-}$ and $x(0) = y(0)  \notin S_{[i+1]}$. Let $R$ be a ray emitting from the origin between $\{c x(0) : c \in \RR^{+}\}$ and $S_{[i+1]}$ in the counterclockwise order with $R \neq \{c x(0) : c \in \RR^{+}\}$. Let $t_{x}, t_{y}>0$ be minimal such that $x(t_{x}), y(-t_{y}) \in R$.
\begin{enumerate}[label=(\roman*)]
    \item \label{item:2i_unswitched_ordering} If $\Delta_{p,q}$ is positive on $C$, then $||x(t_{x})|| < ||y(-t_{y})||$.
    \item \label{item:2ii_unswitched_ordering} If $\Delta_{p,q}$ is negative on $C$, then $||x(t_{x})|| > ||y(-t_{y})||$.
\end{enumerate}
\item \label{item:3_unswitched_ordering} Suppose $p \in \cP_{[i+1]}^{-}, q \in \cP_{[i]}^{+}$ and $x(0) = y(0)  \notin S_{[i]}$. Let $R$ be a ray emitting from the origin between $S_{[i]}$ and $\{c x(0) : c \in \RR^{+}\}$ in the counterclockwise order with $R \neq \{c x(0) : c \in \RR^{+}\}$. Let $t_{x}, t_{y}>0$ be minimal such that $x(t_{x}), y(-t_{y}) \in R$.
\begin{enumerate}[label=(\roman*)]
    \item \label{item:3i_unswitched_ordering} If $\Delta_{p,q}$ is positive on $C$, then $||x(t_{x})|| > ||y(-t_{y})||$.
    \item \label{item:3ii_unswitched_ordering} If $\Delta_{p,q}$ is negative on $C$, then $||x(t_{x})|| < ||y(-t_{y})||$.
\end{enumerate}
\item \label{item:4_unswitched_ordering} Suppose $p, q \in \cP_{[i+1]}^{-}$ and $x(0) = y(0)  \notin S_{[i]}$. Let $R$ be a ray emitting from the origin between $S_{[i]}$ and $\{c x(0) : c \in \RR^{+}\}$ in the counterclockwise direction with $R \neq \{c x(0) : c \in \RR^{+}\}$, and denote by $t_{x}, t_{y}>0$ the minimal times at which $x(t_{x}), y(t_{y}) \in R$.
\begin{enumerate}[label=(\roman*)]
    \item \label{item:4i_unswitched_ordering} If $\Delta_{p,q}$ is positive on $C$, then $||x(t_{x})|| < ||y(t_{y})||$.
    \item \label{item:4ii_unswitched_ordering} If $\Delta_{p,q}$ is negative on $C$, then $||x(t_{x})|| > ||y(t_{y})||$.
\end{enumerate}
\end{enumerate}
\end{lemma}

\begin{proof}
We will prove \ref{item:1i_unswitched_ordering} and  \ref{item:2i_unswitched_ordering}, as the proofs for  \ref{item:1ii_unswitched_ordering}, \ref{item:2ii_unswitched_ordering}, \ref{item:3_unswitched_ordering}, and \ref{item:4_unswitched_ordering} are similar.

For \ref{item:1_unswitched_ordering}, we first show that there exist positive functions $\rho: [0, t_{x}] \to [0, t_{y}]$ and $\lambda: C \to \RR^{+}$ such that $x(t) = \lambda(t) y(\rho(t))$ for $t \in [0, t_{x}]$. Let $\theta(t), \varphi(t)$ be the angular polar coordinates of $x(t), y(t)$ respectively. By Corollary \ref{cor:theta_minimal_speed}, $\varphi(t)$ is strictly increasing hence invertible. Now, consider the function $y(\varphi^{-1} \circ \theta(t)) =: y(\rho(t))$; the angular polar coordinate of $y(\rho(t))$ is $\varphi(\rho(t)) = \varphi \circ \varphi^{-1} \circ \theta(t) = \theta(t)$. Thus $x(t)$ and $y(\rho(t))$ are scalar multiples of each other for $t \in [0, t_{x}]$, so there exists a positive function $\lambda$ such that $x(t) = \lambda(t) y(\rho(t))$ for $t \in [0, t_{x}]$ as claimed.

From the relations $\dot{x} = A_{p}x, \ \dot{y} = A_{q}y$ and $x(t) =  \lambda(t) y(\rho(t))$, we compute
\begin{align*}
    \frac{d}{dt} \lambda(t)y(\rho(t)) = \lambda(t) A_{p}y(\rho(t))
    \implies &\dot{\lambda}(t) y(\rho(t)) + \lambda(t) A_{q}y(\rho(t)) \dot{\rho}(t) = \lambda(t) A_{p} y(\rho(t)) \\
    \implies & A_{q}y(\rho(t)) = \dfrac{1}{\dot{\rho}(t)} \left( A_{p} - \dfrac{\dot{\lambda}(t)}{\lambda(t)}I \right) y(\rho(t)).
\end{align*}
Now, by the assumption that $\Delta_{p,q}(x) = \det(A_{p}x | A_{q}x)>0$ for $x \in C$, for all $t \in (0, t_{x})$ we have
\begin{align*}
    \det ( A_{p}y(\rho(t)) \ | \ A_{q}y(\rho(t)) ) &= \det \left(A_{p}y(\rho(t)) \ \middle| \ \frac{1}{\dot{\rho}(t)} \left( A_{p} - \frac{\dot{\lambda}(t)}{\lambda(t)}I \right) y(\rho(t)) \right) \\
    &= \frac{1}{\dot{\rho}(t)} \left( \det ( A_{p}y(\rho(t)) \ | \ A_{p}y(\rho(t))) - \det \left(A_{p}y(\rho(t)) \ \middle| \ \frac{\dot{\lambda}(t)}{\lambda(t)}y(\rho(t)) \right) \right) \\
    &= \frac{1}{\dot{\rho}(t)} \frac{\dot{\lambda}(t)}{\lambda(t)} \det(y(\rho(t)) \ | \ A_{q}y(\rho(t)))
\end{align*}
with $\dot{\rho}(t), \lambda(t)>0$, $\det(y(\rho(t)) \ | \ A_{q}y(\rho(t)))>0$ by Lemma \ref{lemma:det_x_Apx_positive}, and $\det(A_{p}y(\rho(t)) \ | \ A_{q}y(\rho(t)))>0$. Thus, we conclude that $\dot{\lambda}(t)>0$. Since $x(0) = y(0) = y(\rho(0))$, we have $\lambda(0) = 1$, hence $\dot{\lambda}(t)>0$ for $t \in (0, t_{x})$ yields
\begin{equation*}
    x(t) = \lambda(t)y(\rho(t)) \implies ||x(t)|| = \lambda(t)||y(\rho(t))|| > ||y(\rho(t))||
\end{equation*}
for all $t \in (0, t_{x}]$, completing the proof of \ref{item:1i_unswitched_ordering}.

For \ref{item:2i_unswitched_ordering}, again letting $\theta(t), \varphi(t)$ be the angular polar coordinates of $x(t), y(t)$ respectively, from Corollary \ref{cor:theta_minimal_speed} and Remark \ref{rmk:negative_Ap_intersection_times} we get that $\theta(t)$ is strictly increasing and $\varphi(t)$ is strictly decreasing, hence $\rho: [0, t_{x}] \to [-t_{y}, 0]$, $\rho(t) := \varphi^{-1} \circ \theta(t)$ is strictly decreasing hence invertible. Similar to before, the definition of $\rho$ yields the equality $x(t) = \lambda(t) y(\rho(t))$, $t \in [0, t_{x}]$ for some positive scalar function $\lambda$.

As in the case of \ref{item:1i_unswitched_ordering}, we obtain the same equality
\begin{equation}\label{eq:det_rho_lambda_equality}
    \det(A_{p}y(\rho(t)) \ | \ A_{q}y(\rho(t))) = \frac{1}{\dot{\rho}(t)} \frac{\dot{\lambda}(t)}{\lambda(t)} \det(y(\rho(t)) \ | \ A_{q}y(\rho(t))),
\end{equation}
where this time we have $\dot{\rho}(t)<0$, $\lambda(t)>0$, $\det(y(\rho(t)) \ | \ A_{q}y(\rho(t)))>0$, and \\ $\det(A_{p}y(\rho(t)) \ | \ A_{q}y(\rho(t)))>0$, from which we conclude $\dot{\lambda}(t)<0$. Again since $x(0) = y(0) = y(\rho(0))$, we have $\lambda(0)=1$, and so $\dot{\lambda}(t)<0$ for $t \in (0, t_{x})$ gives us that for $t \in (0, t_{x}]$,
\begin{equation*}
    ||x(t)|| = \lambda(t)||y(\rho(t))|| < ||y(\rho(t))||
\end{equation*}
which for $t=t_{x}$ yields the desired inequality $||x(t_{x})|| < ||y(\rho(t_{x}))|| = ||y(-t_{y})||$ for \ref{item:2i_unswitched_ordering}.
\end{proof}

\begin{remark} \label{rmk:contrapositive_of_growth_ordering_theorem}
We will frequently make use of the contrapositive of Lemma \ref{lemma:unswitched_trajectory_growth_ordering}. For example, for \ref{item:1i_unswitched_ordering} the formulation of the contrapositive we will use is the following: Suppose $p, q \in \cP_{[i]}^{+}$, $x(0) = y(0) \notin S_{[i+1]}$, and there exists a ray $R \neq \{cx(0) : c \in \RR^{+}\}$ emitting from the origin between $\{cx(0) : c \in \RR^{+}\}$ and $S_{[i+1]}$ in the counterclockwise order containing $x(t_{x}), y(t_{y})$ with $||x(t_{x})||<||y(t_{y})||$, where $t_{x}, t_{y}$ are given by Lemma \ref{lemma:unswitched_unique_times}. Then $\Delta_{p,q}$ is negative (so that $\Delta_{q,p}$ is positive) on $C$.
\end{remark}

\begin{remark}\label{rmk:WLOG_shrinking_sector_C}
The proofs of Lemma \ref{lemma:det_x_Apx_positive}, Corollary \ref{cor:theta_minimal_speed}, Lemma \ref{lemma:unswitched_unique_times}, and Lemma \ref{lemma:unswitched_trajectory_growth_ordering} are almost identical if we make the following modifications to their statements:
\begin{itemize}
\item Replace the open cone $C$ bounded between $S_{[i]}$ and $S_{[i+1]}$ with an open cone $D \subseteq C$ bounded between two rays $\tilde{R}_{1}$ and $\tilde{R}_{2}$ such that the directed angle between $\tilde{R}_{1}, \tilde{R}_{2}$ is positive.
\item Replace the set $\cP_{[i]}^{+}$ of indices in $\cP$ whose corresponding forward trajectories starting on $S_{[i]}$ intersect $S_{[i+1]}$ in forward time, with a set of indices in $\cP$ whose forward trajectories starting on $\tilde{R}_{1}$ intersect $\tilde{R}_{2}$ in forward time; and likewise with $\cP_{[i+1]}^{-}$.
\end{itemize}
\end{remark}

Using the above work, we may already conclude an important intermediate result about uniform asymptotic stability of \eqref{intro.1} which will be useful in the proof of Theorem \ref{thm:iff_conditions_for_stability_assuming_all_edges_exist}.

\begin{lemma}\label{lemma:collinear_opposite_directions_implies_not_stable}
Suppose there exists some $i \in \{1, \ldots, n\}$ and $p \in \cP_{[i]}^{+}, q \in \cP_{[i+1]}^{-}$ such that the following holds. If we denote by $C$ the open cone bounded between $S_{[i]}$ and $S_{[i+1]}$ in the counterclockwise direction, then there exists $x \in \cl{C} \setminus \{0\}$ such that $\Delta_{p,q}(x) \leq 0$. Then the switched system \eqref{intro.1} is not uniformly asymptotically stable.
\end{lemma}
\begin{proof}
First suppose there exists $x \in C$ such that $\Delta_{p,q}(x)=0$. Then there exists some $c \in \RR$ for which $A_{p}x+cA_{q}x = 0$. For a contradiction suppose $c<0$. Then by Lemma \ref{lemma:det_x_Apx_positive} we have $\det(x \ | \ A_{p}x)>0$ and $\det(x \ | \ A_{q}x)<0$, yet
\begin{equation*}
    0<\det(x \ | \ A_{p}x) = \det(x \ | \ -cA_{q}x) = -c \det(x \ | \ A_{q}x)
\end{equation*}
with $-c>0$. Hence $\det(x \ | \ A_{q}x)>0$, a contradiction; we must thus have $c>0$. From this and the equality
\begin{equation*}
    \frac{1}{1+c}A_{p}x + \frac{c}{1+c}A_{q}x=0,
\end{equation*}
we conclude that the above convex combination has a zero eigenvalue, implying that solutions to the system
\begin{equation*}
    \dot{x} = \left( \frac{1}{1+c}A_{p}+\frac{c}{1+c}A_{q} \right)x
\end{equation*}
with initial condition on the corresponding eigenspace do not tend to the origin in forward time. Thus by the stability equivalence of \cite{angeli1999note} discussed in Section \ref{section:intro} and by Corollary 2.3 of \cite{liberzon2003switching}, the switched system cannot be uniformly asymptotically stable.

Now suppose $\Delta_{p,q}(x)<0$ for some $x \in C$. Then $x \notin S_{[i+1]}$ and $\Delta_{p,q}$ is negative on $C$ (since $\Delta_{p,q}$ cannot change sign on $C$ by construction; see the discussion after Remark \ref{rmk:negative_Ap_intersection_times}). Consider the solution of \eqref{intro.1} with initial condition $\xi_{0} \in S_{[i+1]}$ which solves $\dot{x} = A_{q}x$ until the first positive time of intersection with the ray $\{ax:a>0\}$, and then solves $\dot{x} = A_{p}x$ until the first positive time of intersection with $S_{[i+1]}$ again at the point $\xi_{1}$; suppose the solution repeats this behaviour to yield subsequent points of intersection $\xi_{2}, \xi_{3}, \ldots$ with $S_{[i+1]}$. Then by Lemma \ref{lemma:unswitched_trajectory_growth_ordering} \ref{item:2ii_unswitched_ordering} we have $||\xi_{0}|| < ||\xi_{1}|| < ||\xi_{2}|| < \cdots$, proving that this solution to the switched system does not converge to the origin in forward time. As in the previous case, \eqref{intro.1} cannot be uniformly asymptotically stable.
\end{proof}

We now turn our attention to indices in $\cP \setminus (\cP_{[i]}^{+} \cup (\cP_{[i+1]}^{-})$.

\begin{lemma}\label{lemma:controlling_noncrossing_trajectories}
Denote by $C$ the open cone bounded between $S_{[i]}$ and $S_{[i+1]}$ for some $i \in \{1, \ldots, n\}$ in the counterclockwise direction. Let $p \in \cP$ and $q \in \cP \setminus (\cP_{[i]}^{+} \cup \cP_{[i+1]}^{-})$, and suppose $x(t), \ y(t)$ solve $\dot{x} = A_{p}x, \ \dot{y} = A_{q}y$ respectively with $x(0) = y(0) \in \cl{C} \setminus \{0\}$.
\begin{enumerate}
\item \label{item:1_noncrossing_trajectories} Suppose $p \in \cP_{[i]}^{+}$ and $x(0) = y(0) \notin S_{[i+1]}$. Let $R$ be a ray emitting from the origin strictly between $\{c x(0) : c \in \RR^{+}\}$ and $S_{[i+1]}$ in the counterclockwise order, and denote by $t_{x}>0$ the minimal positive time given by Lemma \ref{lemma:unswitched_unique_times} at which $x(t_{x}) \in R$. If there exists some $t_{y}>0$ for which $y(t_{y}) \in R$, then $||x(t_{x})|| > ||y(t_{y})||$.
\item \label{item:2_noncrossing_trajectories} Suppose $p \in \cP_{[i+1]}^{-}$ and $x(0) = y(0) \notin S_{[i]}$. Let $R$ be a ray emitting from the origin strictly between $S_{[i]}$ and $\{c x(0) : c \in \RR^{+}\}$ in the counterclockwise order, and denote by $t_{x}>0$ the minimal positive time given by Lemma \ref{lemma:unswitched_unique_times} at which $x(t_{x}) \in R$. If there exists some $t_{y}>0$ for which $y(t_{y}) \in R$, then $||x(t_{x})|| > ||y(t_{y})||$.
\end{enumerate}
\end{lemma}
\begin{proof}
We only prove \ref{item:1_noncrossing_trajectories}, as the proof of \ref{item:2_noncrossing_trajectories} is almost identical.

Suppose $t_{x}, t_{y} > 0$ are minimal such that $x(t_{x}), y(t_{y}) \in R$. Then there exists some $\lambda > 0$ such that $x(t_{x}) = \lambda y(t_{y})$. Since $q \notin \cP_{[i]}^{+} \cup \cP_{[i+1]}^{-}$ and $A_{q}$ is Hurwitz, we must have that $\lim_{t \to \infty}y(t) = 0$ with $y(t) \in C$ for all $t>0$. Let $\theta(t), \varphi(t)$ be the angular polar coordinates of $x(t), y(t)$ respectively; then $\theta(t_{x}) = \varphi(t_{y})$.

We claim $\varphi(t)$ is strictly increasing. Since the ray $R$ is strictly between $\{c x(0) : c \in \RR^{+}\}$ and $S_{[i+1]}$, we have $\varphi(0) < \varphi(t_{y})$, so the Mean Value Theorem gives us $\dot{\varphi}(t_{0})>0$ for some $t_{0} \in (0, t_{y})$. For a contradiction, suppose there exists some $t_{1}>0$ such that $\dot{\varphi}(t_{1}) = 0$. By the same computations as in equations \eqref{eq:rotational_derivative} and \eqref{eq:theta_dot_computation} except with the index not in $\cP_{[i]}^{+}$ or $\cP_{[i+1]}^{-}$, from $\dot{\varphi}(t_{1}) = 0$ we get that $(\cos \varphi(t_{1}), \sin \varphi(t_{1}))^{\top}$ is an eigenvector for $A_{q}$, thus the subspace spanned by this eigenvector is invariant under the differential equation $\dot{y} = A_{q}y$; this implies $\dot{\varphi}(t) = 0$ for all $t \in \RR$, contradicting $\dot{\varphi}(t_{0})>0$.

By Lemma \ref{lemma:unswitched_unique_times}, there exists a unique minimal time $\tau_{x}>0$ such that $x(\tau_{x}) \in S_{[i+1]}$. Let $x_{\min} = \min_{t \in [0, \tau_{x}]}||x(t)||>0$. Noting that $\varphi(t)$ is strictly increasing and bounded from above since $y(t)$ does not intersect $S_{[i+1]}$ in forward time, it has a unique limit point $\lim_{t \to \infty} \varphi(t) =: \varphi_{\infty}$ with $\varphi(t_{y}) < \varphi_{\infty}$. Furthermore, since $\lim_{t \to \infty}||y(t)|| = 0$, there exists some $t_{y,\min}>t_{y}$ such that $||y(t_{y,\min})|| < x_{\min}$. Since $\theta(t)$ is increasing and $x(t)$ intersects $S_{[i+1]}$ in forward time, so that $\theta(t) > \varphi_{\infty}$ for large enough $t$, there exists some minimal $t_{x, \min}>0$ such that $x(t_{x,\min}), y(t_{y,\min})$ are positive scalar multiples of each other. By construction of $t_{y,\min}$, we further have
\begin{equation*}
    ||y(t_{y,\min})|| < x_{\min} < ||x(t_{x,\min})||.
\end{equation*}

Now consider Remark \ref{rmk:WLOG_shrinking_sector_C} applied to the open cone bounded between $\tilde{R}_{1} =  \{cy(0) : c \in \RR^{+}\}$ and $\tilde{R}_{2} = \{cy(t_{y,\min}) : c \in \RR^{+}\}$, noting that the forward trajectory corresponding to $q \in \cP \setminus (\cP_{[i]}^{+} \cup \cP_{[i+1]}^{-})$ starting from $\tilde{R}_{1}$ intersects $\tilde{R}_{2}$ in forward time by construction. By the contrapositive of the appropriately modified Lemma \ref{lemma:unswitched_trajectory_growth_ordering} \ref{item:1_unswitched_ordering}, since $||y(t_{y,\min})|| < ||x(t_{x,\min})||$ with $x(t_{x,\min}), y(t_{y,\min})$ laying on the same ray we must have that $\Delta_{p,q}$ is positive. Applying the modified Lemma \ref{lemma:unswitched_trajectory_growth_ordering} \ref{item:1_unswitched_ordering} again but this time to $x(t_{x})$ and $y(t_{y})$, we obtain $||x(t_{x})|| > ||y(t_{y})||$ as needed.
\end{proof}

For ease of notation we will denote $\cQ_{[i]} = \cP \setminus (\cP_{[i]}^{+} \cup \cP_{[i+1]}^{-})$. Note that if for any $q \in \cQ_{[i]}$ the matrix $A_{q}$ had eigenvalues with nonzero imaginary part, then either the forward trajectories starting at $S_{[i]}$ intersect $S_{[i+1]}$ in forward time or vice versa (depending on whether the solutions spiral towards the origin in a counterclockwise or clockwise manner), hence contradicting membership of $q$ in $\cQ_{[i]}$. Therefore every matrix $A_{q}$ for $q \in \cQ_{[i]}$ must have only real eigenvalues.

\begin{remark}\label{rmk:Q_i_trichotomy}
Let $y(t)$ be a solution to $\dot{y} = A_{q}y$ with $q \in \cQ_{[i]}$ and $y(0) \in \cl{C} \setminus \{0\}$, and denote by $\varphi(t)$ the angular polar coordinate of $y(t)$. In the proof of Lemma \ref{lemma:controlling_noncrossing_trajectories}, we showed that $\varphi(t)$ is either strictly increasing, strictly decreasing, or zero for all $t \geq 0$. However, this only applies to that particular solution $y(t)$ with initial condition $y(0)$; different initial conditions $y(0)$ in $\cl{C} \setminus \{0\}$ may yield different behaviour for $\varphi(t)$.
\end{remark}

Items \ref{item:1_unswitched_ordering} and \ref{item:4_unswitched_ordering} from Lemma \ref{lemma:unswitched_trajectory_growth_ordering} induce a strict total ordering on the indices in $\cP_{[i]}^{+}$ and in $\cP_{[i+1]}^{-}$. In particular, for $p, q \in \cP_{[i]}^{+}$ we may define $p \prec_{+, [i]} q$ if $\Delta_{p,q}$ is negative on $C$, and for $p, q \in \cP_{[i+1]}^{-}$ we may define $p \prec_{-, [i]} q$ if $\Delta_{p,q}$ is positive on $C$. The ordering is strict because we cannot have $\Delta_{p,q}(x) = 0$ for any $x \in C$ by definition of $S_{[i]}$ and $S_{[i+1]}$. With these definitions we have that, for example, $p, q \in \cP_{[i]}^{+}$ with $p \prec_{+, [i]} q$ implies that if $x(t), y(t)$ are as in the statement of Lemma \ref{lemma:unswitched_trajectory_growth_ordering} with $x(0) = y(0) \notin S_{[i+1]}$ and $R, t_{x}, t_{y}$ are as in \ref{item:1_unswitched_ordering}, then $||x(t_{x})|| < ||y(t_{y})||$, and the analogous result related to \ref{item:4_unswitched_ordering} holds for $p \prec_{-, [i]} q$ with $p, q \in \cP_{[i+1]}^{-}$.

Since $\cP_{[i]}^{+}, \cP_{[i+1]}^{-}$ are finite, they admit unique maximal elements in relation to $\prec_{+, [i]}, \prec_{-, [i]}$; call them $p_{[i]}^{+}, p_{[i+1]}^{-}$ respectively. Let $q \in \cP$, and suppose $x(t), y(t)$ solve $\dot{x} = A_{p_{[i]}^{+}}x, \dot{y} = A_{q}y$ respectively. Then by Lemmas \ref{lemma:unswitched_trajectory_growth_ordering} and \ref{lemma:controlling_noncrossing_trajectories}, for any $t_{x},t_{y}>0$ such that $x(t_{x})$ and $y(t_{y})$ lie on the same ray, we have $||x(t_{x})||>||y(t_{y})||$. In fact, this reasoning can be extended to the switched system \eqref{intro.1}.

\begin{lemma} \label{lemma:switched_monotone_direction_bound}
Denote by $C$ the open cone bounded between $S_{[i]}$ and $S_{[i+1]}$ for some $i \in \{1, \ldots, n\}$ in the counterclockwise direction. Consider a solution $x(t)$ to the switched system \eqref{intro.1} with initial condition $x(0) \in \cl{C} \setminus \{0\}$ and where $\sigma : \RR^{\geq 0} \to \cP' \subseteq \cP$ is an arbitrary switching signal. Let $y(t)$ solve the unswitched system $\dot{y} = A_{p}y$ with $p \in \cP$ and initial condition $y(0) = x(0)$. Suppose one of the follow possibilities holds.
\begin{enumerate}
\item \label{item:1_switched_monotone_direction} Suppose $\cP' = \cP_{[i]}^{+} \cup \cQ_{[i]}$, $p = p_{[i]}^{+}$, and $x(0) = y(0) \notin S_{[i+1]}$. Let $\tau_{y}>0$ be such that $y(\tau_{y}) \in S_{[i+1]}$ as in Lemma \ref{lemma:unswitched_unique_times}, and let $\tau_{x} = \sup \{t>0 : x(t) \in \cl{C}\}$. 
\item \label{item:2_switched_monotone_direction} Suppose $\cP' = \cP_{[i+1]}^{-} \cup \cQ_{[i]}$, $p = p_{[i+1]}^{-}$, and $x(0) = y(0) \notin S_{[i]}$. Let $\tau_{y}>0$ be such that $y(\tau_{y}) \in S_{[i]}$ as in Lemma \ref{lemma:unswitched_unique_times}, and let $\tau_{x} = \sup \{t>0 : x(t) \in \cl{C}\}$.
\end{enumerate}
Then for all $t_{x} \in (0, \tau_{x}]$ and $t_{y} \in (0, \tau_{y}]$ such that $x(t_{x}), y(t_{y})$ lie on the same ray emitting from the origin, we have $||x(t_{x})|| \leq ||y(t_{y})||$.
\end{lemma}

\begin{proof}
We will assume \ref{item:1_switched_monotone_direction} holds, as the case of \ref{item:2_switched_monotone_direction} is treated almost identically. Let $t_{x} \in (0, \tau_{x}], t_{y} \in (0, \tau_{y}]$ be such that $x(t_{x}), y(t_{y})$ lie on the same ray emitting from the origin. By the non-Zeno assumption on \eqref{intro.1}, there are a finite number of switching instants $t_{1}, t_{2}, \ldots, t_{m-1} \leq t_{x}$ occurring up to and possibly including time $t_{x}$. Let $t_{m}>t_{x}$ be the first switching instant after time $t_{x}$. Setting $t_{0}=0$, denote by $p_{i} \in \cP_{[i]}^{+} \cup \cQ_{[i]}$ the index of the active subsystem during times $t \in [t_{i-1}, t_{i})$ for $i \in \{1, \ldots, m\}$.

Note that for all $t_{x}' \in (0, t_{m})$, by Lemma \ref{lemma:unswitched_unique_times} there exists a unique time $t_{y}' \in (0, \tau_{y})$ such that $x(t_{x}'), y(t_{y}')$ lie on the same ray emitting from the origin. For ease of notation, let $g(t)$ be the function mapping each $t_{x}'$ to the corresponding $t_{y}'$.

To prove the lemma, it suffices to show that $||x(t_{x}')|| \leq ||y(t_{y}')|| = ||y(g(t_{x}'))||$ for all such $t_{x}' \in (0, t_{m})$; we will prove this by induction on $i \in \{1, \ldots, m\}$. Since the base case is immediately handled by Lemmas \ref{lemma:unswitched_trajectory_growth_ordering} \ref{item:1_unswitched_ordering} and \ref{lemma:controlling_noncrossing_trajectories} \ref{item:1_noncrossing_trajectories} and the definition of $p_{[i]}^{+}$, we suppose that the statement holds for $t_{x}' \in [t_{k-1}, t_{k})$ with $k \in \{1, \ldots, m-1\}$ and prove it for $t_{x}' \in [t_{k}, t_{k+1})$.

Since for all $t_{x}' \in [t_{k-1}, t_{k})$ we have $||x(t_{x}')|| \leq ||y(g(t_{x}'))||$, by continuity of $x(t)$ and $y(t)$ this holds for $t_{x}'=t_{k}$, so that there exists $\lambda \geq 1$ such that $\lambda x(t_{k}) = y(g(t_{k}))$. Now consider solutions $\hat{x}(t)$ solving $\dot{x} = A_{p_{k+1}}x$ and $\hat{y}(t)$ solving $\dot{y} = A_{p_{[i]}^{+}}y$ with initial conditions $\hat{x}(0) = \hat{y}(0) = y(g(t_{x}'))$. Using linearity of $x(t)$ for $t \in [t_{k}, t_{k+1})$ as in Lemma \ref{lemma:same_hitting_time}, we have from $\hat{x}(0) = \lambda x(t_{k})$ that
\begin{equation*}
    \hat{x}(t) = \lambda x(t_{k}+t)
\end{equation*}
for $t \in [0, t_{k+1}-t_{k})$. By Lemmas \ref{lemma:unswitched_trajectory_growth_ordering} \ref{item:1_unswitched_ordering} and \ref{lemma:controlling_noncrossing_trajectories} \ref{item:1_noncrossing_trajectories} applied to $\hat{x}(t)$ and $\hat{y}(t)$ and by the definition of $p_{[i]}^{+}$, for $t \in [0, t_{k+1}-t_{k})$ we have
\begin{equation*}
    ||y(g(t_{k}+t))|| = ||\hat{y}(g(t))|| \geq ||\hat{x}(t)|| = \lambda ||x(t_{k}+t)|| \geq ||x(t_{k}+t)||,
\end{equation*}
completing the proof.
\end{proof}

\begin{remark}\label{rmk:validity_only_in_2_dimensions}
The great amount of control that we have over the ``worst-case'' switched trajectories in Lemmas \ref{lemma:unswitched_trajectory_growth_ordering}, \ref{lemma:controlling_noncrossing_trajectories}, and \ref{lemma:switched_monotone_direction_bound} is only possible because of the low dimensionality of each subsystem restricting where the forward trajectories may travel in phase space. It is clear that these arguments do not carry over to dimensions three and higher; aside from the difficulty of finding analogous objects to the rays $S_{[i]}$, we no longer have the fundamental property that a trajectory passing from one ray to another must intersect every ray in between.
\end{remark}

The following more specific result will be needed later in the paper.

\begin{lemma} \label{lemma:f_p-_p+_negative_forces_others_negative}
Let $C$ be the open cone bounded between $S_{[i]}$ and $S_{[i+1]}$ for some $i \in \{1, \ldots, n\}$ in the counterclockwise direction. Suppose $\Delta_{p_{[i+1]}^{-}, p_{[i]}^{+}}(x) < 0$ for $x \in C$. Then for all $q \in \cP_{[i+1]}^{-}$ we have $\Delta_{q, p_{[i]}^{+}}(x) < 0$ for $x \in C$.
\end{lemma}
\begin{proof}
Since $\Delta_{p_{[i+1]}^{-}, p_{[i]}^{+}}(x) < 0$ for $x \in C$, by Lemma \ref{lemma:unswitched_trajectory_growth_ordering} \ref{item:3ii_unswitched_ordering} any functions $x(t), \ y(t)$ solving $\dot{x} = A_{p_{[i+1]}^{-}}x, \ \dot{y} = A_{p_{[i]}^{+}}y$ respectively with $x(0) = y(0) \in \cl{C} \setminus \{0\}$ must satisfy
\begin{equation*}
    ||x(t_{x})|| < ||y(-t_{y})||
\end{equation*}
for any ray $R$ and any minimal $t_{x}, t_{y}>0$ such that $x(t_{x}), y(-t_{y}) \in R$.

Letting $q \in \cP_{[i+1]}^{-}$ be arbitrary, by Lemma \ref{lemma:switched_monotone_direction_bound} \ref{item:2_switched_monotone_direction} applied to the unswitched trajectory $x_{q}(t)$ solving $\dot{x} = A_{[q]}x$ with $x_{q}(0)=x(0)$, if $t_{q}$ is such that $x_{q}(t_{q})$ lies on the same ray $R$ as $x(t_{x})$ then
\begin{equation*}
    ||x_{q}(t_{q})|| < ||x(t_{x})|| < ||y(-t_{y})||.
\end{equation*}
By the contrapositive of Lemma \ref{lemma:unswitched_trajectory_growth_ordering} \ref{item:3i_unswitched_ordering} as stated in Remark \ref{rmk:contrapositive_of_growth_ordering_theorem}, we conclude that $\Delta_{q, p_{[i]}^{+}}(x) < 0$ for $x \in C$ as needed.
\end{proof}

Before proceeding to the main results, we need two lemmas for bounding the absolute value $||y(t)||$ of unswitched trajectories solving $\dot{y} = A_{p}y$ for $p \in \cP$ or $p \in \cQ_{[i]}$. The first one is a standard exponential lower bound, whose proof is included for completeness.

\begin{lemma} \label{lemma:exponential_lower_bound_P}
Let $y(t)$ solve $\dot{y} = A_{p}y$, $y(0) \neq 0$ for some $p \in \cP$. Then there exists some constant $d_{p}>0$ such that for all $t \geq 0$,
\begin{equation*}
    ||y(t)|| \geq ||y(0)|| e^{-d_{p}t}.
\end{equation*}
\end{lemma}
\begin{proof}
Since the unit circle $\{z \in \RR^{2} : ||z|| = 1\}$ is compact, $\inf_{||z||=1}(z^{\top}A_{p}z)$ is finite. Set \\$d_{p} = \max\{-\inf_{||z||=1}(z^{\top}A_{p}z), 1\}>0$. Then
\begin{equation*}
    \frac{d}{dt}||y(t)|| = \frac{y(t)^{\top} A_{p} y(t)}{||y(t)||} \geq \inf_{||z||=1}(z^{\top}A_{p}z) \cdot ||y(t)|| \geq -d_{p}||y(t)||,
\end{equation*}
and the conclusion follows by the variant of Grönwall's inequality with inequality signs reversed.
\end{proof}

The second bound is a decaying exponential upper bound on $||y(t)||$, which is slightly less trivial than the lower bound.

\begin{lemma} \label{lemma:exponential_upper_bound_Q_i}
Denote by $C$ the open cone bounded between $S_{[i]}$ and $S_{[i+1]}$ for some $i \in \{1, \ldots, n\}$ in the counterclockwise direction. Let $y(t)$ solve $\dot{y} = A_{p}y$, $y(0) = \xi \in \cl{C} \setminus \{0\}$ for some $p \in \cQ_{[i]}$. Then there exists some constant $k_{p}>0$ such that for all $t \geq 0$,
\begin{equation*}
    ||y(t)|| \leq ||y(0)|| e^{-k_{p}t}.
\end{equation*}
\end{lemma}
\begin{proof}
Note first that as discussed after the proof of Lemma \ref{lemma:controlling_noncrossing_trajectories}, since $p \in \cQ_{[i]}$ the eigenvalues of $A_{p}$ must be real. Now,
\begin{equation*}
    \frac{d}{dt}||y(t)|| = \frac{y(t)^{\top}A_{p}y(t)}{||y(t)||} = ||y(t)|| \cdot z(t)^{\top}A_{p}z(t)
    \leq ||y(t)|| \cdot \sup_{z \in S^{1} \cup \cl{C}}z^{\top}A_{p}z
\end{equation*}
where $z(t) = y(t)/||y(t)||$. By construction of the rays $S_{[i]}$ and $ S_{[i+1]}$, for all $z \in \cl{C} \setminus \{0\}$ we have that either $z \in \cl{C_{p}^{+}} \setminus \{0\}$ or $z \notin \cl{C_{p}^{+}}$, where $C_{p}^{+}$ is as in \eqref{eq:Cp+_cases_definition}. If $z \in \cl{C_{p}^{+}} \setminus \{0\}$, then by Lemma \ref{lemma:quadratic_form_negative_lower_bound} we have some $k_{p} > 0$ for which
\begin{equation*}
    \sup_{z \in S^{1} \cap \cl{C^{+}}}z^{\top}A_{p}z < -k_{p} < 0,
\end{equation*}
hence $\frac{d}{dt}||y(t)|| \leq -k_{p}||y(t)||$; applying Grönwall's inequality completes the proof in this case. If $z \notin \cl{C_{p}^{+}}$, then again by Lemma \ref{lemma:quadratic_form_negative_lower_bound} we have that $y(t)$ solving $\dot{y} = A_{p}y$ with $y(0) \notin \cl{C^{+}}$, $y(0) \neq 0$ eventually lies in $\cl{C^{+}} \setminus \{0\}$ for all large enough time $t>0$, hence $y(t)$ must intersect either $S_{[i]}$ or $S_{[i+1]}$, contradicting $p \in \cQ_{[i]}$.
\end{proof}

\subsection{Proof of Proposition \ref{propn:crossing_between_sector_boundaries}}

\begin{proof}[Proof of Proposition \ref{propn:crossing_between_sector_boundaries}]
We will only prove \ref{item:1_crossing_between_sector_boundaries} and \ref{item:3_crossing_between_sector_boundaries}; the proofs for \ref{item:2_crossing_between_sector_boundaries} and \ref{item:4_crossing_between_sector_boundaries} are similar.

First suppose $x(0) \in S_{[i]}$ and $x(\tau) \in S_{[i+1]}$. Let $\theta(t)$ be the angular polar coordinate of $x(t)$. Let $t_{0} = 0$, and denote by $t_{1}, \ldots, t_{m-1} < \tau$ the ordered list of all times less than $\tau$ at which \eqref{intro.1} switches between subsystems during which $\dot{\theta}(t) \geq 0$ and subsystems during which $\dot{\theta}(t) < 0$, and let $t_{m} = \tau$; note that by Remark \ref{rmk:Q_i_trichotomy} and the definitions of $\cP_{[i]}^{+}$ and $\cP_{[i+1]}^{-}$, $\dot{\theta}(t)$ cannot change sign between switches to different subsystems. This list of times $\{t_{j}\}_{j=1}^{m}$ is finite by the non-Zeno assumption on \eqref{intro.1}. Since we are assuming $\tau>0$, $x(0) \in S_{[i]}$, and $x(\tau) \in S_{[i+1]}$, we must have that \eqref{intro.1} switches between subsystems with indices in $\cP_{[i]}^{+} \cup \cQ_{[i]}$ for $t \in [t_{0}, t_{1}) \cup [t_{2}, t_{3}) \cup \cdots \cup [t_{m-1}, t_{m})$ and switches between subsystems with indices in $\cP_{[i+1]}^{-} \cup \cQ_{[i]}$ for $t \in [t_{1}, t_{2}) \cup [t_{3}, t_{4}) \cup \cdots \cup [t_{m-2}, t_{m-1})$.

Let $y(t)$ solve $\dot{y} = A_{p_{[i]}^{+}}y$ with $y(0) = x(0)$, and let $\tau_{1} > 0$ be the minimal time given by Lemma \ref{lemma:same_hitting_time} such that $y(\tau_{1}) \in S_{[i+1]}$. Then $y_{1}(t) := y(t+\tau_{1})$ solves the same differential equation $\dot{y} = A_{p_{[i]}^{+}}y$ and $y_{1}(0) \in S_{[i+1]}$, $y_{1}(-\tau) = y(0) = x(0)$. Now let $y_{2}(t)$ solve $\dot{y} = A_{p_{[i+1]}^{-}}y$ with $y_{2}(0) = y_{1}(0)$, and let $\tau_{2} > 0$ be the minimal time such that $y_{2}(\tau_{2}) \in S_{[i]}$. Since $w([i], [i+1]) \cdot w([i+1], [i]) < 1$ by assumption, we have $||y_{2}(\tau_{2})|| < ||y_{1}(-\tau_{1})||$. By the contrapositive to Lemma \ref{lemma:unswitched_trajectory_growth_ordering} \ref{item:3i_unswitched_ordering} with $p = p_{[i+1]}^{-}$ and $q = p_{[i]}^{+}$ we have that $\Delta_{p_{[i+1]}^{-}, p_{[i]}^{+}}$ is negative on $C$.

As in the proof of Lemma \ref{lemma:switched_monotone_direction_bound}, for ease of notation we define the function $g: [0, \tau] \to [0, \tau_{1}]$ mapping $t_{x} \in [0, \tau]$ to the unique minimal time $t_{y} = g(t_{x}) \in [0, \tau_{1}]$ such that $x(t_{x}), y(t_{y})$ lie on the same ray emitting from the origin. Since $||y(g(\tau))|| = ||y(\tau_{1})|| = w([i], [i+1]) \cdot ||y(0)||$, it suffices to show $||x(\tau)|| \leq ||y(g(\tau))||$ to prove \ref{item:1_crossing_between_sector_boundaries}. Similar again to the proof of Lemma \ref{lemma:switched_monotone_direction_bound}, we will achieve this by proving by induction on $i \in \{1, \ldots, m\}$ that for $t_{x} \in [t_{i-1}, t_{i})$ we have $||x(t_{x})|| \leq ||y(g(t_{x}))||$.

Since $x(0) = y(0)$ the base case $i=1$ is immediately handled by Lemma \ref{lemma:switched_monotone_direction_bound} \ref{item:1_switched_monotone_direction}, and continuity of $x(t), y(t)$, and $g(t)$ yields $||x(t_{1})|| \leq ||y(g(t_{1}))||$. Now suppose the conclusion holds for $1 \leq j <m-1$ with $j$ odd, and we will prove it for $j+1$ even. As above, by continuity we have $||x(t_{j})|| \leq ||y(g(t_{j}))||$, so there exists $\lambda \geq 1$ such that $\lambda x(t_{j}) = y(g(t_{j}))$. Now consider solutions $y_{1}(t), y_{2}(t)$ solving $\dot{y} = A_{p_{[i+1]}^{-}}y$ with $y_{1}(0) = x(t_{j})$ and $y_{2}(0) = y(g(t_{j})) = \lambda x(t_{j})$; notice that $y_{1}(0)$ and $y_{2}(0)$ lie on the same ray with $||y_{1}(0)|| \leq ||y_{2}(0)||$, so as in the proof of Lemma \ref{lemma:same_hitting_time} we can conclude that $\lambda y_{1}(t) = y_{2}(t)$. Let $\tau_{3}>0$ be the unique minimal time at which $y_{2}(\tau_{3}) \in S_{[i]}$. Similar to the function $g$, define the function $h: [0, g(t_{1})] \to [0, \tau_{3}]$ mapping $t_{y} \in [0, g(t_{1})]$ to the unique minimal time $h(t_{y})$ such that $y(t_{y}), y_{2}(h(t_{y}))$ lie on the same ray emitting from the origin. Then by Lemma \ref{lemma:unswitched_trajectory_growth_ordering} \ref{item:3ii_unswitched_ordering}, since $\Delta_{p_{[i+1]}^{-}, p_{[i]}^{+}}$ is negative on $C$ we have
\begin{equation*}
    ||y_{2}(h(t_{y}))|| < ||y(t_{y})|| \hspace{10px} \forall t_{y} \in [0, g(t_{1})].
\end{equation*}
Thus for $t \in (0, t_{j+1}-t_{j})$, by Lemma \ref{lemma:switched_monotone_direction_bound} \ref{item:2_switched_monotone_direction} we have
\begin{equation}\label{eq:strict_inequality_crossing_between_boundaries}
    ||y(g(t_{j}+t))|| > ||y_{2}(h \circ g(t_{j}+t))|| \geq ||y_{1}(h \circ g(t_{j}+t))|| \geq ||x(t_{j}+t)||
\end{equation}
with $y(g(t_{j}+t)), y_{2}(h \circ g(t_{j}+t)), y_{1}(h \circ g(t_{j}+t)), x(t_{j}+t)$ all laying on the same ray. Thus for $t_{x} = t_{j} + t \in [t_{j}, t_{j+1})$, we have $||y(g(t_{x}))|| \geq ||x(t_{x})||$, with strict inequality for $t_{x}>t_{j}$. In fact, by applying Lemma \ref{lemma:unswitched_trajectory_growth_ordering} \ref{item:3ii_unswitched_ordering} for $t$ in the slightly extended interval $(0, t_{j+1}-t_{j}]$, we may conclude $||y(g(t_{j+1}))|| > ||x(t_{j+1})||$.

Now suppose the conclusion holds for $1<j<m$ even, and we will prove it for $j+1$ odd. By continuity, $||x(t_{j})|| \leq ||y(g(t_{j}))||$, so there exists $\lambda \geq 1$ such that $\lambda x(t_{j}) = y(g(t_{j}))$. By Lemma \ref{lemma:switched_monotone_direction_bound} \ref{item:1_switched_monotone_direction} applied to $\lambda x(t_{x})$ and $y(g(t_{x}))$ for $t_{x} \in [t_{j}, t_{j+1})$, we get
\begin{equation*}
    ||y(g(t_{x}))|| \geq \lambda||x(t_{x})|| \geq ||x(t_{x})||,
\end{equation*}
completing the induction and hence the proof of \ref{item:1_crossing_between_sector_boundaries}.

We proceed to the proof of \ref{item:3_crossing_between_sector_boundaries}. In the above proof of \ref{item:1_crossing_between_sector_boundaries}, without using the assumption that $x(\tau) \in S_{[i+1]}$ we showed that if $y(t)$ solves $\dot{y} = A_{p_{[i]}^{+}}y$ with $y(0) = x(0)$ and $y(\tau_{1}) \in S_{[i+1]}$ and $g: [0, \tau] \to [0, \tau_{1}]$ are as above, then for all $t_{x} \in [0, \tau]$ we have $||x(t_{x})|| \leq ||y(g(t_{x}))||$. By the assumption $x(\tau) \in S_{[i]}$ for this case, there must exist at least one interval $[t_{j-1}, t_{j})$ as above with $j$ even, for otherwise the angular polar coordinate of $x(t)$ would be nonincreasing for $t \in [0, \tau]$. From the above proof of  \ref{item:1_crossing_between_sector_boundaries}, it can be seen that the strict inequality in \eqref{eq:strict_inequality_crossing_between_boundaries} propagates to all further intervals $[t_{k-1}, t_{k})$ for $k>j$. That is, using $||y(g(t_{x}))|| > ||x(t_{x})||$ for $t_{x} \in (t_{j}, t_{j+1}]$ and in particular the fact that $||y(g(t_{j+1}))|| > ||x(t_{j+1})||$, we may write
\begin{equation*}
    ||y(g(t_{j+1}))|| \geq ||\mu x(t_{j+1})||
\end{equation*}
for some $\mu>1$ and continue each step of the induction with the function $\mu x(t)$ instead of $x(t)$; upon reaching $j=m$, we get
\begin{equation*}
    ||y(g(\tau))|| = ||y(g(t_{m}))|| \geq ||\mu x(t_{m})|| > ||x(t_{m})|| = ||x(\tau)||.
\end{equation*}
So, we can conclude that 
\begin{equation*}
    ||x(0)|| = ||y(0)|| = ||y(g(\tau))|| > ||x(\tau)||,
\end{equation*}
completing the proof of \ref{item:3_crossing_between_sector_boundaries}.
\end{proof}

\subsection{Proof of Proposition \ref{propn:staying_within_sector_boundaries}}

\begin{proof}[Proof of Proposition \ref{propn:staying_within_sector_boundaries}]
First, exactly as in the proof of Proposition \ref{propn:crossing_between_sector_boundaries}, from the assumption $w([i], [i+1]) \cdot w([i+1], [i])<1$ we can conclude that $\Delta_{p_{[i+1]}^{-},p_{[i]}^{+}}$ is negative on $C$. By Lemma \ref{lemma:f_p-_p+_negative_forces_others_negative}, for any $q \in \cP_{[i+1]}^{-}$ the expression $\Delta_{q,p_{[i]}^{+}}$ is negative on $C$ as well. Since $\Delta_{q,p_{[i]}^{+}}(x) \neq 0$ for $x \in S_{[i]} \cup S_{[i+1]}$ by assumption, we further have that $\Delta_{q,p_{[i]}^{+}}$ is negative on $\cl{C} \setminus \{0\}$ for all $q \in \cP_{[i+1]}^{-}$.

We claim that since $x(t) \in \cl{C} \setminus \{0\}\}$ for $t>0$, an infinite amount of time must be spent in subsystems with corresponding indices in $\cP_{[i+1]}^{-} \cup \cQ_{[i]}$; for a contradiction, suppose only a finite amount of time is spent in such subsystems, say $T_{0} \geq 0$ amount of time. Denote by $\theta(t)$ the angular polar coordinate of $x(t)$. Since $x(t) \in \cl{C} \setminus \{0\}$, for some $\theta_{1} < \theta_{2}$ we have $\theta_{1} \leq \theta(t) \leq \theta_{2}$ for all $t \geq 0$. For $p \in \cP_{[i+1]}^{-} \cup \cQ_{[i]} = \cP \setminus \cP_{[i]}^{+}$ and $z(t)$ solving $\dot{z} = A_{p}z$, denote by $(s(t),\varphi(t))$ the radial and angular polar coordinates of $z(t)$. Since $\dot{\varphi}(t)$ does not depend on $s(t)$ by equation \eqref{eq:theta_dot_computation}, we have
\begin{equation}\label{eq:c_abs}
\begin{split}
    c_{abs} &:= \max_{p \in \cP \setminus \cP_{[i]}^{+}} \{|\dot{\varphi}(t)|: \dot{z} = A_{p}z, z(0) \in S^{1} \cap \cl{C}, z = (s, \varphi)\} \\
    &= \sup_{p \in \cP \setminus \cP_{[i]}^{+}} \{|\dot{\varphi}(t)|: \dot{z} = A_{p}z, z(0) \in \cl{C} \setminus \{0\}, z = (s, \varphi)\} > 0,
\end{split}
\end{equation}
where the first expression is a ``$\max$'' instead of ``$\sup$'' since we are taking the supremum over the compact set $(\cP \setminus \cP_{[i]}^{+}) \times (S^{1} \cap \cl{C})$. So, over all $t>0$ such that the active subsystem is indexed by $p \in \cP \setminus \cP_{[i]}^{+}$ we can only have that $\theta(t)$ decreases by at most $c_{abs}T_{0}$ (more precisely, the negative variation of $\theta(t)$ is at most $c_{abs}T_{0}$; see \cite{folland1999real} Section 3.5). On the other hand, from Corollary \ref{cor:theta_minimal_speed} we have
\begin{equation}\label{eq:c_+}
    c^{+} := \inf_{p \in \cP_{[i]}^{+}} \{\dot{\varphi}(t): \dot{z} = A_{p}z, z(0) \in \cl{C} \setminus \{0\}, z = (s, \varphi)\} > 0,
\end{equation}
so for $t> \frac{c_{abs}T_{0} + 2(\theta_{2} - \theta_{1})}{c^{+}}$ we must have had that $\theta(t)$ increases by at least
\begin{equation*}
    c^{+} \cdot \frac{c_{abs}T_{0} + 2(\theta_{2} - \theta_{1})}{c^{+}} = c_{abs}T_{0} + 2(\theta_{2} - \theta_{1})
\end{equation*}
in this time period (more precisely, the positive variation of $\theta(t)$ for $t \in (0, \frac{c_{abs}T_{0} + 2(\theta_{2} - \theta_{1})}{c^{+}})$ is at least $c_{abs}T_{0} + 2(\theta_{2} - \theta_{1})$). But this forces $\theta(t_{0})>\theta_{2}$ for some $t_{0} \in (0, \frac{c_{abs}T_{0} + 2(\theta_{2} - \theta_{1})}{c^{+}})$, contradicting the assumption that $x(t) \in \cl{C} \setminus \{0\}$ for all $t>0$.

\textit{\underline{Case 1.}} Suppose that an infinite amount of time is spent in subsystems with corresponding indices in $\cP_{[i+1]}^{-}$. That is, there exist consecutive time intervals $J_{j} = [\ell_{j}, r_{j})$ with $\ell_{j} < r_{j} < \ell_{j+1} < r_{j+1}$ for $j \in \NN$ with $\sum_{j=1}^{\infty}(r_{j}-\ell_{j}) = \infty$ such that for $t \in J_{j}$, $x(t)$ solves $\dot{x} = A_{p_{j}}x$ for some $p_{j} \in \cP_{[i+1]}^{-}$. We will show that $\lim_{t \to \infty}x(t)=0$ in this case.

Let $y_{s}(t)$ solve $\dot{y} = A_{p_{[i]}^{+}}y$ with $y_{s}(0) = x(s)$, and let $\tau_{1,s} \leq 0 \leq \tau_{2,s},\ \tau_{1,s} < \tau_{2,s}$ be the unique times with minimal absolute value given by Lemma \ref{lemma:unswitched_unique_times} and Remark \ref{rmk:negative_Ap_intersection_times} such that $y_{s}(\tau_{1,s}) \in S_{[i]}$ and $y_{s}(\tau_{2,s}) \in S_{[i+1]}$. Denote $T_{s} := [\tau_{1,s}, \tau_{2,s}]$ and $Y(s) = \{y_{s}(t) : t \in T_{s}\}$, and let $D(s)$ be the compact region enclosed by $S_{[i]}, S_{[i+1]}, Y(s)$, and the origin. Using the same linear scaling property as in Lemma \ref{lemma:same_hitting_time}, for each $s>0$ there exists a scalar $\mu(s)>0$ such that
\begin{align*}
    &Y(s) = \mu(s)Y(0) := \{\mu(s)y : y \in Y(0)\} \hspace{10px} \mathrm{and} \\
    &D(s) = \mu(s)D(0) := \{\mu(s)y : y \in D(0))\}.
\end{align*}
By Lemma \ref{lemma:switched_monotone_direction_bound}, Lemma \ref{lemma:unswitched_trajectory_growth_ordering} \ref{item:3ii_unswitched_ordering}, and the fact that $\Delta_{p_{[i+1]}^{-},p_{[i]}^{+}}$ is negative on $C$, we have for all $s \geq 0$ and $t \geq s$ that $x(t) \in D(s)$.

Similar to equation \eqref{eq:c_+}, by Corollary \ref{cor:theta_minimal_speed} we have
\begin{equation}\label{eq:c_-}
    c^{-} := \sup_{p \in \cP_{[i+1]}^{-}} \{\dot{\varphi}(t): \dot{z} = A_{p}z, z(0) \in \cl{C} \setminus \{0\}, z = (s, \varphi)\} < 0.
\end{equation}
Fix $j \in \NN$ and consider $x(t)$ for $t \in J_{j} = [\ell_{j}, r_{j})$, so that $\dot{x}(t)=A_{p_{j}}x(t)$ with $p_{j} \in \cP_{[i+1]}^{-}$. Let $\theta(t), \varphi(t)$ be the angular polar coordinates of $x(t), y_{\ell_{j}}(t)$ respectively for $t \in J_{j}$, noting that since $\dot{\theta}(t)<c^{-}<0$ and $\dot{\varphi}(t)>c^{+}>0$, the function $\rho_{j}: [\ell_{j}, r_{j}] \to \RR$, $\rho_{j} = \varphi^{-1} \circ \theta$ is strictly decreasing with $\dot{\rho_{j}}(t) < -c_{\rho} < 0$ for some constant $c_{\rho}>0$. Now as in the proof of Lemma \ref{lemma:unswitched_trajectory_growth_ordering}, we have $x(t) = \lambda_{j}(t) y_{\ell_{j}}(\rho_{j}(t))$ for some positive scalar function $\lambda_{j}$ depending on $p_{j} \in \cP_{[i+1]}^{-}$. Since $\Delta_{p_{j},p_{[i]}^{+}}$ is negative on $\cl{C} \setminus \{0\}$ and $Y_{0} \subseteq \cl{C} \setminus \{0\}$ is compact, there exists $c_{\det,j} > 0$ such that
\begin{equation*}
    \det (A_{p_{j}}y_{0}(t) \ | \ A_{p_{[i]}^{+}}y_{0}(t)) < -c_{\det,j}
\end{equation*}
for all $t \in T_{0}$. Let $c_{\det} := \min_{j \in \NN} c_{\det,j} > 0$, where we have ``$\min$'' instead of ``$\inf$'' since the set $\{p_{j}\}_{j \in \NN} \subseteq \cP_{[i+1]}^{-}$ is finite. Similarly, by Lemma \ref{lemma:det_x_Apx_positive} and compactness of $Y_{0} \subseteq \cl{C} \setminus \{0\}$, there exists a constant $c_{dir} > 0$ such that
\begin{equation*}
    \det (y_{0}(t) \ | \ A_{p_{j}}y_{0}(t)) < -c_{dir}
\end{equation*}
for all $t \in T_{0}$ and $j \in \NN$.

Putting all of this together, starting from the equality \eqref{eq:det_rho_lambda_equality} in the proof of Lemma \ref{lemma:unswitched_trajectory_growth_ordering} and noticing that
\begin{equation*}
    y_{0,j}(t) := \frac{y_{\ell_{i}}(\rho_{j}(t))}{\mu(\ell_{i})} \in Y_{0}
\end{equation*}
for $t \in T_{\ell_{j}}$, we get
\begin{align*}
    &\det(A_{p_{j}}y_{\ell_{j}}(\rho_{j}(t)) \ | \ A_{p_{[i]}^{+}}y_{\ell_{j}}(\rho_{j}(t))) = \frac{1}{\dot{\rho_{j}}(t)} \frac{\dot{\lambda_{j}}(t)}{\lambda_{j}(t)} \det(y_{\ell_{j}}(\rho_{j}(t)) \ | \ A_{p_{j}}y_{\ell_{j}}(\rho_{j}(t))) \\
    &\implies \mu(\ell_{j})^{2} \det(A_{p_{j}}y_{0,j}(t) \ | \ A_{p_{[i]}^{+}}y_{0,j}(t)) = \frac{1}{\dot{\rho_{j}}(t)} \frac{\dot{\lambda_{j}}(t)}{\lambda_{j}(t)} \mu(\ell_{j})^{2} \det(y_{0,j}(t) \ | \ A_{p_{j}}y_{0,j}(t)) \\
    &\implies \dot{\lambda_{j}}(t) = \dot{\rho_{j}}(t) \lambda_{j}(t) \frac{\det(A_{p_{j}}y_{0,j}(t) \ | \ A_{p_{[i]}^{+}}y_{0,j}(t))}{\det(y_{0,j}(t) \ | \ A_{p_{j}}y_{0,j}(t))} < - \frac{ c_{\rho} c_{\det} }{c_{dir}} \lambda_{j}(t) =: -c \lambda_{j}(t)
\end{align*}
with $-c<0$, for all $j \in \NN$ and $t \in J_{j}$. Hence since $\lambda_{j}(\ell_{j})=1$ by definition of $y_{\ell_{j}}(t)$, we have
\begin{align*}
    ||x(r_{j})|| &= \lambda_{j}(r_{j}) ||y_{\ell_{j}}(\rho_{j}(r_{j}))|| \\
    &< e^{-c(r_{j}-\ell_{j})} \lambda_{j}(\ell_{j}) ||y_{\ell_{j}}(\rho_{j}(r_{j}))|| \\
    &= e^{-c(r_{j}-\ell_{j})} ||y_{\ell_{j}}(\rho_{j}(r_{j}))||
\end{align*}
for $x(r_{j}), y_{\ell_{j}}(\rho_{j}(r_{j}))$ laying on the same ray emitting from the origin. Thus for all $t \geq r_{j}$,
\begin{equation*}
    x(t) \in e^{-c(r_{j}-\ell_{j})}D(\ell_{j}),
\end{equation*}
and repeatedly applying this inequality for $j \in \NN$ we get that for any $N \in \NN$ and $t \geq N$,
\begin{equation*}
    x(t) \in e^{-c\sum_{j=1}^{N}(r_{j}-\ell_{j})}D(\ell_{1});
\end{equation*}
using the fact that $D(\ell_{1})$ is bounded, $0 \in D(\ell_{1})$, and $\lim_{N \to \infty}\sum_{j=1}^{N}(r_{j}-\ell_{j}) = \infty$, the convergence of sets
\begin{equation*}
     e^{-c\sum_{j=1}^{N}(r_{j}-\ell_{j})}D(\ell_{1}) \xrightarrow{N \to \infty} \{0\}
\end{equation*}
in, for example, the Hausdorff metric concludes the proof that $\lim_{t \to \infty} x(t) = 0$ in this case.

\textit{\underline{Case 2.}} Now suppose that a finite amount of time is spent in subsystems with corresponding indices in $\cP_{[i+1]}^{-}$, so that an infinite amount of time is spent in subsystems with indices in $\cQ_{[i]}$. Similar to Case 1, let $J_{j} = [\ell_{j}, r_{j})$ for $j \in \NN$ be the intervals on which $x(t)$ solves $\dot{x} = A_{p_{j}}x$ for $p_{j} \in \cQ_{[i]}$, with the property that $\sum_{j=1}^{\infty}(r_{j}-\ell_{j}) = \infty$.

Let $y^{+}_{s}(t), y^{-}_{s}(t)$ solve $\dot{y} = A_{p_{[i]}^{+}}y$, $\dot{y} = A_{p_{[i+1]}^{-}}y$ respectively with $y^{\pm}_{s}(0) = x(s)$, and let $\tau^{\pm}_{1,s} \leq 0 \leq \tau^{\pm}_{2,s},\ \tau^{\pm}_{1,s} < \tau^{\pm}_{2,s}$ be the unique times with minimal absolute value given by Lemma \ref{lemma:unswitched_unique_times} and Remark \ref{rmk:negative_Ap_intersection_times} such that $y^{+}_{s}(\tau^{+}_{1,s}), y^{-}_{s}(\tau^{-}_{2,s}) \in S_{[i]}$ and $y^{+}_{s}(\tau^{+}_{2,s}), y^{-}_{s}(\tau^{-}_{1,s}) \in S_{[i+1]}$. Let $T^{\pm}_{s} := [\tau^{\pm}_{1,s}, \tau^{\pm}_{2,s}]$ and $Y^{\pm}(s) = \{y^{\pm}_{s}(t) : t \in T^{\pm}_{s}\}$, and let $D^{\pm}(s)$ be the compact region enclosed by $S_{[i]}, S_{[i+1]}, Y^{\pm}(s)$, and the origin. Since $\Delta_{p_{[i+1]}^{-},p_{[i]}^{+}}$ is negative on $C$, by Lemma \ref{lemma:switched_monotone_direction_bound}, Lemma \ref{lemma:unswitched_trajectory_growth_ordering} \ref{item:3ii_unswitched_ordering}, and Lemma \ref{lemma:unswitched_trajectory_growth_ordering} \ref{item:2i_unswitched_ordering} we have 
\begin{equation} \label{eq:D+-_forward_invariance_property}
    \forall s \geq 0, \ \forall t \geq s, \ \ x(t) \in D^{\pm}(s)
\end{equation}

for all $s \geq 0$ and $t \geq s$ that $x(t) \in D^{\pm}(s)$.

Now fix $j \in \NN$ and consider $x(t)$ for $t \in J_{j} = [\ell_{j}, r_{j})$, so that $\dot{x}(t)=A_{p_{j}}x(t)$ with $p_{j} \in \cQ_{[i]}$. By Remark \ref{rmk:Q_i_trichotomy}, the angular polar coordinate $\psi(t)$ of $x(t)$ for $t \in J_{j}$ is either strictly increasing, strictly decreasing, or constantly zero. If $\dot{\psi}(t) \equiv 0$, then $x(t) = e^{-a_{p_{j}}t}x(\ell_{j})$ for some constant $a_{p_{j}}>0$ depending only on $p_{j} \in \cQ_{[i]}$; since $\cQ_{[i]}$ is finite, we have $a:=\min_{q \in \cQ_{[i]}}a_{p_{j}} > 0$, and $||x(t)|| \leq e^{-at}||x(\ell_{j})||$ for $t \in J_{j}$; thus
\begin{equation}\label{eq:x_straight_line_final_estimate}
    ||x(r_{j})|| \leq e^{-a(r_{j}-\ell_{j})}||x(\ell_{j})||
\end{equation}

Now we will assume $\dot{\psi}(t)<0$; the case $\dot{\psi}(t)>0$ is treated similarly. Since $\psi$ is strictly decreasing, we have $\psi([\ell_{j}, r_{j})) = (\psi_{2}, \psi_{1}]$ with $\psi(\ell_{j}) = \psi_{1}$. Let $w:[0, t_{w}) \to \cl{C}$ solve $\dot{w} = A_{p_{[i+1]}^{-}}w$ with $w(0) = y(0) = x(\ell_{j})$, $w(t_{w}) \in S_{[i]}$, and denote by $\varphi(t), \theta(t)$ the angular polar coordinates of $y^{+}_{\ell_{j}}(t), w(t)$ respectively. Since $p_{[i+1]}^{-} \in \cP_{[i+1]}^{-}$, Corollary \ref{cor:theta_minimal_speed} gives $\dot{\theta}(t) < 0$ for all $t \geq 0$ such that $w(t) \in \cl{C} \setminus \{0\}$, and also $\theta(0) = \psi(\ell_{j}) = \psi_{1}$ since $w(0) = x(\ell_{j})$; thus $\theta(t)$ is invertible with $\theta^{-1}((\psi_{2}, \psi_{1}]) = [0, t_{\max})$ for some $t_{\max}>0$. Defining
\begin{align*}
    &\rho_{j}: [0, t_{\max}) \to \RR, \ \ \rho_{j} = \varphi^{-1} \circ \theta, \\
    &\kappa_{j}: [0, t_{\max}) \to \RR, \ \ \kappa_{j} = \psi^{-1} \circ \theta,
\end{align*}
we have exactly as in Case 1 that for some positive scalar function $\lambda_{j}$ depending on $p_{j}$,
\begin{equation*}
    w(t) = \lambda_{j}(t) y^{+}_{\ell_{j}}(\rho_{j}(t))
\end{equation*}
with $\lambda_{j}(t) \leq e^{-c^{-}t}$ for $t \in [0, t_{\max})$ and where $c^{-}$ does not depend on $j$. Additionally, $w(t)$ and $x(\kappa_{j}(t))$ lie on the same ray emitting from the origin; by Lemma \ref{lemma:switched_monotone_direction_bound} \ref{item:2_switched_monotone_direction}, we have $||w(t)|| \geq ||x(\kappa_{j}(t))||$ for all $t \in [0, t_{\max})$. For ease of notation, denote $\tilde{x} = x \circ \kappa_{j}$.

By Lemmas \ref{lemma:exponential_lower_bound_P} and \ref{lemma:exponential_upper_bound_Q_i}, there exist constants $d_{p_{j}}, k_{p_{j}}>0$ such that 
\begin{equation*}
    ||\tilde{x}(0)|| e^{-d_{p_{j}}t} \leq ||\tilde{x}(t)|| \leq ||\tilde{x}(0)|| e^{-k_{p_{j}}t}
\end{equation*}
for all $t \in [0, t_{\max})$. Letting $d = \max_{p_{j} \in \cQ_{[i]}}d_{p_{j}}>0$ and $k = \min_{p_{j} \in \cQ_{[i]}}k_{p_{j}}>0$, we have
\begin{equation}\label{eq:double_bound_on_x(t)}
    ||\tilde{x}(0)|| e^{-dt} \leq ||\tilde{x}(t)|| \leq ||\tilde{x}(0)|| e^{-kt}
\end{equation}
for all such $t$. By the upper bound in equation \eqref{eq:double_bound_on_x(t)},
\begin{equation*}
    ||x(r_{j})|| \leq ||x(\ell_{j})|| e^{-k(r_{j} - \ell_{j})} =: ||x(\ell_{j})|| \cdot \eta.
\end{equation*}
Let $t_{int}>0$ be the time of intersection of $\tilde{x}(t)$ with the circle $C_{\eta} := \{z \in \RR^{2} : ||z|| = ||x(\ell_{j})|| \cdot \eta\}$; by the lower bound in \eqref{eq:double_bound_on_x(t)}, $t_{int} \geq -\frac{1}{d}\log(\eta)$, and by the upper bound, $t_{int} \leq t_{\max}$. So, we have
\begin{align*}
     ||\tilde{x}(t_{int})|| &\leq ||w(t_{int})|| \leq e^{-c^{-} t_{int}}||y_{\ell_{j}}^{+}(\rho(t_{int}))|| \\
     &\leq \exp \left(\frac{c^{-}}{d}\log(\eta) \right)  ||y_{\ell_{j}}^{+}(\rho(t_{int}))|| \\
     &\leq \eta^{c^{-}/d} ||y_{\ell_{j}}^{+}(\rho(t_{int}))||.
\end{align*}
As in Case 1, this proves
\begin{equation*}
    x(\kappa_{j}(t)) \in \eta^{c^{-}/d} D^{+}(\kappa_{j}(0)) = \exp \left( \frac{-kc^{-}}{d}(r_{j} - \ell_{j}) \right)D^{+}(\ell_{j}),
\end{equation*}
and by property \eqref{eq:D+-_forward_invariance_property} of $D^{\pm}$ we have
\begin{equation} \label{eq:x_D+_final_estimate}
    x(r_{j}) = x(\kappa_{j}(t_{\max})) \in \exp \left( \frac{-kc^{-}}{d}(r_{j} - \ell_{j}) \right)D^{+}(\ell_{j}).
\end{equation}
Following the same steps as above but for $y_{\ell_{j}}^{-}$ and $D^{-}$ in the case $\dot{\psi}(t)>0$, for some fixed $c^{+}>0$ we get a similar estimate
\begin{equation} \label{eq:x_D-_final_estimate}
    x(r_{j}) \in \exp \left( \frac{-kc^{+}}{d}(r_{j} - \ell_{j}) \right)D^{-}(\ell_{j}).
\end{equation}
Partition $\NN = \cN^{-} \sqcup \cN^{+} \sqcup \cN^{o}$ where $j \in \cN^{-}$ implies $\dot{\psi}(t)<0$ for $t \in J_{j}$, $j \in \cN^{+}$ implies $\dot{\psi}(t)>0$ for $t \in J_{j}$, and $j \in \cN^{o}$ implies $\dot{\psi}(t) \equiv 0$ for $t \in J_{j}$. Since $\sum_{j=1}^{\infty}(r_{j}-\ell_{j}) = \infty$, for some $\cN \in \{\cN^{-}, \cN^{+}, \cN^{o}\}$ we have $\sum_{j \in \cN}(r_{j}-\ell_{j}) = \infty$; since the other cases are treated similarly, without loss of generality we assume $\cN = \cN^{-}$. Then by the property \eqref{eq:D+-_forward_invariance_property} for $D^{+}$ and repeated application of the estimate \eqref{eq:x_D+_final_estimate}, for any $N \in \cN^{-}$ and $t \geq N$ we have
\begin{equation*}
    x(t) \in \exp \left( \frac{-kc^{-}}{d} \sum_{j \leq N, \ j \in \cN^{-}}(r_{j} - \ell_{j}) \right)D^{+}(\ell_{1}).
\end{equation*}
As in Case 1, since we also have
\begin{equation*}
     \exp \left( \frac{-kc^{-}}{d} \sum_{j \leq N, \ j \in \cN^{-}}(r_{j} - \ell_{j}) \right)D^{+}(\ell_{1}) \xrightarrow{N \to \infty} \{0\},
\end{equation*}
we conclude $\lim_{t \to \infty}x(t) = 0$. In the other two cases of $\cN = \cN^{+}$ and $\cN = \cN^{o}$, we use the estimates \eqref{eq:x_D+_final_estimate} and \eqref{eq:x_straight_line_final_estimate} respectively.
\end{proof}

\subsection{Proof of Theorem \ref{thm:iff_conditions_for_stability_assuming_all_edges_exist}}

We begin with a lemma which helps reduce condition \ref{item:1_nonfinal_stability_condition} in Theorem \ref{thm:iff_conditions_for_stability_assuming_all_edges_exist} to a statement in terms of $w([i], [i+1])$ and $w([i+1], [i])$.

\begin{lemma}\label{lemma:main_thm_condition_redundancy}
Suppose that for all $i \in \{1, \ldots, n\}$ and all $p \in \cP_{[i]}^{+}, q \in \cP_{[i+1]}^{-}$, if $x \neq 0$ is in the closed cone between $S_{[i]}$ and $S_{[i+1]}$ in the counterclockwise direction then $\Delta_{p,q}(x) > 0$. Then for all $i \in \{1, \ldots, n\}$ such that $w([i], [i+1]), w([i+1], [i])$ both exist, we have $w([i], [i+1]) \cdot w([i+1], [i]) < 1$.
\end{lemma}

\begin{proof}
Let $\xi \in S_{[i]}$, and consider a solution of the switched system \eqref{intro.1} which satisfies $\dot{x} = A_{p_{[i]}^{+}}x$, $x(0) = \xi$, for all times $t \in [0, \tau)$ where $\tau>0$ is minimal such that $x(\tau_{1}) \in S_{[i+1]}$, and satisfies $\dot{x} = A_{p_{[i+1]}^{-}}x$ for all times $t \in [\tau_{1}, \tau_{2})$ where $\tau_{2}>\tau_{1}$ is minimal such that $x(\tau_{2}) \in S_{[i]}$. Then by Lemma \ref{lemma:unswitched_trajectory_growth_ordering} \ref{item:3ii_unswitched_ordering}, $||x(\tau_{2})||<||x(0)||$, hence $w([i], [i+1]) \cdot w([i+1], [i]) < 1$.
\end{proof}

Armed with the above lemma, Proposition \ref{propn:crossing_between_sector_boundaries}, and Proposition \ref{propn:staying_within_sector_boundaries}, we are ready to prove Theorem \ref{thm:iff_conditions_for_stability_assuming_all_edges_exist}.

\begin{proof}[Proof of Theorem \ref{thm:iff_conditions_for_stability_assuming_all_edges_exist}] ($\implies$) We prove the forward direction by contrapositive; that is, we will assume one of the conditions \ref{item:1_nonfinal_stability_condition} or \ref{item:3_nonfinal_stability_condition} do not hold and show that \eqref{intro.1} is not uniformly attractive, hence not uniformly asymptotically stable by the equivalence from \cite{angeli1999note}. If condition \ref{item:1_nonfinal_stability_condition} does not hold, we are done by Lemma \ref{lemma:collinear_opposite_directions_implies_not_stable}.

Suppose now that condition \ref{item:3_nonfinal_stability_condition} does not hold, so that either $\prod_{i=1}^{n}w([i], [i+1]) \geq 1$ or $\prod_{i=1}^{n}w([i+1], [i]) \geq 1$. A similar construction to the one in the proof of Lemma \ref{lemma:main_thm_condition_redundancy} of a trajectory which switches to subsystem $p_{[i]}^{+}$ (respectively, $p_{[i+1]}^{-}$) upon intersecting $S_{[i]}$ (respectively, $S_{[i+1]}$) yields a sequence of times $\{\tau_{j}\}_{j=1}^{\infty}$ such that for the corresponding solution $x(t)$ to \eqref{intro.1} we have that $\{||x(\tau_{j})||\}_{j=1}^{\infty}$ is a nondecreasing sequence. Hence $x(t)$ does not converge to the origin in forward time for this case either, which completes the proof of the forward direction.

($\impliedby$) Now suppose conditions \ref{item:1_nonfinal_stability_condition} and \ref{item:3_nonfinal_stability_condition} hold. By Lemma \ref{lemma:main_thm_condition_redundancy}, we additionally have for all $i \in \{1, \ldots, n\}$ that 
\begin{equation}\label{eq:w_product_less_than_one_condition}
    w([i],[i+1]) \cdot w([i+1], [1])<1
\end{equation}
Let $x(t)$ solve the switched system \eqref{intro.1} with initial condition $x(0) \neq 0$; by the equivalence between uniform asymptotic stability and uniform attractivity over arbitrary switching signals shown in \cite{angeli1999note}, it suffices to prove $\lim_{t \to \infty}x(t) = 0$. For $i \in \{1, \ldots, n\}$, denote by $C_{[i]}$ the open cone bounded between $S_{[i]}$ and $S_{[i+1]}$ in the counterclockwise direction.

\textit{\underline{Case 1.}} First suppose there exist some $T \geq 0$ and $i \in \{1, \ldots, n\}$ such that for all $t \geq T$, $x(t)$ remains in $C_{[i]} \cup S_{[i+1]} \cup C_{[i+1]}$ for $t \geq T$. Without loss of generality, suppose $x(T) \in \cl{C_{[i]}}$ (it is understood that since $x(t) \in C_{[i]} \cup S_{[i+1]} \cup C_{[i+1]}$, we have $x(T) \notin S_{[i]}$). Let $J_{j} = [\ell_{j}, r_{j}]$ be consecutive time intervals with $T = \ell_{1} < r_{1} = \ell_{2} < r_{2} = \ell_{3} < \cdots$ such that for $t \in J_{j}$ with $j$ odd we have $x(t) \in \cl{C_{[i]}}$ and for $j$ even we have $x(t) \in \cl{C_{[i+1]}}$. By construction, all of $x(r_{1}), x(\ell_{2}), x(r_{2}), x(\ell_{2}), x(r_{2}), \ldots$ lie in $S_{[i+1]}$. Now, we must have either $\sum_{k=1}^{\infty}(r_{2k-1}-\ell_{2k-1}) = \infty$ or $\sum_{k=1}^{\infty}(r_{2k}-\ell_{2k}) = \infty$; without loss of generality, suppose the former is true. We wish to construct a solution $y(t)$ of the switched system \eqref{intro.1} with the following properties:
\begin{itemize}
    \item $y(t) \in \cl{C_{[i]}}$ for $t \geq 0$, and $\lim_{t \to \infty}y(t) = 0$;
    \item there exists a function $\rho: \bigcup_{k \in \ZZ^{+}}J_{2k-1} \to \RR^{\geq 0}$ such that $\lim_{t \to \infty}\rho(t) = \infty$ and for every $t \in J_{j}$ with $j$ odd we have $||x(t)|| \leq ||y(\rho(t))||$.
\end{itemize}
Given such a solution $y(t)$, we would then have
\begin{equation}\label{eq:zero_limit_of_auxiliary_function_y}
    \lim_{k \to \infty} \sup_{t \in J_{2k-1}}||x(t)|| \leq \lim_{k \to \infty} \sup_{t \in J_{2k-1}}||y(\rho(t))|| = \lim_{s \to \infty} ||y(s)|| = 0.
\end{equation}
We now proceed to construct such a solution. For $t \in [0, r_{1}-\ell_{1}]$ define
\begin{equation*}
    y(t) := x(t+\ell_{1}),
\end{equation*}
and for $t \in \left[\sum_{k=1}^{m-1}(r_{2k-1}-\ell_{2k-1}),\sum_{k=1}^{m}(r_{2k-1}-\ell_{2k-1}) \right]$, $m>1$, define
\begin{equation}\label{eq:auxiliary_function_y_definition}
    y(t) := \prod_{k=1}^{m-1}\dfrac{||x(r_{2k-1})||}{||x(r_{2k})||} \cdot x \left(t-\sum_{k=1}^{m-1} (r_{2k-1}-\ell_{2k-1})+\ell_{2m-1} \right).
\end{equation}
A straightforward computation shows $y(t)$ is well-defined and continuous on the endpoints of the intervals $\left[\sum_{k=1}^{m-1}(r_{2k-1}-\ell_{2k-1}),\sum_{k=1}^{m}(r_{2k-1}-\ell_{2k-1}) \right]$. Since $x(t)$ solves \eqref{intro.1} and each subsystem is homogeneous, $y(t)$ solves \eqref{intro.1} also. By construction, $y(t) \in \cl{C_{[i]}}$ for $t \geq 0$, and further we have $\lim_{t \to \infty}y(t) = 0$ by Proposition \ref{propn:staying_within_sector_boundaries}, equation \eqref{eq:w_product_less_than_one_condition}, and assumption \ref{item:1_nonfinal_stability_condition}.

By \eqref{eq:w_product_less_than_one_condition} and Proposition \ref{propn:crossing_between_sector_boundaries} \ref{item:4_crossing_between_sector_boundaries}, for each $m>1$ we have
\begin{equation} \label{eq:auxiliary_trajectory_coefficient_product_bound}
    \prod_{k=1}^{m-1}\dfrac{||x(r_{2k-1})||}{||x(r_{2k})||} < 1.
\end{equation}
Let $\rho^{-1}: \RR^{\geq 0} \to \bigcup_{k \in \ZZ^{+}}J_{2k-1}$ be the function mapping $t \in [0, r_{1}-\ell_{1}]$ to $t + \ell_{1} \in J_{1}$ and mapping
\begin{equation*}
    t \in \left[\sum_{k=1}^{m-1}(r_{2k-1}-\ell_{2k-1}),\sum_{k=1}^{m}(r_{2k-1}-\ell_{2k-1}) \right]
\end{equation*}
to $t-\sum_{k=1}^{m-1} (r_{2k-1}-\ell_{2k-1})+\ell_{2m-1} \in J_{2m-1}$. The same computation as above shows that $\rho^{-1}(t)$ is well-defined, bijective, and strictly increasing. Letting $\rho := (\rho^{-1})^{-1}$, we conclude that $\lim_{t \to \infty}\rho(t) = \infty$ and by \eqref{eq:auxiliary_trajectory_coefficient_product_bound}, for any $t \in \left[\sum_{k=1}^{m-1}(r_{2k-1}-\ell_{2k-1}),\sum_{k=1}^{m}(r_{2k-1}-\ell_{2k-1}) \right]$ and $m > 1$, that we have
\begin{align*}
    ||y(\rho(t))|| &= \prod_{k=1}^{m-1}\dfrac{||x(r_{2k-1})||}{||x(r_{2k})||} \cdot \left|\left| x \left(\rho(t)-\sum_{k=1}^{m-1} (r_{2k-1}-\ell_{2k-1})+\ell_{2m-1} \right) \right|\right| \\
    &> \left|\left| x \left(\rho(t)-\sum_{k=1}^{m-1} (r_{2k-1}-\ell_{2k-1})+\ell_{2m-1} \right) \right|\right| = ||x(t)||.
\end{align*}
Therefore \eqref{eq:zero_limit_of_auxiliary_function_y} holds for $y(t)$ defined as in \eqref{eq:auxiliary_function_y_definition}.

Now, if we also have $\sum_{k=1}^{\infty}(r_{2k}-\ell_{2k}) = \infty$ in addition to $\sum_{k=1}^{\infty}(r_{2k-1}-\ell_{2k-1}) = \infty$, then the same procedure as above yields analogous functions $z(t), \psi(t)$ to $y(t), \rho(t)$ respectively with
\begin{equation*}
    \lim_{k \to \infty, t \in J_{2k}}||x(t)|| \leq \lim_{k \to \infty, t \in J_{2k}}||z(\psi(t))|| = \lim_{s \to \infty}||z(s)|| = 0,
\end{equation*}
hence in this case $\lim_{t \to \infty}||x(t)||=0$. On the other hand, if $\sum_{k=1}^{\infty}(r_{2k}-\ell_{2k}) < \infty$ then for every $\epsilon>0$ there exists a large enough positive integer $m$ such that $\sum_{k=m}^{\infty}(r_{2k}-\ell_{2k}) < \epsilon/2$ and for $t \geq l_{2m}$, the following hold:
\begin{itemize}
    \item If $t \in J_{2k-1}$ for some $k$, then $||x(t)|| \leq \epsilon/2$.
    \item If $t \in J_{2k}$ for some $k$, then $\frac{d}{dt}||x(t)|| \leq 1$ for almost every $t$ (in the sense of Lebesgue).
\end{itemize}
From these conditions it follows that $||x(t)|| \leq \epsilon$ for $t \geq \ell_{2m}$; since $\epsilon>0$ is arbitrary, this proves $\lim_{t \to \infty}||x(t)||=0$.

\textit{\underline{Case 2.}} Now suppose that there exist $[i] \neq [j]$ such that for all $T \geq 0$, $x([T, \infty))$ has nonempty intersection with both $S_{[i]}$ and $S_{[j]}$. Such $[i]$ and $[j]$ exist by finiteness of the set of indices $\{[i]\}_{i=1}^{n}$. Then there exists a strictly increasing sequence $\{t_{k}\}_{k=1}^{\infty}$ of times such that for each $k$, $x(t_{k}) \in S_{[i]}$ and also $x([t_{k}, t_{k+1}])$ has nonempty intersection with $S_{[i']}$ for some $[i'] \neq [i]$. We will first show $\lim_{k \to \infty}||x(t_{k})|| = 0$, and then use this to conclude  $\lim_{t \to \infty}||x(t)|| = 0$.

Fix $k \in \ZZ^{+}$. Let $w_{0}>0$ be the infimum of the products of weights corresponding to all the directed cycles starting and ending at index $[i]$ in the graph $(V, E, w)$; by Lemma \ref{lemma:main_thm_condition_redundancy} and assumptions \ref{item:1_nonfinal_stability_condition} and \ref{item:3_nonfinal_stability_condition}, $w_{0} < 1$ (see Remark \ref{rmk:condition_and_graph_equivalence}). Since $x(t_{k}), x(t_{k+1}) \in S_{[i]}$ and there exists some $[i'] \neq [i]$ for which $x([t_{k}, t_{k+1}])$ has nonempty intersection with $S_{[i']}$, we have some ordered list of indices
\begin{equation*}
    [i] = [i_{1}], [i_{2}], \ldots, [i_{\ell-1}], [i_{\ell}] = [i'], [i_{\ell+1}], \ldots, [i_{m}] = [i]
\end{equation*}
and corresponding times
\begin{equation*}
    t_{k} = s_{1} < s_{2} < \cdots < s_{m} = t_{k+1}
\end{equation*}
such that for all $\ell \in \{1, \ldots, m-1\}$, the following conditions hold:
\begin{itemize}
    \item $x(s_{\ell}) \in S_{[i_{\ell}]}$ and $x(s_{\ell+1}) \in S_{[i_{\ell+1}]}$.
    \item $x([s_{\ell}, s_{\ell+1}])$ has empty intersection with all $S_{[h]}$ such that $[h] \notin \{[i_{\ell}], [i_{\ell+1}]\}$.
\end{itemize}
Then by Proposition \ref{propn:crossing_between_sector_boundaries}, equation \eqref{eq:w_product_less_than_one_condition}, and assumption \ref{item:3_nonfinal_stability_condition} we have
\begin{equation}\label{eq:graph_path_construction_weight_inequality}
    ||x(s_{\ell+1})|| \leq w([i_{\ell}], [i_{\ell+1}]) \cdot ||x(s_{\ell})||.
\end{equation}
for all $\ell \in \{1, \ldots, m-1\}$. Applying \eqref{eq:graph_path_construction_weight_inequality} repeatedly yields
\begin{equation}\label{eq:w_{0}_path_bound_wrt_k}
    ||x(t_{k+1})|| \leq \left( \prod_{\ell=1}^{m} w([i_{\ell}], [i_{\ell+1}]) \right) ||x(t_{k})|| \leq w_{0} \cdot ||x(t_{k})||.
\end{equation}
Since $w_{0} < 1$, applying \eqref{eq:w_{0}_path_bound_wrt_k} repeatedly gives us $\lim_{k \to \infty}||x(t_{k})|| = 0$ as needed.

Now, a quick aside. Let $\xi \in S_{[i]}$ be arbitrary. Consider the forward trajectory $x^{+}(t)$ solving \eqref{intro.1} with initial condition $\xi$ such that $x^{+}(t)$ solves $\dot{x} = A_{p_{[i]}^{+}}x$ until the first time $\tau_{[i+1]}$ with $x^{+}(\tau_{[i+1]}) \in S_{[i+1]}$, at which instant a switch occurs and then $x^{+}(t)$ solves $\dot{x} = A_{p_{[i+1]}^{+}}x$ until the first time $\tau_{[i+2]}$ with $x^{+}(\tau_{[i+2]}) \in S_{[i+2]}$, and so on until the first time $\tau^{+} \geq \tau_{[i+1]}$ with $x^{+}(\tau^{+}) \in S_{[i]}$ again. Similarly, let $x^{-}(t)$ solve \eqref{intro.1} as above except with each index $p_{[i]}^{+}$ replaced by $p_{[i+1]}^{-}$ such that $x^{-}(t)$ solves $\dot{x} = A_{p_{[i+1]}^{-}}x$ between each $S_{[i+1]}$ and $S_{[i]}$, and let $\tau^{-}$ be the first time at which $x^{-}(t)$ makes a full revolution about the origin in the clockwise direction to intersect $S_{[i]}$ again. 

Let $r>0$ be large enough such that the ball $B_{r}(0)$ centered at the origin contains both $x^{+}([0, \tau^{+}])$ and $x^{-}([0, \tau^{-}])$. We claim that any solution $z(t)$ to \eqref{intro.1} with $z(0) = \xi$ must lie in $B_{r}(0)$ for all time $t \geq 0$. In the proof of Proposition \ref{propn:crossing_between_sector_boundaries} we showed (under the same assumptions as in this proof) that if $y(t)$ solves $\dot{y} = A_{p_{[i]}^{+}}y$ with $y(0) = z(0) = \xi$, then if $\tau>0$ is such that $z(t)$ is in the sector formed by $S_{[i]}$ and $S_{[i+1]}$ in the counterclockwise direction for $t \in [0, \tau]$, then whenever $z(t)$ and $y(s)$ lie on the same ray we have $||z(t)|| \leq ||y(s)||$. Likewise, it is similar to show the analogous result for $y(t)$ solving $\dot{y} = A_{p_{[i]}^{-}}y$ and $z(t)$ in the sector formed by $S_{[i]}$ and $S_{[i-1]}$ in the clockwise direction. A straightforward repeated application of these results proves the claim.

Returning to our solution $x(t)$ of \eqref{intro.1}, we have shown above that for any $k \in \ZZ^{+}$ and $t \in [t_{k}, t_{k+1}]$ we have
\begin{equation*}
    ||x(t)|| \leq ||x(t_{k})|| \cdot \frac{||r||}{||\xi||}.
\end{equation*}
Therefore since the constant $||r|| / ||\xi||$ is independent of $k$ and $\lim_{k \to \infty}||x(t_{k})|| = 0$, we conclude $\lim_{t \to \infty}||x(t)|| = 0$ as needed.
\end{proof}

\section{Technical proofs from Section \ref{subsection:stability_without_assumptions}} \label{section:omitted_proofs_2}

To begin, write $\cP = \cP_{\CC} \sqcup \cP_{\RR}$, where $p \in \cP_{\CC}$ if the eigenvalues of $A_{p}$ have nonzero imaginary part and $p \in \cP_{\RR}$ otherwise. Recall that for $\delta>0$, we have defined the matrices
\begin{equation}\label{eq:small_rotation_matrices}
    \cR_{\delta^{+}} = \begin{pmatrix} -1 & -\delta \\ \delta & -1 \end{pmatrix}, \ \cR_{\delta^{-}} = \begin{pmatrix} -1 & \delta \\ -\delta & -1 \end{pmatrix}
\end{equation}
which define additional subsystems to be appended to the original switched system \ref{intro.1}.

\begin{lemma}\label{lemma:delta_matrices_det_complex_case}
Consider $A_{p}$ for $p \in \cP_{\CC}$. There exists $\delta^{\pm}_{p}>0$ such that for all $\delta \in (0, \delta^{\pm}_{p}]$ and $x \in \RR^{2} \setminus \{0\}$,
\begin{equation}
    \det(\cR_{\delta^{\pm}}x \ | \ A_{p}x) \neq 0.
\end{equation}
\end{lemma}

\begin{proof}
For a contradiction, suppose there exist sequences $\{\delta^{\pm}_{p,n}\}_{n=1}^{\infty}$, $\{x_{n}\}_{n=1}^{\infty}$ such that $\delta^{\pm}_{p,n} \to 0$ and $\det(\cR_{\delta^{\pm}_{p,n}}x_{n} \ | \ A_{p}x_{n}) = 0$. Since the determinant is bilinear in the columns of its argument, we may without loss of generality assume $||x_{n}||=1$ for all $n$. Since the unit circle is compact, we may take a convergent subsequence $\{x_{n_{k}}\}_{k=1}^{\infty} \subseteq \{x_{n}\}_{n=1}^{\infty}$ such that $x_{n_{k}} \to x^{*}$ with $||x^{*}||=1$ and $\delta^{\pm}_{p,n_{k}} \to 0$. By continuity of the determinant and the fact that $\cR_{\delta^{\pm}_{p,n_{k}}} \xrightarrow{k \to \infty} I_{2}$ (the $2 \times 2$ identity matrix), we have
\begin{equation*}
    0 = \det(\cR_{\delta^{\pm}_{p,n_{k}}}x_{n_{k}} \ | \ A_{p}x_{n_{k}}) \xrightarrow{k \to \infty} \det(x^{*} \ | \ A_{p}x^{*}).
\end{equation*}
Thus $x^{*}$ and $A_{p}x^{*}$ are linearly dependent, so that $x^{*} \neq 0$ is a real eigenvector of $A_{p}$, contradicting $p \in \cP_{\CC}$.
\end{proof}

\begin{lemma}\label{lemma:delta_matrices_det_real_case}
Consider $A_{p}$ for $p \in \cP_{\RR}$. For all $\epsilon>0$ there exists $\delta_{p}^{\pm}>0$ such that for all $\delta \in (0, \delta_{p}^{\pm}]$, if $x \in \Delta_{\delta^{\pm},p}^{-1}(0)$ then $||x-w||<\epsilon$ for some unit eigenvector $w$ of $A_{p}$.
\end{lemma}

\begin{proof}
For $\delta>0$, recall that we may write
\begin{equation}\label{eq:quadratic_form_representation}
    \Delta_{\delta^{\pm},p}(x) = \det(\cR_{\delta^{\pm}}x \ | \ A_{p}x) = x^{\top}(\cR_{\delta^{\pm}}^{\top} R(-\pi/2) A_{p})x =: q(x, \delta),
\end{equation}
where $R(-\pi/2) = \begin{pmatrix}0 & 1\\ -1 & 0\end{pmatrix}$ and $q: \RR^{3} \to \RR$ is a polynomial in $x$ and $\delta$. Define the closed set $Z = (S^{1} \times [0, \infty)) \cap q^{-1}(0)$. For a contradiction, suppose there exists some $\epsilon_{0}>0$ and a sequence $\{(x_{k}, \delta_{k})\}_{k=1}^{\infty} \subseteq Z$ with $\delta_{k} \xrightarrow{k \to \infty} 0$ and $||x_{k}-w|| \geq \epsilon_{0}$ for all unit eigenvectors $w$ of $A_{p}$ and all $k$. Since $S^{1}$ is compact and $\delta_{k} \to 0$, there exists a subsequence $\{(x_{k_{\ell}}, \delta_{k_{\ell}})\}_{\ell=1}^{\infty} \subseteq Z$ which converges to some point $(\tilde{x}, 0)$ with $\tilde{x}$ not an eigenvector of $A_{p}$. Since $Z$ is closed, $(\tilde{x}, 0) \in Z$. However, this implies $\det(\tilde{x} \ | \ A_{p}x) = 0$ so that $\tilde{x}$ is indeed an eigenvector of $A_{p}$, a contradiction.
\end{proof}

For $s \in \{+,-\}$ and $p \in \cP_{\RR}$ as in Lemma \ref{lemma:delta_matrices_det_real_case}, since $A_{p}$ has a finite number of unit eigenvectors by Remark \ref{rmk:can_exclude_minus_identity} there exists a single number
\begin{equation*}
    \delta_{p}^{s} = \min_{\text{unit eigenvectors } ||w||=1} \delta_{p,w}^{s} >0
\end{equation*}
for which the conclusion of Lemma \ref{lemma:delta_matrices_det_real_case} holds independent of the particular unit eigenvector $w$.

Now write $\tilde{\cP}(\delta) = \cP \cup \{\delta^{+}, \delta^{-}\}$, and consider the $\delta$-dependent switched system
\begin{equation}\label{eq:auxiliary_system}
    \dot{x}(t) = A_{\sigma(t)} x(t), \hspace{20px} \sigma : \RR^{\geq0} \to \tilde{\cP}(\delta).
\end{equation}
We wish to go through the same setup as in the beginning of Section \ref{subsection:stability_under_assumptions} to obtain a collection of half-open rays $\tilde{\cS}(\delta) = \{S_{[1]}(\delta), S_{[2]}(\delta), \ldots, S_{[n(\delta)]}(\delta)\}$, this time using the indexing set $\tilde{\cP}(\delta)$ instead of $\cP$. Since the eigenvalues of $\cR_{\delta^{\pm}}$ have nonzero imaginary part, the sets $\tilde{\cS}_{1}(\delta)$ and $\tilde{\cS}_{3}(\delta)$ are the same as the corresponding sets $\cS_{1}$ and $\cS_{3}$ from the beginning of Section \ref{subsection:stability_under_assumptions}. For the set $\tilde{\cS}_{2}(\delta)$ corresponding to $\cS_{2}$, we have
\begin{equation}
    \tilde{\cS}_{2}(\delta) = \cS_{2} \cup \{\Delta^{-1}_{p, \delta^{+}}(0) : p \in \cP \} \cup \{\Delta^{-1}_{p, \delta^{-}}(0) : p \in \cP \}.
\end{equation}

For $i \in \{1, \ldots, n\}$ and $p \in \cP_{[i]}^{+}$, $q \in \cP_{[i+1]}^{-}$, if $x \neq 0$ is in the closed cone between $S_{[i]}$ and $S_{[i+1]}$ in the counterclockwise direction then we have
\begin{align}
    \label{eq:positive_det_delta_requirement_1}\det(A_{p}x \ | \ \cR_{\delta^{-}}x) &= \det(x \ | \ A_{p}x) - \delta \cdot \det \left( \begin{pmatrix} 0 & 1 \\ -1 & 0 \end{pmatrix}x \ \Bigg| \ A_{p}x \right) > 0, \\
    \label{eq:positive_det_delta_requirement_2}\det(\cR_{\delta^{+}}x \ | \ A_{q}x) &= -\det(x \ | \ A_{q}x) - \delta \cdot \det \left( \begin{pmatrix} 0 & -1 \\ 1 & 0 \end{pmatrix}x \ \Bigg| \ A_{q}x \right) > 0
\end{align}
for all small enough $\delta>0$, since $\det(x \ | \ A_{p}x) > 0$ and $\det(x \ | \ A_{q}x) < 0$ for all such $x$ by Lemma \ref{lemma:det_x_Apx_positive}. We may assume the above inequalities \eqref{eq:positive_det_delta_requirement_1} and \eqref{eq:positive_det_delta_requirement_2} hold for all $\delta \in (0, \tilde{\delta}_{1}]$.

By Lemma \ref{lemma:delta_matrices_det_complex_case}, there exists $\tilde{\delta}_{2} \in (0, \tilde{\delta}_{1}]$ such that for all $\delta \in (0, \tilde{\delta}_{2}]$ and for $p \in \cP$ with $A_{p}$ having eigenvalues with nonzero imaginary part, 
\begin{equation}
    \Delta^{-1}_{p, \delta^{\pm}}(0) = \{0\}.
\end{equation}
Now let $\epsilon>0$ be less than half the distance between each unit vector $w \in S_{[i]}$ for each $[i]$. By Lemma \ref{lemma:delta_matrices_det_real_case} there exist subcollections $\cP_{\RR}^{+}, \cP_{\RR}^{-} \subseteq \cP_{\RR}$ and some 
\begin{equation}\label{eq:delta_tilde_introduction}
    \tilde{\delta}_{3} \in (0, \min\{ \min_{p \in \cP_{\RR}^{+}}\delta_{p}^{+}, \min_{p \in \cP_{\RR}^{-}}\delta_{p}^{-}, \tilde{\delta}_{2}\})
\end{equation}
such that if $\delta \in (0, \tilde{\delta}_{3}]$ and $x \in \Delta_{\delta^{\pm},p}^{-1}(0)$ then $||x-w||<\epsilon$ for some unit eigenvector $w$ of $A_{p}$.

Noting that $w$ is a unit eigenvector of $A_{p}$ if and only if $S_{[i]} = \{aw : a>0\}$ for some $S_{[i]} \subseteq \cS_{3}$, for each $s \in \{+,-\}$ we will denote by $h_{\delta^{s}}(S_{[i]})$ the set of all rays $\{ax: x \in \Delta_{\delta^{s},p}^{-1}(0), ||x||=1\}$ such that if we write $S_{[i]} = \{aw : a>0, ||w||=1\}$ then $||x-w||<\epsilon$. By Lemma \ref{lemma:delta_matrices_det_real_case}, for $\delta \in (0, \tilde{\delta}_{3}]$ we have
\begin{equation*}
    \bigcup_{S_{[i]} \subseteq \cS_{3},s \in \{+,-\}} h_{\delta^{s}}(S_{[i]}) = \tilde{\cS}_{2}(\delta).
\end{equation*}
By the work in Section \ref{section:preliminaries}, the zero set of the quadratic form
\begin{equation}
    \Delta_{\delta^{\pm},p}(x) = x^{\top}(\cR_{\delta^{\pm}}^{\top}R(-\pi/2)A_{p})x
\end{equation}
for fixed $\delta$ is either a single one-dimensional subspace or two distinct one-dimensional subspaces of $\RR^{2}$. Since the eigenvalues of the symmetric part of $\cR_{\delta^{\pm}}^{\top}R(-\pi/2)A_{p}$ vary continuously with respect to $\delta$, we may take $\tilde{\delta} \in (0, \tilde{\delta}_{3}]$ small enough that for all $\delta \in (0, \tilde{\delta}]$ and all $p \in \cP_{\RR}$ the symmetric part of $\cR_{\delta^{\pm}}^{\top}R(-\pi/2)A_{p}$ is always positive semi-definite, always negative semi-definite, or always indefinite (compare with the cases discussed at the beginning of Section \ref{section:preliminaries}); that is, the number of one-dimensional subspaces is constant.

Then the number of connected components of $\tilde{S}_{2}(\delta)$, and thus of $\tilde{S}(\delta)$, does not change for $\delta \in (0, \tilde{\delta}]$; in this case we say $\tilde{\cS}(\delta)$ has $\tilde{n} := n(\delta)$ components. For later use, note that for any $\delta$ as in \eqref{eq:delta_tilde_introduction} we cannot have for $S_{[j]}(\delta), S_{[j+1]}(\delta) \in \tilde{\cS}(\delta)$ and any $i \in \{1, \ldots, n\}$ that $S_{[j]}(\delta) \in h_{\delta^{\pm}}(S_{[i+1]})$ and $S_{[j+1]}(\delta) \in h_{\delta^{\pm}}(S_{[i]})$ by our choice of $\epsilon$ above.

For $i \in \{1, \ldots, \tilde{n}\}$, denote by $w_{\delta}([i], [i+1])$ and $w_{\delta}([i+1], [i])$ the ``weight'' values analogous to \eqref{eq:weight_defn_+} and \eqref{eq:weight_defn_-} respectively but for the switched system \eqref{eq:auxiliary_system} instead of system \eqref{intro.1}. Note that if we take $\delta=0$ the matrices $\cR_{\delta^{\pm}}$ in \eqref{eq:small_rotation_matrices} are both $-I_{2}$, hence by Remark \ref{rmk:can_exclude_minus_identity} they do not affect the uniform asymptotic stability properties of the switched system and may effectively be removed from the system. As such, with $\delta = 0$ we would have $w_{\delta}([i], [i+1]) = w([i], [i+1])$ and $ w_{\delta}([i+1], [i]) = w([i+1], [i])$ for $i \in \{1, \ldots, n\}$ and $\tilde{n} = n(\delta) = n$.

\begin{lemma}\label{lemma:Si_Sj_delta_correspondence_using_unit_vectors}
Let $\delta \in (0, \tilde{\delta}_{3}]$ and suppose $u(\delta), \tilde{u}(\delta), u, v(\delta), \tilde{v}(\delta), v$ are unit vectors in $\RR^{2}$ such that $u(\delta) \to u$, $\tilde{u}(\delta) \to u$, $v(\delta) \to v$, $\tilde{v}(\delta) \to v$ as $\delta \to 0^{+}$.
\begin{enumerate}
    \item \label{item:1_lemma:Si_Sj_delta_correspondence} If for some $i \in \{1, \ldots, n\}$ and $j \in \{1, \ldots, \tilde{n}\}$ we have
    \begin{align*}
        &S_{[i]} = \{au : a >0\}, S_{[i+1]} = \{av : a >0\}, \\
        &S_{[j]}(\delta) = \{a u(\delta) : a>0\}, S_{[j+1]}(\delta) = \{a v(\delta) : a>0\},
    \end{align*}
    then $w_{\delta}([j], [j+1]) \xrightarrow{\delta \to 0^{+}} w([i], [i+1])$ and $w_{\delta}([j+1], [j]) \xrightarrow{\delta \to 0^{+}} w([i+1], [i])$.

    \item \label{item:2_lemma:Si_Sj_delta_correspondence} If for some $j \in \{1, \ldots, \tilde{n}\}$ we have
    \begin{equation*}
        S_{[j]}(\delta) = \{a u(\delta) : a>0\}, S_{[j+1]}(\delta) = \{a \tilde{u}(\delta) : a>0\}
    \end{equation*}
    or
    \begin{equation*}
        S_{[j]}(\delta) = \{a u(\delta) : a>0\}, S_{[j+1]}(\delta) = \{au : a>0\},
    \end{equation*}
    with $S_{[i]} = \{au:a>0\}$, then
    \begin{equation}\label{eq:each_w_delta_j_to_1}
        w_{\delta}([j], [j+1]), w_{\delta}([j+1], [j]) \xrightarrow{\delta \to 0^{+}} 1.
    \end{equation}
    The analogous result holds when swapping the definitions of $S_{[j]}$ and $S_{[j+1]}$, and also for $v(\delta), \tilde{v}(\delta), v$ instead of $u(\delta), \tilde{u}(\delta), u$.
\end{enumerate}
\end{lemma}

\begin{proof}
For \ref{item:1_lemma:Si_Sj_delta_correspondence}, we will assume the conditions of the lemma hold and show
\begin{equation}
    w_{\delta}([j], [j+1]) \xrightarrow{\delta \to 0^{+}} w([i], [i+1]).
\end{equation}
Denote by $\tilde{\cP}_{[j]}^{+}(\delta)$ the set of all indices $p \in \cP \cup \{\delta^{+}, \delta^{-}\}$ such that the forward trajectory of a solution of \eqref{stability.2} with initial condition in $S_{[j]}(\delta)$ intersects $S_{[j+1]}(\delta)$ in finite time before it (possibly) intersects $S_{[j-1]}(\delta)$. Since $u(\delta) \to u$, $v(\delta) \to v$, and the set $\cP \cup \{\delta^{+}, \delta^{-}\}$ is finite, for all small enough $\delta>0$ we must have $\tilde{\cP}_{[j]}^{+}(\delta) = \cP_{[i]}^{+} \cup \{\delta^{+}\}$.

For $p \in \cP_{[i]}^{+} \cup \{\delta^{+}\}$, define $F_{p}(t,x,y) = \det(e^{A_{p}t}x \ | \ y)$. Letting $\tau_{p}$ be the unique minimal positive time such that a solution of \eqref{stability.2} with this choice of $p$ and with initial condition $u$ intersects $S_{[i+1]}$ at time $\tau_{p}$, we have
\begin{equation}
    F_{p}(t,u,v) = \det(e^{A_{p}\tau_{p}u} \ | \ v) = 0.
\end{equation}
We compute
\begin{align*}
    \frac{\partial}{\partial t} \bigg|_{(t,x,y) = (\tau_{p},u,v)} F_{p}(t,x,y) &= \det \left( \frac{\partial}{\partial t} \bigg|_{(t,x) = (\tau_{p},u)} e^{A_{p}t}x \ \Bigg| \ v \right) \\
    &= \det(A_{p}e^{A_{p}\tau_{p}}u \ | \ v) = \det(A_{p}v \ | \ v) \neq 0
\end{align*}
by Lemma \ref{lemma:det_x_Apx_positive}. Since $F_{p}$ is smooth in all its components and 
\begin{equation*}
    \frac{\partial}{\partial t} \bigg|_{(t,x,y) = (\tau_{p},u,v)} F_{p}(t,x,y) \neq 0,
\end{equation*}
by the Implicit Function Theorem there exist open neighbourhoods $U_{p}$ of $(u,v)$ and $V_{p}$ of $\tau_{p}$ and a smooth function $g_{p}:U_{p} \to V_{p}$ such that $F_{p}(g_{p}(x,y),x,y) = 0$ for all $(x,y) \in U_{p}$, and $(g_{p}(x,y),x,y)$ is the unique solution of this equation for a given point $(x,y) \in U_{p}$. If needed, shrink $U_{p}, V_{p}$ so that $0 \notin V_{p}$.

Since $(u(\delta), v(\delta)) \xrightarrow{\delta \to 0^{+}} (u,v)$ and $(u,v) \in U_{p}$, for all small enough $\delta>0$ we have $(u(\delta), v(\delta)) \in U_{p}$ as well. For such points $(u(\delta), v(\delta))$, there exists a unique time $\tau(p,\delta) = g_{p}(u(\delta), v(\delta)) \in V_{p}$ such that $F_{p}(\tau(p, \delta), u(\delta), v(\delta)) = 0$, so that the forward trajectory of \eqref{stability.2} with initial condition $u(\delta)$ intersects $S_{[j+1]}(\delta)$ at time $\tau(p,\delta)>0$.

We claim that for all small enough $\delta>0$, $\tau(p, \delta)$ is the minimal positive time such that $e^{A_{p}\tau(p, \delta)}u(\delta) \in S_{[j+1]}(\delta)$. For a contradiction, suppose there exists a sequence $\{\delta_{k}\}_{k=1}^{\infty}$ with $\delta_{k} \xrightarrow{k \to \infty} 0$ such that $\tau(p, \delta_{k})>0$ is not minimal in this sense for all $k$. By homogeneity of the system \eqref{stability.2} and since the trajectory starting at $u(\delta_{k})$ must pass $S_{[j+1]}(\delta_{k})$ and circle about the origin at least once before time $\tau(p, \delta_{k})$ by Lemma \ref{lemma:det_x_Apx_positive}, for all large enough $k$ there must exist some constant $\tilde{\tau}>0$ such that
\begin{equation*}
    \tau(p, \delta_{k}) > \tau_{p} + \tilde{\tau}.
\end{equation*}
But then
\begin{equation*}
    \tau_{p} = g_{p}(u,v) = \lim_{k \to \infty} g_{p}(u(\delta_{k}), v(\delta_{k})) = \lim_{k \to \infty} \tau(p,\delta_{k}) > \tau_{p} + \tilde{\tau},
\end{equation*}
a contradiction. Thus $\tau(p, \delta)>0$ is minimal.

Putting all this together, we have
\begin{equation}\label{eq:limit_w_delta_j_to_w_i_computation}
\begin{split}
    \lim_{\delta \to 0^{+}} w_{\delta}([j], [j+1]) &= \lim_{\delta \to 0^{+}} \max_{ p \in \cP_{[i]}^{+} \cup \{\delta^{+}\} } ||e^{A_{p}\tau(p, \delta)} u(\delta)|| \\
    &= \max_{ p \in \cP_{[i]}^{+} \cup \{\delta^{+}\} } ||e^{A_{p}\tau_{p}}u|| = w([i], [i+1]).
\end{split}
\end{equation}

We proceed to proving part \ref{item:2_lemma:Si_Sj_delta_correspondence} of this lemma. Suppose
\begin{equation}\label{eq:S_j,S_j+1_both_vary_with_delta_to_u}
    S_{[j]}(\delta) = \{a u(\delta) : a>0\}, S_{[j+1]}(\delta) =\{a \tilde{u}(\delta) : a>0\};
\end{equation}
the other possibility in the lemma statement is handled almost identically. Using a similar application of the Implicit Function Theorem to $F_{p}(t,x,y)$ as in the proof of part \ref{item:1_lemma:Si_Sj_delta_correspondence} except with the point $(0, u, u)$ instead of $(\tau_{p}, u, v)$, we have that the minimal positive time $\tau(p, \delta)$ at which the forward trajectory of \eqref{stability.2} starting at $u(\delta)$ intersects $S_{[j+1]}(\delta)$ satisfies $\tau(p, \delta) \xrightarrow{\delta \to 0^{+}} 0$. A similar computation to \eqref{eq:limit_w_delta_j_to_w_i_computation} shows in this case that
\begin{equation*}
    \lim_{\delta \to 0^{+}} w_{\delta}([j],[j+1]) = ||u|| = 1;
\end{equation*}
proving $\lim_{\delta \to 0^{+}} w_{\delta}([j+1],[j]) = 1$ is almost identical, so \eqref{eq:each_w_delta_j_to_1} holds.
\end{proof}

We are now ready to proceed to the proof of Theorem \ref{thm:iff_conditions_for_stability}.

\begin{proof}[Proof of Theorem \ref{thm:iff_conditions_for_stability}]

We will show that there exists $\delta'>0$ such that the following hold:
\begin{enumerate}[label=(\Roman*)]
    \item \label{item:1_final_theorem_strategy} The uniform asymptotic stability of system \eqref{intro.1} is equivalent to the uniform asymptotic stability of system \eqref{eq:auxiliary_system} for all $\delta \in (0, \delta']$.
    
    \item \label{item:2_final_theorem_strategy} The conditions \ref{item:1_final_stability_condition}, \ref{item:3_final_stability_condition} of this Theorem \ref{thm:iff_conditions_for_stability} for the system \eqref{intro.1} are equivalent to the conditions \ref{item:1_nonfinal_stability_condition}, \ref{item:3_nonfinal_stability_condition} of Theorem \ref{thm:iff_conditions_for_stability_assuming_all_edges_exist} for the system \eqref{eq:auxiliary_system} for all $\delta \in (0, \delta']$.
\end{enumerate}
Applying Theorem \ref{thm:iff_conditions_for_stability_assuming_all_edges_exist} to system \eqref{eq:auxiliary_system} with any choice of $\delta \in (0, \delta']$ will complete the proof.

The backward direction of equivalence \ref{item:1_final_theorem_strategy} is trivial since $\cP \subseteq \tilde{\cP}(\delta)$ for any $\delta>0$. For the forward direction, we utilize a converse Lyapunov theorem for uniformly asymptotically stable linear switched systems which is a direct consequence of the analogous result for stable linear differential inclusions from (\cite{molchanov1989criteria}, Theorem 6). By this theorem, the switched system \eqref{intro.1} is uniformly asymptotically stable if and only if there exist positive integers $N, p$ and vectors $l_{i} \in \RR^{2}$, $i \in \{1, \ldots, N\}$ satisfying certain conditions, and a constant $\eta>0$ such that
\begin{equation}\label{eq:lyapunov_function_condition}
    \sum_{i=1}^{N} 2p(l_{i} \cdot x)^{2p-1}(l_{i} \cdot A_{q}x) < - \eta ||x||^{2p}
\end{equation}
for each $i \in \{1, \ldots, N\}$, $q \in \cP$, and $x \in \RR^{2}$. To show system \eqref{eq:auxiliary_system} is uniformly asymptotically stable for a given $\delta>0$ it suffices to prove there exists some positive constant $\tilde{\eta} \leq \eta$ such that
\begin{equation}
    \sum_{i=1}^{N} 2p(l_{i} \cdot x)^{2p-1}(l_{i} \cdot \cR_{\delta^{\pm}}x) < - \tilde{\eta} ||x||^{2p}.
\end{equation}
Writing $x = (x_{1}, x_{2})^{\top}$, we have
\begin{equation*}
    \cR_{\delta^{\pm}}x = - \begin{pmatrix} x_{1} \\ x_{2} \end{pmatrix} + \delta \begin{pmatrix} 0 & \mp1 \\ \pm1 & 0 \end{pmatrix} \begin{pmatrix} x_{1} \\ x_{2} \end{pmatrix} =: -x + \delta R(\pm\pi/2) x
\end{equation*}
where $||R(\pm\pi/2)x|| = ||x||$. Set $\delta_{1} = 1/2$, and let $\tilde{\eta} = \min \{\eta, \min_{1 \leq i \leq N} p ||l_{i}||^{2p} \} > 0$. Then for any $\delta \in (0, \delta_{1}]$, the Cauchy-Schwarz inequality gives
\begin{align*}
    \sum_{i=1}^{N} 2p(l_{i} \cdot x)^{2p-1}(l_{i} \cdot \cR_{\delta^{\pm}}x) &= \sum_{i=1}^{N} \left(-2p(l_{i} \cdot x)^{2p-1}(l_{i} \cdot x) + 2p(l_{i} \cdot x)^{2p-1}(l_{i} \cdot \delta R(\pm\pi/2)x) \right) \\
    & \leq (\delta-1) \sum_{i=1}^{N} 2p ||l_{i}||^{2p} \cdot ||x||^{2p} \leq \tilde{\eta}||x||^{2p}.
\end{align*}
Therefore for any $\delta \in (0, \delta_{1}]$, system \eqref{eq:auxiliary_system} is uniformly asymptotically stable.

We now proceed to show equivalence \ref{item:2_final_theorem_strategy}. First suppose conditions \ref{thm:iff_conditions_for_stability}-\ref{item:1_final_stability_condition}, \ref{thm:iff_conditions_for_stability}-\ref{item:3_final_stability_condition} of this theorem hold for the system \eqref{intro.1}. Choosing $\delta_{2} = \min\{\delta_{1}, \tilde{\delta}_{3}\}$ where $\tilde{\delta}_{3}$ is as in Lemma \ref{lemma:Si_Sj_delta_correspondence_using_unit_vectors}, for system \eqref{eq:auxiliary_system} with $\delta \in (0, \delta_{2}]$ we have that condition \ref{thm:iff_conditions_for_stability_assuming_all_edges_exist}-\ref{item:1_nonfinal_stability_condition} holds due to the estimates \eqref{eq:positive_det_delta_requirement_1} and \eqref{eq:positive_det_delta_requirement_2}. Since condition \ref{thm:iff_conditions_for_stability}-\ref{item:3_final_stability_condition} holds for system \eqref{intro.1}, there exists some $\rho \in (0,1)$ for which
\begin{equation}\label{eq:cycle_weight_product_rho_estimate_i}
    \prod_{i=1}^{n}(w([i], [i+1])+ \rho) < 1/ (1+\rho)^{\tilde{n} - n}, \ \ \prod_{i=1}^{n}(w([i+1], [i])+ \rho) < 1/ (1+\rho)^{\tilde{n} - n}.
\end{equation}
Let $j \in \{1, \ldots, \tilde{n}\}$ and consider $S_{[j]}(\delta), S_{[j+1]}(\delta)$ with $\delta \in (0, \delta_{2}]$. By definition of the rays $S_{[j]}(\delta)$, there must exist some $i \in \{1, \ldots, n\}$ such that one of the following possibilities holds:
\begin{enumerate}[label=(\roman*)]
    \item \label{item:1_Sj_possibilities_list} $S_{[j]}(\delta) = S_{[i]}, S_{[j+1]}(\delta) = S_{[i+1]}$;
    \item \label{item:2_Sj_possibilities_list} $S_{[j]}(\delta) = S_{[i]}, S_{[j+1]}(\delta) = h_{\delta^{\pm}}(S_{[i]})$ or vice versa;
    \item \label{item:3_Sj_possibilities_list} $S_{[j]}(\delta) = S_{[i+1]}, S_{[j+1]}(\delta) = h_{\delta^{\pm}}(S_{[i+1]})$ or vice versa;
    \item \label{item:4_Sj_possibilities_list} $S_{[j]}(\delta) = S_{[i]}, S_{[j+1]}(\delta) = h_{\delta^{\pm}}(S_{[i+1]})$ or $S_{[j]}(\delta) = h_{\delta^{\pm}}(S_{[i]}), S_{[j+1]}(\delta) = S_{[i+1]}$;
    \item \label{item:5_Sj_possibilities_list} $S_{[j]}(\delta) = h_{\delta^{\pm}}(S_{[i]}), S_{[j+1]}(\delta) = h_{\delta^{\pm}}(S_{[i+1]})$ or $S_{[j]}(\delta) = h_{\delta^{\pm}}(S_{[i]}), S_{[j+1]}(\delta) = h_{\delta^{\mp}}(S_{[i+1]})$
\end{enumerate}
where $h_{\delta^{\pm}}$ is defined above. Then if one of \ref{item:1_Sj_possibilities_list}, \ref{item:4_Sj_possibilities_list}, \ref{item:5_Sj_possibilities_list} holds, by Lemma \ref{lemma:Si_Sj_delta_correspondence_using_unit_vectors} \ref{item:1_lemma:Si_Sj_delta_correspondence} there exists some $\delta^{j}>0$ such that for all $\delta \in (0, \delta^{j}]$,
\begin{align*}
    &w_{\delta}([j], [j+1]) < w([i], [i+1]) + \rho, \\
    &w_{\delta}([j+1], [j]) < w([i+1], [i]) + \rho;
\end{align*}
similarly, if one of \ref{item:2_Sj_possibilities_list}, \ref{item:3_Sj_possibilities_list} holds then by Lemma \ref{lemma:Si_Sj_delta_correspondence_using_unit_vectors} \ref{item:2_lemma:Si_Sj_delta_correspondence} there exists some $\delta^{j}>0$ such that
\begin{equation*}
    w_{\delta}([j], [j+1]), w_{\delta}([j+1], [j]) < 1 + \rho.
\end{equation*}
Thus by \eqref{eq:cycle_weight_product_rho_estimate_i},
\begin{align*}
    &\prod_{i=1}^{\tilde{n}} w([j], [j+1]) \leq \prod_{i=1}^{n}(w([i], [i+1]) + \rho) \cdot \prod_{i=1}^{\tilde{n}-n}(1+\rho) < 1, \\
    &\prod_{i=1}^{\tilde{n}} w([j+1], [j]) \leq \prod_{i=1}^{n}(w([i+1], [i]) + \rho) \cdot \prod_{i=1}^{\tilde{n}-n}(1+\rho) < 1.
\end{align*}
Letting $\delta_{3} = \min\{\delta_{2}, \min_{1 \leq j \leq \tilde{n}}\delta^{j}\}$, we have shown that condition \ref{thm:iff_conditions_for_stability_assuming_all_edges_exist}-\ref{item:3_nonfinal_stability_condition} holds for system \eqref{eq:auxiliary_system} with $\delta \in (0, \delta_{3}]$.

Now for the other direction of equivalence \ref{item:2_final_theorem_strategy}, suppose conditions \ref{item:1_nonfinal_stability_condition}, \ref{item:3_nonfinal_stability_condition} of Theorem \ref{thm:iff_conditions_for_stability_assuming_all_edges_exist} hold for system \eqref{eq:auxiliary_system} for all $\delta \in (0, \delta']$. Then there exists some $\gamma \in (0,1)$ such that
\begin{equation}\label{eq:cycle_weight_product_rho_estimate_j}
    \prod_{j=1}^{\tilde{n}} (w([j], [j+1])+\gamma) < 1, \ \ \prod_{j=1}^{\tilde{n}} (w([j+1], [j])+\gamma) < 1
\end{equation}
for all $\delta \in (0, \delta']$. Since $\cP \subseteq \tilde{\cP}(\delta)$ for any $\delta>0$, condition \ref{thm:iff_conditions_for_stability}-\ref{item:1_final_stability_condition} holds for system \eqref{intro.1}. Let $i \in \{1, \ldots, n\}$ and choose $j(i) \in \{1, \ldots, \tilde{n}\}$ and $k(i) \in \ZZ^{+}$ minimal such that $S_{[i]} = S_{[j(i)]}(\delta)$ and $S_{[i+1]} = S_{[j(i)+k(i)]}(\delta)$.

If $k(i)=1$, evidently $w([i], [i+1]) = w_{\delta}([j(i)], [j(i)+1])$ and $w([i+1], [i]) = w_{\delta}([j(i)+1], [j(i)])$, so suppose $k(i)>1$. Let $1 \leq \tilde{k}(i) < k(i)$ be maximal such that either $S_{[j(i)+\tilde{k}(i)]}(\delta) = S_{[i]}$ or $S_{[j(i)+\tilde{k}(i)]}(\delta) = h_{\delta^{\pm}}(S_{[i]})$; then by definition of the rays $S_{[j(i)]}(\delta)$, the following must hold:
\begin{itemize}
    \item For $ 1 \leq \ell \leq \tilde{k}(i)$, either $S_{[j(i)+\ell]}(\delta) = S_{[i]}$ or $S_{[j(i)+\ell]}(\delta) = h_{\delta^{\pm}}(S_{[i]})$.
    \item For $\tilde{k}(i) + 1 \leq \ell \leq k(i)$, either $S_{[j(i)+\ell]}(\delta) = S_{[i+1]}$ or $S_{[j(i)+\ell]}(\delta) = h_{\delta^{\pm}}(S_{[i+1]})$.
\end{itemize}
By Lemma \ref{lemma:Si_Sj_delta_correspondence_using_unit_vectors}, we may choose $\delta^{i} \in (0, \delta']$ such that for all $\delta \in (0, \delta^{i}]$ and all $\ell \in \{1, \ldots, k(i)\} \setminus \{\tilde{k}(i)+1\}$,
\begin{align*}
    &w([i], [i+1]) < w_{\delta}([j(i)+\tilde{k}(i)], [j(i)+\tilde{k}(i)+1]) + \gamma, \\
    &w([i+1], [i]) < w_{\delta}([j(i)+\tilde{k}(i)+1], [j(i)+\tilde{k}(i)]) + \gamma, \\
    &w_{\delta}([j(i)+\ell-1], [j(i)+\ell]) > 1 - \gamma, \\
    &w_{\delta}([j(i)+\ell], [j(i)+\ell-1]) > 1 - \gamma.
\end{align*}
Set $\delta_{4} = \min \{\delta', \min_{1 \leq i \leq n}\delta^{i}\}$, and let $\{1, \ldots, n\} = I \sqcup I'$ where $i \in I'$ if $k(i) = 1$ and $i \in I$ otherwise. By \eqref{eq:cycle_weight_product_rho_estimate_j},
\begin{align*}
    \prod_{i=1}^{n} w([i], [i+1]) &< \prod_{i=1}^{n}(w_{\delta}([j(i)+\tilde{k}(i)],[j(i)+\tilde{k}(i)+1]) + \gamma) \\
    &= \prod_{i \in I}(w_{\delta}([j(i)+\tilde{k}(i)],[j(i)+\tilde{k}(i)+1]) + \gamma) \\ 
    & \hspace{20px} \cdot \prod_{i \in I'}(w_{\delta}([j(i)+\tilde{k}(i)],[j(i)+\tilde{k}(i)+1]) + \gamma) \\
    &< \Bigg( \prod_{i \in I}(w_{\delta}([j(i)+\tilde{k}(i)],[j(i)+\tilde{k}(i)+1]) + \gamma ) \\
    & \hspace{20px} \cdot \prod_{\ell \in \{1, \ldots, k(i)\} \setminus \{\tilde{k}(i)+1\}} (w_{\delta}([j(i)+\tilde{k}(i)],[j(i)+\tilde{k}(i)+1]) + \gamma) \Bigg) \\ 
    & \hspace{20px} \cdot \left( \prod_{i \in I'}(w_{\delta}([j(i)+\tilde{k}(i)],[j(i)+\tilde{k}(i)+1]) + \gamma) \right) \\
    &= \prod_{j=1}^{\tilde{n}} (w_{\delta}([j], [j+1]) + \gamma) < 1.
\end{align*}
Thus condition \ref{thm:iff_conditions_for_stability}-\ref{item:3_final_stability_condition} holds for system \eqref{intro.1} with $\delta \in (0, \delta_{4}]$ as well, completing the proof.

\end{proof}

\section*{Funding}
I. O. Shevchenko was supported by the Ontario Graduate Scholarship. X. Liu was supported by the Natural Sciences and Engineering Research Council of Canada.

\section*{Acknowledgment}
The authors are grateful (in advance) to the anonymous reviewers for their valuable feedback.


\bibliographystyle{apalike}
\bibliography{ref}

\end{document}